\documentclass[12pt]{amsart}

\usepackage[a4paper]{geometry}
\usepackage{amsmath, amsthm}
\usepackage{amssymb, amsfonts}
\usepackage{hyperref}
\usepackage{tikz}
\usepackage{bm}

\usepackage{microtype}
\usepackage{enumerate}

\usepackage{stmaryrd}
\DeclareSymbolFont{bbold}{U}{bbold}{m}{n}
\DeclareSymbolFontAlphabet{\mathbbm}{bbold}

\title{Category forcing and generic absoluteness}

\theoremstyle{plain}
	\newtheorem{theorem}{Theorem}[section]
	\newtheorem{proposition}[theorem]{Proposition}
	\newtheorem{lemma}[theorem]{Lemma}
	\newtheorem{corollary}[theorem]{Corollary}
	\newtheorem{fact}[theorem]{Fact}

	\newtheorem{claim}{Claim}
	
\theoremstyle{definition}
	\newtheorem{definition}[theorem]{Definition}
	\newtheorem{notation}[theorem]{Notation}
	
	\newtheorem{notation*}{Notation}

\theoremstyle{remark}
	\newtheorem{remark}[theorem]{Remark}

\newcommand{\id}{\ensuremath{\mathrm{Id}}}
\newcommand{\Ord}{\ensuremath{\mathrm{Ord}}}

\newcommand{\ZFC}{\ensuremath{\mathsf{ZFC}}}
\newcommand{\ZF}{\ensuremath{\mathsf{ZF}}}
\newcommand{\MK}{\ensuremath{\mathsf{MK}}}

\DeclareMathOperator{\dom}{dom}

\DeclareMathOperator{\cof}{cof}

\DeclareMathOperator{\coker}{coker}
\DeclareMathOperator{\supp}{supp}

\DeclareMathOperator{\St}{St}

\DeclareMathOperator{\Coll}{Coll}

\DeclareMathOperator{\RO}{\mathsf{RO}}

\newcommand{\NS}{\ensuremath{\mathbf{NS}}} %non stationary ideal
\newcommand{\bool}[1]{\mathsf{#1}}

\newcommand{\FFF}{\mathcal{F}}

\newcommand{\pow}[1]{\mathcal{P}\left(#1\right)}

\newcommand{\Reg}[1]{\text{Reg}\left(#1\right)}

\newcommand{\Qp}[1]{\left\llbracket #1 \right\rrbracket}
\newcommand{\ap}[1]{\langle #1 \rangle}
\newcommand{\bp}[1]{\left\lbrace #1 \right\rbrace}

\newcommand{\AC}{\ensuremath{\text{{\sf AC}}}} 
 
\newcommand{\SP}{\ensuremath{\text{{\sf SP}}}}
\newcommand{\SSP}{\ensuremath{\text{{\sf SSP}}}}
\newcommand{\PR}{\ensuremath{\text{{\sf PR}}}}
\newcommand{\SPFA}{\ensuremath{\text{{\sf SPFA}}}}

\newcommand{\FA}{\ensuremath{\text{{\sf FA}}}}

\newcommand{\CFA}{{\sf CFA}}

\newcommand{\2}{\mathsf{2}}

\theoremstyle{definition}
	\newtheorem{question}[theorem]{Question}

\newcommand{\UnderTilde}[1]{{\setbox1=\hbox{$#1$}\baselineskip=0pt\vtop{\hbox{$#1$}\hbox to\wd1{\hfil$\sim$\hfil}}}{}}

\newcommand{\MA}{\ensuremath{\text{{\sf MA}}}}
\DeclareMathOperator{\Add}{Add}

\newcommand{\dirlim}{\varinjlim}
\newcommand{\invlim}{\varprojlim}
\newcommand{\rcslim}{\lim\limits_{\sf rcs}}

\newcommand{\SRP}{\ensuremath{\text{{\sf SRP}}}}

\let \restr = \upharpoonright

\let \into = \longrightarrow

\let \sub = \subseteq
\let \elsub = \preccurlyeq

\let \av = \arrowvert

\let \a = \alpha
\let \b = \beta
\let \g = \gamma
\let \d = \delta

\let \l = \lambda
\let \k = \kappa
\let \m = \mu
\let \n = \nu
\let \p = \pi

\let \t = \theta

\let \s = \sigma
\let \x = \xi

\let \o = \omega

\let \P = \Pi
\let \S = \Sigma

\let \al = \aleph
\let \la = \langle
\let \ra = \rangle
\let \mtcl = \mathcal
\let \mtbb = \mathbb
\let \it = \item

\DeclareMathOperator{\range}{range}
\DeclareMathOperator{\cf}{cf}
\DeclareMathOperator{\ot}{ot}
\DeclareMathOperator{\Lim}{Lim}
\DeclareMathOperator{\MRP}{\textsf{MRP}}
\DeclareMathOperator{\CH}{\textsf{CH}}

\DeclareMathOperator{\BPFA}{\textsf{BPFA}}

\DeclareMathOperator{\STP}{\textsf{STP}}
\DeclareMathOperator{\TWCG}{\textsf{TWCG}}
\DeclareMathOperator{\TCG}{\textsf{TCG}}
\DeclareMathOperator{\LC}{\textsf{LC}}
\usepackage[mode=multiuser,status=draft,lang=english]{fixme}
\fxsetup{theme=colorsig}

\FXRegisterAuthor{M}{aM}{M}
\FXRegisterAuthor{D}{aD}{David}

\title{Incompatible category forcing axioms}

\author[D.\ Asper\'o]{David Asper\'o}

\address{David Asper\'{o}, School of Mathematics, University of East Anglia, Norwich NR4 7TJ, UK}

\email{d.aspero@uea.ac.uk}

\author[M.\ Viale]{Matteo Viale}

\address{Matteo Viale, Department of Mathematics ``Giuseppe Peano'', Univ.\ Torino, via Carlo Alberto 10, 10125, Torino, Italy}

\email{matteo.viale@unito.it}

\date{}

\begin{document}

\thanks{The first author acknowledges support of EPSRC Grant EP/N032160/1.
The second author acknowledges support of GNSAGA. 
Parts of this research was partially done whilst the authors were visiting fellows at the Isaac Newton Institute for Mathematical Sciences in the programme `Mathematical, Foundational and Computational Aspects of the Higher Infinite' (HIF)}

\subjclass[2010]{03E57, 03E35, 03E47, 03E55}

\maketitle
\pagestyle{myheadings}\markright{Incompatible category forcing axioms}

\begin{abstract} Given a cardinal $\l$, category forcing axioms for $\l$--suitable classes $\Gamma$ are strong forcing axioms which completely decide 
the theory of the Chang model $\mtcl C_\l$, modulo generic extensions via forcing notions from $\Gamma$. 
$\mathsf{MM}^{+++}$ was the first category forcing axiom to be isolated (by the second author). In this paper we present, without proofs, a general theory of category forcings, and prove the existence of $\al_1$--many pairwise incompatible category forcing axioms for $\o_1$--suitable classes. \end{abstract}

%\listoffixmes
	
	\tableofcontents

	% !TEX root = david-matteo.tex

\section{Introduction}
It is a matter of fact that forcing axioms are successful in settling a wide range of problems undecidable on the basis of the commonly accepted axioms for set theory. 
Some explanation for this success should exist.

This paper 
provides such an explanation in line with a series of results appearing 
in~\cite{VIAAUDSTEBOOK,AUDVIA,VACVIA,VIAMM+++,VIAMMREV,VIAUSAX}.
One of the key observations we make is that forcing 
axioms come in pairs with generic absoluteness properties for their models: briefly, we show that
 there are natural strengthenings $\CFA(\Gamma)$ of well--known forcing axioms such as $\mathsf{PFA}$, 
$\mathsf{MM}$, and so on
-- where $\Gamma$ is a suitably chosen class of forcings, e.g.\ the class of forcings which are proper, semiproper, stationary set preserving, etc.\ --
with the property that forcings in $\Gamma$ which preserve $\CFA(\Gamma)$
do not change the theory of a large fragment of the universe of sets. We refer to these statements $\CFA(\Gamma)$ as \emph{category forcing axioms}.

The particular fragments of the universe we are referring to are the Chang models. Given an infinite cardinal  $\lambda$, \emph{the $\lambda$--Chang model}, denoted by $\mathcal C_\lambda$, is the $\subseteq$--minimal transitive model of $\ZF$ containing all ordinals and closed under $\lambda$--sequences. It can be construed as $\mathcal C_\lambda=L(\Ord^\lambda)$.\footnote{Where $L(\Ord^\lambda)=\bigcup_{\alpha\in\Ord}L(\Ord^\lambda\cap V_\a)=\bigcup_{\a\in \Ord}L[\Ord^{\lambda}\cap V_\alpha]$.}  In particular, $L(\mathcal P(\lambda))$ is a definable inner model of $\mathcal C_\lambda$, and $H_{\lambda^+}\subseteq \mathcal C_\lambda$ is definable in $\mathcal C_\lambda$ from $\lambda$. As is well--known, $\mathcal C_\lambda$ need not satisfy $\AC$. For instance, by a result of Kunen \cite[Thm. 1.1.6 and Rmk. 1.1.28]{LARSONBOOK}, if there are $\lambda^+$--many measurable cardinals, then $\mathcal C_\lambda\models\lnot\AC$. 

One natural starting point of our investigation is the following classical result of Woodin, which shows that, in the presence of large cardinals, the theory of the $\o$--Chang model is completely immune to set--forcing.

\begin{theorem}\label{gen-abs-chang-woodin} (Woodin) Suppose there is  a proper class of Woodin cardinals, $\mathcal P$ is a forcing notion, and $G$ is $\mathcal P$--generic over $V$. Then $$(\mathcal C_\omega^V; \in, r)_{r\in\mathcal P(\omega)^V}$$ and    $$(\mathcal C_\omega^{V[G]}; \in, r)_{r\in\mathcal P(\omega)^V}$$ are elementarily equivalent structures.
\end{theorem}

It is of course not possible to extend the generic absoluteness contained in Theorem \ref{gen-abs-chang-woodin} to higher Chang models, even if we restrict to forcings preserving stationary subsets of $\omega_1$ (or even to, say, proper forcings). The reason is simply that both $\CH$ and $\lnot\CH$ are statements than can always be forced by a proper forcing, and which are expressible in $H_{\omega_2}$.\footnote{There are of course many other pairs of such statements.} 

One natural retreat at this point is to aim for a higher analogue of Theorem \ref{gen-abs-chang-woodin} conditioned both to a suitable class of forcing notion \emph{and} to a reasonable theory $T$, i.e., a statement of the following form.
\emph{
\begin{quote}\label{general-gen-abs} 
(Large cardinals + $T$) Suppose $\mathcal P$ is a forcing notion in  $\Gamma$, $G$ is $\mathcal P$--generic over $V$, and $V[G]\models T$. Then     $$(\mathcal C_\lambda^V; \in, r)_{r\in\mathcal P(\lambda)^V}$$ and    $$(\mathcal C_\lambda^{V[G]}; \in, r)_{r\in\mathcal P(\lambda)^V}$$ are elementarily equivalent structures.
\end{quote}
}
This is indeed our approach. We show that our category forcing axioms $\CFA(\Gamma)$ provide such theories. Specifically, we prove the following.

\begin{theorem}\label{gen-abs-cfa} Suppose there is a proper class of supercompact cardinals. Let $\lambda$ be an infinite regular cardinal. Suppose $\Gamma$ is a $\lambda$--suitable\footnote{We will define all relevant notions in due course.} class and $\CFA(\Gamma)$ holds. Suppose $\mathcal P\in\Gamma$, $G$ is $\mathcal P$--generic over $V$, and $V[G]\models \CFA(\Gamma)$. Then     $$(\mathcal C_\lambda^V; \in, r)_{r\in\mathcal P(\lambda)^V}$$ and    $$(\mathcal C_\lambda^{V[G]}; \in, r)_{r\in\mathcal P(\lambda)^V}$$ are elementarily equivalent structures.
\end{theorem}

%The precise large cardinal hypothesis we use lies, consistency--wise, between hugeness and the existence of a $2$--huge cardinal. 
We also prove that the result is non--vacuous, in the sense that, modulo the existence of a $2$--superhuge cardinal, $\CFA(\Gamma)$ can always be forced if $\Gamma$ is $\lambda$--suitable.\footnote{Also, the axioms $\CFA(\Gamma)$ are not some sort of artificial contrivance designed to make Theorem \ref{gen-abs-cfa} true (e.g., designed to explicitly provide a coding of the theory of the ground model Chang model), but are relatively simple -- and this includes syntactical simplicity -- and natural strengthenings of existent forcing axioms (s.\ below).}

In view of Theorem \ref{gen-abs-cfa}, category forcing axioms turn forcing from
a tool useful to prove undecidability results into a tool useful to prove theorems: Suppose, for simplicity of exposition, than $\l$ is a definable cardinal in the ambient set theory, e.g.\ $\l=\o_1$. Then, in order to show that $\CFA(\Gamma)$, together with the ambient set theory, implies $\phi^{L(\Ord^\lambda)}$ for a given sentence $\phi$, it suffices to show that $\CFA(\Gamma)$, together with the ambient set theory, implies the existence of a forcing notion $\bool{B}$ in $\Gamma$ preserving $\CFA(\Gamma)$ and forcing $\phi^{L(\Ord^\lambda)}$.

On the other hand, and this is the second main point we make in this paper, the complete pictures about the relevant Chang model that we get from these strong axioms are provably incompatible, \emph{even} when the  classes of forcing notions they refer to are comparable under inclusion; for example we prove that, letting $\Gamma$  be
the class of proper forcings and $\Gamma'$ that of forcings preserving stationary subsets of $\omega_1$,
$\CFA(\Gamma)$ and $\CFA(\Gamma')$ give logically incompatible axiomatizations of set theory,
even if $\CFA(\Gamma)$ implies $\bool{PFA}$, $\CFA(\Gamma')$ implies $\bool{MM}$,
and $\bool{MM}$ implies $\bool{PFA}$. Actually we produce $\aleph_1$--many classes $\Gamma$
with the property that, modulo large cardinals, 
\begin{itemize}
\item $\CFA(\Gamma)$ is consistent
and 
\item $\CFA(\Gamma)$ and $\CFA(\Gamma')$ are pairwise incompatible axioms
if $\Gamma\neq\Gamma'$.%\footnote{This is incompatibility even in first order logic.}
\end{itemize}
\medskip

We will next briefly introduce category forcing, and will give some intuition behind the corresponding notion of category forcing axiom. Given a class of forcings $\Gamma$ closed under two--step iterations, 
we analyze the generic multiverse
\[
V(\Gamma)=\bp{V^{P}:P\in \Gamma^V}
\]
computed over a model of set theory $V$.
We stipulate that the accessibility relation between the elements of this multiverse is given by the requirement
that $V^Q$ be accessible from $V^P$ if and only if $V^Q$ can be obtained as a generic extension of $V^P$
by a forcing in $\Gamma^{V^P}$.
There are a number of good reasons to adopt this approach; we refer the reader to~\cite{VIAUSAX} for more details.
%
%%There are certain statements (such as $\omega=\omega_1$, where $\omega$ is the unique set 
%%satisfying the first order property of being the least infinite ordinal and  $\omega_1$
%%is the unique set 
%%satisfying the first order property of being the least uncountable ordinal) which are clearly false in $V$ but
%%which can be forced to hold in some generic extension of $V$, for example in $V[G]$ where $G$ is $V$-generic for 
%%$\Coll(\omega,\omega_1)$. We want to rule out this counterexamples.
%%$\mathsf{PFA}$ $\mathsf{MM}$ and other forcing axioms provide a solution to a variety of problems formalizable in third order arithmetic
%%by $\Pi_2$-properties in parameter $\omega_1$ or more generally using parameters in $H_{\omega_2}$. Hence we are led to consider just forcing notions which at least
%%do not change the basic meaning of the parameters used to formalize the problem.
Once we adopt this strategy to analyze the generic multiverse,
our purpose becomes to develop a general theory of category forcings, i.e., forcings 
whose conditions are partial orders in a given class $\Gamma$ and in which the order relation $Q\leq_\Gamma P$ is given
by the requirement that the corresponding forcing extensions $V^P$, $V^Q$ be such that $V^Q$ is reachable
from $V^P$ using a forcing in $\Gamma^{V^P}$.
In the presence of strong enough background assumptions, we will be able to link our analysis of these category forcings to generic absoluteness.

Let us be more specific in our definition of a category forcing.
By a category forcing we understand a class--forcing whose conditions are forcing notions, ordered by (a refinement of) the 
notion of absorption: a partial order $Q$
will refine a partial order $P$
if whenever $H$ is $V$--generic for $Q$ there is some $G\in V[H]$ which is $V$--generic for $P$ 
(often we will require $G$ to have further nice properties in $V[H]$ besides $V$--genericity).
It will soon become transparent that it is more convenient to develop this theory resorting to a Boolean algebraic formulation 
of the relevant properties; in any case, given that posets  with isomorphic Boolean completions produce the 
same generic extensions, an immediate outcome of our definition of order between posets brings to our 
attention that there is no loss of information in taking our relevant forcing conditions to be complete Boolean algebras, rather than the variety of
posets which 
are contained in them as dense suborders.
The ordering relation between posets we outlined above translates in this context in the algebraic notion
asserting that $V^Q$ is accessible from $V^P$ 
only if there is a complete (possibly non--injective) homomorphism $i:\RO(P)\to\RO(Q)$,
where $\RO(R)$ denotes the Boolean completion of $R$.
Hence, we will focus in this paper on categories 
$(\Gamma,\to^\Theta)$, with $\Gamma$ a class of complete Boolean algebras
and $\to^\Theta$ a class of complete homomorphisms between them;
the ordering is given by setting
$\bool{C}\leq_\Theta\bool{B}$ if there is a homomorphism $i:\bool{B}\to\bool{C}$ in $\to^\Theta$.

\begin{remark}
$(\Gamma, \leq_\Theta)$ can be forcing equivalent to the trivial partial order.
For example, suppose $\Gamma=\Omega_{\aleph_0}$ is the class of all complete Boolean algebras and 
$\Theta$ is the class of all complete embeddings. Then any
two conditions in $(\Gamma,\leq_\Theta)$ are compatible, giving that
$(\Gamma,\leq_\Theta)$ is forcing equivalent to the trivial partial order.
This is the case since for any pair of partial orders $P$, $Q$ and any set $X$ of size at least
$2^{|P|+|Q|}$ there are 
complete injective homomorphisms of $\RO(P)$ and $\RO(Q)$ into the Boolean completion of
$\Coll(\omega,X)$ (see~\cite[Thm A.0.7]{LARSONBOOK} and the remark thereafter). 
These embeddings witness the compatibility of $\RO(P)$ and $\RO(Q)$ in $(\Gamma, \leq_\Theta)$.

On the other hand, letting $\SSP$ be the class of all complete Boolean algebras preserving stationary subsets of $\omega_1$ and, again, letting  $\to^\Theta$ be given by the class of all complete homomorphisms with a generic quotient in $\Gamma$ or even the class of all complete homomorphisms, it follows that
$(\SSP, \leq_\Theta)$ is non--trivial.
For this, observe that if $P$ is Namba forcing on $\aleph_2$ and $Q$ is
 $\Coll(\omega_1,\omega_2)$, then $\RO(P)$ and $\RO(Q)$ are incompatible conditions in 
$(\SSP, \leq_\Omega)$:
 If $\bool{R}\leq\RO(P)$, $\RO(Q)$, we would have that if $H$ is $V$--generic for 
 $\bool{R}$, then $\omega_1^{V[H]}=\omega_1$ (since $\bool{R}\in\SSP$)
 and there are $G$, $K\in V[H]$, $V$--generic filters for 
 $P$ and $Q$, respectively (since $\bool{R}\leq\RO(P)$, $\RO(Q)$).
 $G$ would give rise in $V[H]$ to a sequence cofinal in $\omega_2^V$ of order type $\omega$, while $K$
would allow one to define in $V[H]$ a sequence cofinal in $\omega_2^V$ of order type $\omega_1^V$.
These two facts together would entail that $\cof(\omega_1^V)=\omega$ in $V[H]$,
contradicting the assumption that $\omega_1^{V[H]}=\omega_1^V$.
\end{remark}
%\end{proof}

%In almost all cases $\Gamma$ is closed under two steps iterations, $\to^\Gamma$ is the class
%of complete homomorphisms with a generic quotient in $\Gamma$.

The basic theory of category forcings requires that the category forcing $(\Gamma, \to^\Theta)$ at hand 
satisfy 
the following properties:\footnote{
The first two properties on $\Gamma$ are reminiscent, modulo the requirement of restricting 
to \emph{complete} homomorphisms (and \emph{complete} 
Boolean algebras) rather than arbitrary homomorphism and Boolean algebras, to  
Birkhoff's characterization of categories of varieties.
%Birkhoff's algebraic characterization of a variety $\Delta$
%is equivalent to the assertion that $\Delta$ is a class of first order structures definable by equations on terms.
It is well possible that (following Birkhoff's type of results) 
in our list of properties for $\Gamma$,
(some natural strengthening of) the last requirement, asking that $\Gamma$ be simply definable in logical terms, is a 
direct outcome of the previous requests we impose on $\Gamma$. We have not explored this topic further, though.}
\begin{enumerate}
%\item $\Gamma$ is closed under isomorphisms, restrictions and complete subalgebras.
\item $\Gamma$ is closed under preimages by complete injective homomorphisms.\footnote{In the sense that if $\bool{B}$ and $\bool{C}$ are complete Boolean algebras, $\bool{C}\in\Gamma$ and there is complete injective homomorphism from $\bool{B}$ into $\bool{C}$, then $\bool{B}\in\Gamma$.}
\item $\Gamma$ is closed under lottery sums.
\item $\Gamma$ is closed under two--step iterations, and $\to^\Theta$ (in this case it is more appropriate to write 
$\to^\Gamma$) is the class
of complete homomorphisms with a generic quotient in $\Gamma$.
\item $\Gamma$ is closed under many types of (co)limit constructions 
(essentially, there is an iteration theory granting that $\Gamma$ be closed under limits of iterations built using a given rule of a reasonably general form).
\item $\to^\Gamma$ has a certain type of rigidity (which means that for a dense subclass $\bool{Rig}^\Gamma$ 
of $\Gamma$, with respect
to $\leq_\Gamma$, there is at most one homomorphism in $\to^\Gamma$ between elements of $\bool{Rig}^\Gamma$).
\item $\Gamma$ can be defined by a property of low logical complexity.
\end{enumerate}

The first main outcome of our analysis is the following:

\begin{quote}
\emph{Assume $\Gamma$ is a class of complete Boolean algebras which
satisfies the above properties, let $\to^\Gamma$ denote the class of complete homomorphisms
with a generic quotient in $\Gamma$, and let $\leq_\Gamma$ be the corresponding partial order on $\Gamma$.
Then there is a cardinal $\lambda_\Gamma$, uniquely associated to $\Gamma$, 
and an axiom $\CFA(\Gamma)$ (which, as we will see next with a bit more detail, can be formulated as the assertion that
a certain class of posets is dense in $(\Gamma,\leq_\Gamma)$) such that the following holds:}
\begin{enumerate}
\item 
\emph{
$\FA_{\lambda^+_\Gamma}(\bool{B})$ provably fails for some $\bool{B}\in\Gamma$.\footnote{$\FA_\kappa(\bool{B})$ holds if and only if any 
$\kappa$--sized family of dense open subsets of $\St(\bool{B})$ has non--empty intersection, where 
$\St(\bool{B})$ is the Stone space of ultrafilters on $\bool{B}$; in other words, if and only if for every $\kappa$--sized family $\mathcal D$ of dense subsets of $\bool{B}$ there is a filter of $\bool{B}$ intersecting all members of $\mathcal D$.}}
\item 
\emph{
$\CFA(\Gamma)$ entails that $\FA_{\lambda_\Gamma}(\bool{B})$ holds if $\bool{B}\in \Gamma$.}
\item
\emph{
In the presence of sufficiently strong large cardinals, the first order theory of $L(\Ord^{\lambda_\Gamma})$ is invariant with respect to forcings in $\Gamma$ which preserve
$\CFA(\Gamma)$ (this is Theorem \ref{gen-abs-cfa}).}
\item
\emph{
$\CFA(\Gamma)$ is consistent relative to strong large cardinal axioms. In fact, if $\delta$ is a $2$--superhuge cardinal, then the intersection of the category $(\Gamma, \leq_\Gamma)$ (using the above notation) with $V_\delta$ forces $\CFA(\Gamma)$.}
\end{enumerate}
\end{quote}

%\begin{aMnote}{I'm not sure this remark is pertinent. the formulation of $\CFA(\Gamma)$ does not need any reference to $\lambda_\Gamma$.}
\begin{remark} Strictly speaking, we should have written something like $\CFA_{\lambda_\Gamma}(\Gamma)$ instead of $\CFA(\Gamma)$, as the formulation of the axiom is certainly dependent on $\lambda_\Gamma$. Nevertheless, we will always be able to read off $\lambda_\Gamma$ from $\Gamma$, and so there will be no ambiguity in just writing $\CFA(\Gamma)$. \end{remark}
%\end{aMnote}

Items 1 and 2 above give a maximality  property of the axiom $\CFA(\Gamma)$ 
naturally formulated in topological terms 
(specifying the exact amount of a strong form of Baire's category theorem that holds for all Boolean algebras in $\Gamma$), 
while item 3 outlines a maximality property of the axioms $\CFA(\Gamma)$ formulated in logical terms
(the theory of a large fragment of the universe cannot be changed using forcings in $\Gamma$ which preserve the axiom).

Before saying more about the axioms $\CFA(\Gamma)$ in general, let us recall the following natural and well--known strengthening $\bool{MM}^{++}$ of the forcing axiom $\FA_{\omega_1}(\SSP)$, also known as \emph{Martin's Maximum}  and denoted by $\bool{MM}$ (see \cite{FOREMAGISHELAH}): $\bool{MM}^{++}$ holds if and only if for every $\mtcl P\in \SSP$, every collection $\mtcl D$ of $\al_1$--many dense subsets of $\o_1$, and every sequence $(\dot S_i)_{i<\o_1}$ of $\mtcl P$--names for stationary subsets of $\o_1$ there is a filter $G\sub\mtcl P$ such that
\begin{itemize}
\item $G\cap D\neq\emptyset$ for all $D\in\mtcl D$, and such that 
\item $\{\nu<\o_1\,:\, p\Vdash_{\mtcl P}\nu\in \dot S_i\mbox{ for some }p\in G\}$ is a stationary subset of $\o_1$ for each $i<\o_1$.
\end{itemize}

As proved in \cite{FOREMAGISHELAH}, $\bool{MM}^{++}$ can be forced to be true, starting from the existence of a supercompact cardinal. 
The second author of the present article observed in \cite{VIAMMREV} that 
$\bool{MM}^{++}$ can be characterized, in the presence of a proper class of Woodin cardinals, as the following density principle for the $\SSP$ category.

\begin{proposition}\label{viale0} Suppose there is a proper class of Woodin cardinals. Then the following are equivalent.

\begin{enumerate}
\item $\bool{MM}^{++}$
\item For every $\mtcl P\in\SSP$ there is a pre-saturated $\o_2$--tower $\mtcl T$ and a $\mtcl P$--name $\dot{\mtcl Q}$ such that 
\begin{enumerate}
\item $\Vdash_{\mtcl P}\dot{\mtcl Q}\in\SSP$ and
\item $\mtcl T$ is forcing equivalent to $\mtcl P\ast\dot{\mtcl Q}$.
\end{enumerate}
Furthermore, we may take $\mtcl T$ to be of cardinality the least Woodin cardinal above $\arrowvert\mtcl P\arrowvert$.

\end{enumerate}
\end{proposition}
It is out of the scope of the present paper to introduce and develop the basic theory of pre-saturated $\lambda$--towers.
It is sufficient for us to know that the conjunction of requirements 2(a) and 2(b) above exactly corresponds to the statement 
$\mtcl T\leq_\SSP\mtcl \RO(\mtcl P)$.
In particular, if there are class--many Woodin cardinals,  then
$\bool{MM}^{++}$ holds if and only if there is a dense class of 
pre-saturated $\o_2$--towers in the category forcing $(\SSP,\leq_\SSP)$.\footnote{Incidentally, note that  $\bool{MM}^{++}$, in its original formulation, is a  $\Pi_2$ sentence, and that so is its characterization in (2) of Proposition \ref{viale0}  with the `furthermore' clause. Also, note that the existence of a proper class of Woodin cardinals is a $\Pi_3$ sentence.}
Now, whenever $\mtcl T$ is a $\lambda$--pre-saturated tower of height $\delta$ and $G$ is
$V$--generic for $\mtcl T$, in $V[G]$ one can define en elementary embedding 
$j:V\to M$ with critical point $\lambda$, and such that
$j(\lambda)=\delta$ and $M^{<\delta}\subseteq M$. This suffices to have that
$L(\Ord^{{<}\lambda})^{V[G]}=L(\Ord^{{<}\lambda})^M$, and therefore also that
$j\restriction L(\Ord^{{<}\lambda})^V$
defines an elementary embedding of $L(\Ord^{{<}\lambda})^{V}$ into $L(\Ord^{{<}\lambda})^{V[G]}$
with critical point $\lambda$ and such that $j(\lambda)=\delta$.
The interested reader can consult~\cite{FOREMAN,LARSONBOOK} or the forthcoming book~\cite{VIAAUDSTEBOOK}
for a development of the theory of generic ultraprowers induced by pre-saturated tower forcings.

In \cite{VIAMM+++}, the second author defines a certain strengthening of the notion of pre-saturated 
$\o_2$--tower, which he calls \emph{$\SSP$--super rigid tower} in \cite{VIAAUDSTEBOOK} and strongly presaturated towers in \cite{VIAMM+++}, together with the corresponding strengthening of 
$\bool{MM}^{++}$ (as given by the characterisation in Proposition \ref{viale0}).  

\begin{definition} $\bool{MM}^{+++}$ holds if and only if for every $\mtcl P\in\SSP$ there is an $\SSP$-super rigid tower $\mtcl T$ and a $\mtcl P$--name $\dot{\mtcl Q}$ such that 
\begin{enumerate}
\item $\Vdash_{\mtcl P}\dot{\mtcl Q}\in\SSP$, and
\item $\mtcl T$ is forcing equivalent to $\mtcl P\ast\dot{\mtcl Q}$.
\end{enumerate}
I.e., $\bool{MM}^{+++}$ holds if and only if there is a dense class of 
$\SSP$-super rigid towers in the category forcing $(\SSP,\leq_\SSP)$.
\end{definition}

In particular $\bool{MM}^{+++}$ is a natural strengthening of $\bool{MM}^{++}$.
He also proves the following theorems.

\begin{theorem}\label{con-mm+++} If there is an almost super--huge cardinal $\delta$,\footnote{We do not need to get into the definition of this large cardinal notion, but will only say that consistency-wise it is somewhat weaker than the existence of a huge cardinal.} then there is a forcing notion $\mtcl P\sub V_\delta$ such that $\Vdash_{\mtcl P}\bool{MM}^{+++}$. In fact, the intersection of $(\SSP, \leq_{\SSP})$ with $V_\delta$ forces $\bool{MM}^{+++}$.
\end{theorem}   

\begin{theorem}\label{gen-abs-chang0} Suppose there is a proper class of almost super--huge cardinals and $\bool{MM}^{+++}$ holds. If $\mtcl P\in\SSP$, $\Vdash_{\mtcl P}\bool{MM}^{+++}$, and $G$ is $\mtcl P$--generic over $V$, then $$(\mtcl C_{\o_1}^V; \in, r)_{r\in\mtcl P(\o_1)^V}$$ and $$(\mtcl C_{\o_1}^{V[G]}; \in, r)_{r\in\mtcl P(\o_1)^V}$$ are elementarily equivalent structures.
\end{theorem}

The more general theory we will present generalizes these results to arbitrary suitable classes $\Gamma$ of forcings other than $\SSP$. The formulation of the corresponding category forcing axiom becomes thus the following density property for $\Gamma$.

\begin{definition}\label{definition-of-cfa(gamma)} Given a $\lambda$--suitable class $\Gamma$ of forcing notions, \emph{the category forcing axiom for $\Gamma$}, $\CFA(\Gamma)$, is the following statement: The class of $\Gamma$-super rigid towers $\mtcl T\in\Gamma$ is dense in the category $(\Gamma, \to^\Gamma)$ given by $\Gamma$. \end{definition} 

We will give the definition of the notion of $\Gamma$-super rigid  tower in the next subsection, but will rather refer the reader to ~\cite{VIAAUDSTEBOOK} for a thorough development of the main properties of these towers. We do point out that this is a natural, in the context of the present analysis of category forcing, strengthening of the classical notion of pre-saturated tower. Furthermore, the definition of $\Gamma$-super rigid tower is both $\Sigma_2$ and $\Pi_2$.\footnote{Given that all our classes will be $\Sigma_2$ definable (possibly from some parameter), it follows from the above that all our axioms $\CFA(\Gamma)$ will be $\P_3$ statements, possibly in some parameter defining $\Gamma$.} It also follows, in light of the characterization in Proposition \ref{viale0} of $\bool{MM}^{++}$, that the resulting category forcing axioms are naturally seen as strong forms of the forcing axiom for the corresponding class of forcings. 
%(and in fact, as we have already mentioned, they imply the forcing axiom).  
Moreover, in the next subsection we will also give a (second--order) equivalent 
(in the presence of a proper class of supercompact cardinals) formulation of 
$\CFA(\Gamma)$ 
which does not mention towers.
%, a formulation of possibly more illuminating nature than $\CFA(\Gamma)$ 
%in its original formulation.
%\footnote{On the other hand, in its formulation it is less of a `strong forcing axiom' than $\CFA(\Gamma)$ is.} 

The second main outcome of our analysis gives the following:
\emph{
\begin{quote}
There are uncountably many classes $\Gamma$ with $\lambda_\Gamma=\omega_1$ for which $\CFA(\Gamma)$ is consistent. These $\Gamma$ 
produce pairwise incompatible first order theories for $H_{\omega_2}$ which are generically invariant
with respect to forcings in $\Gamma$ preserving $\CFA(\Gamma)$.
Among these $\Gamma$ we mention: the class of stationary set preserving forcings, 
%the class of semiproper forcings,
the class of proper forcings, the class of (semi)proper forcings which are also $\omega^\omega$--bounding, and so on.\footnote{If we pay attention to the form of these category forcing axioms, we see that the conclusion that $\CFA(\Gamma_0)$ and $\CFA(\Gamma_1)$ be incompatible axioms even when $\Gamma_0\subseteq\Gamma_1$ is not as counterintuitive as it could at first seem. These axioms say that a certain subclass $\Gamma^\ast$ of $\Gamma$ is dense in the category corresponding to $\Gamma$ \emph{via the corresponding accessibility relation $\leq_\Gamma$}, i.e., for every $\bool{B}\in\Gamma$ there is some $\bool{C}\in\Gamma^\ast$ such that $\bool{C}\leq_\Gamma\bool{B}$. Even if $\Gamma_0\subseteq\Gamma_1$ and $\Gamma_0^\ast\subseteq\Gamma_1^\ast$, it does not follow that $\CFA(\Gamma_1)$ should imply $\CFA(\Gamma_0)$  (since $\leq_{\Gamma_1}$ need not be contained in $\leq_{\Gamma_0}$).}
\end{quote}
}

It follows from the above that any claim of naturalness of these axioms made on the face of the amount of logical completeness they account for -- as given by Theorem \ref{gen-abs-cfa} %i.e., point (3) above 
-- is undermined by their mutual incompatibility, 
at least if we expect natural axioms to be mutually compatible. On this conception of naturalness, maximality considerations based on topological properties provide better justifications. From this point of view, $\bool{MM}$ is more natural than, say, $\bool{PFA}$, and hence, due to the fact that $\CFA(\Gamma)$ implies $\FA_{\lambda_\Gamma}(\Gamma)$, a similar comparison applies to the corresponding category forcing axioms. 

One more general conclusion of the above is that considerations of logical completeness do not by themselves constitute a feasible criterion towards isolating natural axioms for set theory, even when this completeness is derived from non--contrived natural strengthenings of forcing axioms, as is the case for category forcing axioms.\footnote{At least if, again, we expect natural axioms to be mutually compatible.} Maximality considerations based on topological properties, 
as the ones motivating forcing axioms, are preferable in this respect. On the other hand, as our first set of results for category forcing axioms shows, we may have that logical completeness can be a pleasant added feature of axioms satisfying a maximality property with respect to topological considerations. 

Finally, it is worth pointing out that our category forcing axioms may be seen as \emph{conditional} maximality principles aiming at describing, given a suitable class $\Gamma$ of forcing notions, what the universe will look like after `saturating' it, in the presence of sufficiently strong large cardinals, via the formation of forcing extensions coming \emph{only} from $\Gamma$. One moral of our results is that, so long as our classes have a reasonable iteration theory, as well as our other requirements, the corresponding saturated pictures of the relevant Chang model that we obtain are incompatible, even when the classes  themselves are comparable under inclusion.

We now give a brief description of the structure of the paper. In the first part we 
outline the general theory of our category forcings, leading to the formulation of $\CFA(\Gamma)$
for a class of forcings $\Gamma$,
to the listing of the desired consequences of this axiom, and to a simple set of properties on $\Gamma$ 
sufficient to make $\CFA(\Gamma)$ a consistent axiom for which the intended form of generic absoluteness holds (where both outcomes are conditioned to the presence of sufficiently strong large cardinals).
Nonetheless we will omit almost all proofs, since this would make the paper far too long. 
The interested reader will find all these things in~\cite{VIAAUDSTEBOOK}. 
%Our approach will be that of giving precise mathematical definitions of all relevant notions and
%precise statements of all the relevant results, nonetheless we will just limit ourselves to give a general 
%overview of how we are led to formulate certain properties of a category forcing $\Gamma$ and 
%how these properties are used to analyze the theory of its corresponding multiverse $V(\Gamma)$.
The second part of the paper gives a detailed list of examples of classes $\Gamma$ which fulfill the requests outlined in the first part.

From now on, we will work in the Morse--Kelley axiomatization of set theory $\MK$. It will be convenient for us to assume that
the universe of sets $V$ is a model of the theory
\begin{equation}
\MK^\ast=\MK+
\emph{there are \emph{stationarily} many
inaccessible cardinals,} 
\end{equation}
since we will often need to handle proper classes and relavitizations of these to initial segments of the universe. 

We will also need to talk about specific elements of $V$, among others: 
\begin{itemize}
\item
the first uncountable cardinal $\omega_1$, 
\item 
a fixed countable indecomposable ordinal $\rho$,
\item
a regular uncountable cardinal $\lambda$.
\end{itemize} 
We will need to keep track of precisely how the properties of these objects affect the properties of certain 
classes of forcings $\Gamma$ defined in $V$ using them as parameters 
(such as the class of semiproper forcings, the class of 
$\rho$--proper forcings, the class of ${<}\lambda$--closed forcings, etc). 
It will be important to know which axioms of set theory holding in $V$ and which properties of the objects 
used to define $\Gamma$
are used to establish certain properties of this $\Gamma$. 
To properly address this issue we proceed a follows:

\begin{itemize}
\item We expand the language of set theory to a language $\mathcal L_\in^+$ with new constant symbols whose interpretation
in the standard model of set theory is clear. For this reason, the name of each of these constant symbols will be that
of their interpretation in $V$ (e.g.\ $\omega_1$, $\rho$, etc).
\item We add to the axioms of $\MK^\ast$
%those other properties of $V$ and of those specific elements of $V$ which we 
%need to carry our arguments in $V$, specifically in what follows $\rho$ is a fixed countable indecomposable ordinal of $V$,
%$\omega_1$ is the first uncountable cardinal of $V$, $\kappa$ is a fixed regular uncountable cardinal of $V$ and we add
%to the $\MK$-axiom list 
certain first order sentences that we will need in order to carry out the relevant arguments in $V$; these will be statements like ``\emph{$\omega_1$ is the first uncountable cardinal}'',
``\emph{$\rho$ is a countable indecomposable ordinal}'', ``\emph{$\lambda$ is a regular uncountable cardinal}'' 
(and possibly a few others). The particular statements we are considering are always made explicit, or else will be clear from the context. 
\item 
Given a theory $T$ in the language $\mathcal L_\in^+$ extending  $\MK^\ast$ and given a class of forcings $\Gamma$ defined by a formula
in $\mathcal L_\in^+$, we will introduce a notion of canonicity relating $\Gamma$ and $T$
which holds if the following two conditions are satisfied: 
\begin{itemize}
\item
The desired properties of $\Gamma$ follow from the axioms of $T$; 
\item 
For any axiom $\sigma$ of 
$T$, it is provable from $T$ that $\sigma$ holds in all forcing 
extensions obtained by forcings in $\Gamma$.\footnote{Where of course every relevant constant symbol is interpreted in the generic extension as the same object as in the ground model.} 
\end{itemize}
\end{itemize}

We also adopt the following notational conventions:
\begin{itemize}
\item 
$\MK^\ast$ is formalized in a language $\bp{\in, =, \subseteq, \text{Set}}$ with $\text{Set}$ a unary predicate and
the axiom stating $\forall x(\text{Set}(x)\leftrightarrow\exists y\,x\in y)$.
\item $V$ denotes the universe of sets \emph{and classes}. This is the standard model of $\MK^\ast$
and contains a proper class of inaccessible cardinals.
%\item For any inaccessible $\delta$ $V_{\delta
\item
We say that $\phi$ is a $\Sigma_n$--property over some theory $T\supseteq\MK^\ast$ 
if it is provably equivalent in $T$ to a $\Sigma_n$--formula according to the Levy hierarchy \emph{whose
quantifiers range just over sets}. We say that $\phi$ is $\Sigma^1_n$ over $T$
if it is provably equivalent in $T$ to a $\Sigma_n$ formula according to the Levy hierarchy
\emph{whose
quantifiers can range over sets and classes}.
\item
If we are interested just in $\Sigma_n$--properties, we often consider the $\ZFC$--models $V_\delta$ rather than the 
$\MK$--models
$V_{\delta+1}$ for inaccessible $\delta$.

%\footnote{$\kappa$-canonical theories will play a central
% role in subsequent sections, and $\MK^*+$\emph{$\kappa$ is a regular cardinal}
%is not $\kappa$-canonical for any $\kappa$, exactly because it is axiomatized using a $\Pi^1_1$-statement.}.
\end{itemize}

Note that $\MK^*$ is obtained by adding to $\MK$ a given $\Pi^1_1$--statement.

From now on, when we state that a certain property holds for a poset $\mtcl P$, 
we automatically infer that the property holds for its Boolean completion $\RO(\mtcl P)$.
Conversely, when we assert that a complete Boolean algebra $\bool{B}$ has a property defined for posets,
we mean that the poset $\bool{B}^+=\bool{B}\setminus\bp{0_{\bool{B}}}$ with the order inherited from 
$\bool{B}$ has this property. This is not problematic, since all the properties of posets we define 
are forcing invariant, i.e. stable with respect to the equivalence relation on posets given 
by $P\equiv Q$ if and only if $P$ and $Q$ have isomorphic Boolean completions.

	%\input{david-introduction}
	
	% !TEX root = david-matteo.tex

\section{$\CFA(\Gamma)$ and generic absoluteness for $L(\Ord^{\lambda})$.}

In this section we introduce the key definitions and results on category forcings. We proceed as follows:
\begin{itemize}
\item In Subsection~\ref{subsec:genabsforax}, and building on the stationary tower proof of Woodin's  generic absoluteness theorem for the $\omega$--Chang  model in the presence of large cardinals, we give an informal outline of our general strategy to link forcing axioms to 
generic absoluteness through category forcings.
\item This will bring us to isolate certain key features that category forcings must have. 
These features are collected in the notion of  
$\lambda$--suitable class in Subsection~\ref{subsec:suitclassforc}.
%\item We give a rigorous  definition of what is a $\kappa$-suitable class of forcings $\Gamma$
%in section. 
\item In Subsection~\ref{subsec:mainresultsM} we give the definition of $\CFA(\Gamma)$ 
for a $\lambda$--suitable class of forcings $\Gamma$
and state the two main results for such a $\Gamma$: on the one hand that $\CFA(\Gamma)$ makes  the theory of
$L(\Ord^{\lambda})$ generically invariant with respect to forcings in $\Gamma$ which preserve
$\CFA(\Gamma)$, and on the other hand that $\CFA(\Gamma)$ is consistent relative to large cardinal axioms.
However (due to their length) we completely omit the proofs of these results and refer
the reader to the forthcoming~\cite{VIAAUDSTEBOOK}. 
\item In Subsection~\ref{subsec:freezetotrig} we give a detailed description of $\Gamma$--rigidity and $\Gamma$--freezea\-bi\-li\-ty, which are two of the key provisions
a $\lambda$--suitable class $\Gamma$ must satisfy. In Section \ref{suitable classes} 
we will show that there are a great variety of
$\omega_1$--suitable classes $\Gamma$. We felt it was worth focusing our attention on these notions given that, as we will see in next section, the main difficulty when establishing the $\omega_1$--suitability of a  class $\Gamma$ will be in showing that it has the $\Gamma$--freezeability property. 

%\item We also include a short section containing a fast review of the approach to forcing 
%iterations via directed system of complete injective homomorphisms of complete boolean algebras
%and its equivalence with the common approach via directed system of complete embeddings of posets.
\end{itemize}
The key definitions and concepts needed to follow the results in other parts of the paper are given 
in Subsections~\ref{subsec:suitclassforc} and ~\ref{subsec:mainresultsM}. The other subsections, \ref{subsec:genabsforax} and \ref{subsec:freezetotrig}, are mainly meant to give some more information
useful to grasp the content of ~\ref{subsec:suitclassforc} and ~\ref{subsec:mainresultsM}.

\subsection{Generic absoluteness and forcing axioms}\label{subsec:genabsforax}
Since the definition of $\lambda$--suitable class of forcing $\Gamma$ is rather technical, it is
useful to understand step by step why we introduce each of its provisions. This is what we do
in this section. The reader may skip it entirely and still be able to follow all the remaining parts of the paper.

%The first observation we need to make is that the axiom of choice can be seen as a global forcing axiom.
%Roughly Zorn's Lemma states that for any cardinal $\lambda$
%any partial order $P$ which has lower bounds for all its subchains of size less than
%$\lambda$ contains also subchains of length $\lambda$ (eventuallynot bounded in $P$). 
%This can be reformulated as the assertion that for all cardinals $\lambda$,
%for every $<\lambda$-closed forcing $P$, and for every family of $\lambda$-many dense subsets of $P$,
%there is a filter $G$ on $P$ meeting all the ense sets in the family. This is clearly a global forcing axiom.
%In particular the axiom of choice amounts to say that

We assume the reader is more or less familiar with Woodin's proof (by means of stationary towers)
of the generic absoluteness results (see Chapter 3 of \cite{LARSONBOOK})
for $L(\mathbb{R})$ and/or for $L(\Ord^\omega)$; let us briefly recall its salient steps.
\begin{itemize}
\item 
Consider the forcing $\Coll(\omega, {<}\delta)$ given by finite functions $s:\omega\times\delta\to\delta$ such that
$s(n,\alpha)\in\alpha$ for all $(n,\alpha)\in\dom(s)$ and ordered by reverse inclusion.
For a large enough cardinal $\delta$ in $V$
(e.g.\ supercompact, but $\delta$ being a Woodin cardinal is enough), $\Coll(\omega, {<}\delta)$ is such that: 
\begin{itemize}
\item
Whenever $G$ is $V$--generic for $\Coll(\omega, {<}\delta)$:
\begin{itemize}
\item
$\delta=\omega_1^{V[G]}$, 
\item there is an elementary map
$j:L(\Ord^\omega)^V\to L(\Ord^\omega)^{V[G]}$  definable in $V[G]$ with critical point 
$\omega_1^V$ and such that
$j(\omega_1^V)=\delta$.
\end{itemize}
\item For all $P\in V_\delta$, $P$ is a complete subforcing of $\Coll(\omega, {<}\delta)$; moreover, whenever
$H$ is $V$--generic for $P$, $\Coll(\omega, {<}\delta)^{V[H]}=\Coll(\omega, {<}\delta)^V$ and $\delta$ remains
large enough in $V[H]$ (i.e., it remains Woodin, supercompact, etc, in $V[H]$).
\end{itemize}

\item
Now assume $H$ is $V$--generic for some $P\in V_\delta$ and $G$ is $V[H]$--generic for 
$\Coll(\omega, {<}\delta)$ for some suffficiently large $\delta$ in $V$ (again, $\delta$ being a Woodin cardinal is enough).
Then $\delta$ remains sufficiently large in $V[H]$ and $G$ is also $V$--generic for $\Coll(\omega, {<}\delta)$.
Therefore we have $j_0:L(\Ord^\omega)^V\to L(\Ord^\omega)^{V[G]}$ and
$j_1:L(\Ord^\omega)^{V[H]}\to L(\Ord^\omega)^{V[G]}$, both elementary and with critical point
$\omega_1^V$ and $\omega_1^{V[H]}$, respectively, and with image of the critical points in both cases being $\delta$.
This gives that 
\[
\ap{L(\Ord^\omega)^{V[H]},\in,\pow{\omega}^V}\equiv\ap{L(\Ord^\omega)^{V[G]},\in,\pow{\omega}^V}\equiv
\ap{L(\Ord^\omega)^V,\in,\pow{\omega}^V}.
\]
\end{itemize}

Our purpose is to rerun the same proof scheme with the following modifications:
\begin{itemize}
\item We fix a $\lambda$--suitable class of forcings $\Gamma$.
\item We replace $\omega_1$ all over with $\lambda^+$. 
\item We replace all over the role of $\Coll(\omega, {<}\delta)$ with that of
$\Gamma\cap V_\delta$. Therefore we will require that for all large enough cardinals $\delta$:
\begin{enumerate} \label{keytask1}
\item $\Gamma\cap V_\delta$ preserves the regularity of $\delta$. \label{keytask1.1}
\item $\Gamma\cap V_\delta$ absorbs any forcing $\bool{B}\in\Gamma\cap V_\delta$
in such a way that: whenever $H$ is $V$--generic for $\bool{B}$,
 in $V[H]$ we have that $(\Gamma\cap V_\delta)^{V[H]}$ is the quotient of $(\Gamma\cap V_\delta)^V$
 by $H$. \label{keytask1.2}
\item If a certain axiom $\mathsf{CFA}(\Gamma)$ holds in $V$, then whenever $G$ is $V$--generic for $\Gamma\cap V_\delta$, 
we have that there
is an elementary embedding $j:V\to M$ definable in $V[G]$ with critical point $(\lambda^+)^V$, $j((\lambda^+)^V)=\delta$, 
and such that $M^\lambda\subseteq M$
holds in $V[G]$. \label{keytask1.3} \label{page:tasks}
\end{enumerate} 
\end{itemize}
Suppose we are successful with all the above tasks. Then we can run the last step of Woodin's proof as follows:
\begin{quote}
Assume $\CFA(\Gamma)$ holds in $V$ and 
$H$ is $V$--generic for some $P\in \Gamma$ such that $V^P$ forces $\CFA(\Gamma)$.
Find $\delta$ large enough with $P\in V_\delta\cap \Gamma$.

Let $G$ be $V[H]$--generic for 
$(\Gamma\cap V_\delta)^{V[H]}$.
Then $\delta$ remains large enough in $V[H]$ and $H\ast G$ is also $V$--generic for $(\Gamma\cap V_\delta)^{V}$.
Therefore (since $V$ and $V[H]$ are both models of $\mathsf{CFA}(\Gamma)$)
we have elementary maps $j_0:V\to M_0\subseteq V[H\ast G]$ and
$j_1:V[H]\to M_1\subseteq V[H\ast G]$ such that:
\begin{itemize}
\item
The critical point of $j_0$ is $(\lambda^+)^V$, the critical point of $j_1$ is $(\lambda^+)^{V[H]}$,
and both critical points are mapped to $\delta$.
\item
$L(\Ord^\lambda)^{M_i}=L(\Ord^\lambda)^{V[H\ast G]}$  for both $i=0,1$, 
since $M_i^\lambda\subseteq M_i$ for both $i=0,1$.
\end{itemize}
This gives that $\ap{L(\Ord^\lambda)^{V[H]},\in,\pow{\lambda}^V}\equiv\ap{L(\Ord^\lambda)^V,\in,\pow{\lambda}^V}$.
\end{quote}

\medskip

To get started, we require that $\Gamma$ be closed under two steps iterations and under preimages by complete injective homomorphisms; these are natural closure properties on $\Gamma$ without which most of the above arguments
cannot be run smoothly.

Let us first of all address task~(\ref{keytask1.2}). We start with the following weaker request:
\begin{enumerate}[(2')]
%\item $\Gamma\cap V_\delta$ preserves the regularity of $\delta$ if $\delta$ is large enough.
\item $\Gamma$ absorbs any forcing $\bool{B}\in\Gamma$.
%in such a way that: whenever $H$ is $V$-generic for $\bool{B}$,
% in $V[H]$ $(\Gamma\cap V_\delta)^{V[H]}$ is the quotient of $(\Gamma\cap V_\delta)^V$
% by $H$.
\end{enumerate}
To achieve this, we have a very simple strategy:
\begin{quote}
Given $\bool{B}\in\Gamma$, let us consider the map $i_{\bool{B}}:\bool{B}\into \Gamma$ given by  $$i_{\bool{B}}:b\mapsto \bool{B}\restriction b$$
If $c\leq_{\bool{B}}b$, the map $k:a\mapsto a\wedge c$ defines a complete surjective homomorphism
of $\bool{B}\restriction b$ into $\bool{B}\restriction c$ with a generic quotient in 
$\Gamma$;\footnote{More precisely, this means that if $G$ is $V$--generic for $\bool{B}\restriction b$, then  the quotient $\bool{B}\restriction c /  k(G)$ is in $\Gamma$ with Boolean value at least $\coker(k)$. This is of course the case in the present situation since $\coker(k)=c$ forces that $\bool{B}\restriction c /  k(G)$
is isomorphic to the trivial Boolean algebra
$\bp{0,1}$, and the latter
%is isomorphic in $\bool{B}\restriction c$, which 
is trivially in $\Gamma$.}
hence $\bool{B}\restriction c\leq_\Gamma \bool{B}\restriction b$, i.e., $i_{\bool{B}}$ is order preserving.

With some more effort we can also check that $i_{\bool{B}}$ preserves suprema: i.e., if
$a=\bigvee_{\bool{B}}\bp{a_i:i\in I}$, then $\bool{B}\restriction a$ is the supremum of the set
$\bp{\bool{B}\restriction a_i :i\in I}$ in $(\Gamma,\leq_\Gamma)$.

The critical issue is the preservation of the incompatibility relation.
Consider any homogeneous forcing $\bool{B}\in\Gamma$; its homogeneity grants that
 $\bool{B}\restriction s$ and $\bool{B}\restriction t$ are isomorphic for all
$s$, $t\in \bool{B}$, hence the map $i_{\bool{B}}$ does not preserves the incompatibility relation.
This is an unavoidable and critical issue of the map $i_{\bool{B}}$
we must address; for example the forcing $\Coll(\omega_1, \mu)$ (given by injective functions with
domain a countable ordinal and range contained in $\mu$) is homogeneous
and belong to the class $\Gamma$ of proper forcings for all  $\mu\geq\omega_1$.
\end{quote}

To overcome this issue, we introduce the notion of $\Gamma$--rigidity: a forcing $\bool{B}\in \Gamma$
is $\Gamma$--rigid if the map $i_{\bool{B}}$ defined above is incompatibility--preserving.

Now assume $(\Gamma, \leq_\Gamma)$ has a dense set of $\Gamma$--rigid elements.
We can then embed any $\bool{B}\in\Gamma$ into $\Gamma$
as follows: We first find a $\Gamma$--rigid $\bool{C}\leq_\Gamma\bool{B}$ in $V_\delta$
and $i:\bool{B}\to\bool{C}$, a complete injective homomorphism with a generic quotient in $\Gamma$. Then
the map $i_{\bool{C}}\circ i$ is a complete embedding of $\bool{B}$ into $\Gamma\restriction\bool{C}$.

It can also be shown that:

\emph{
If $\Gamma$ is well behaved
and $\bool{B}\in\Gamma$, whenever $H$ is $V$-generic for $\bool{B}$, we have that 
$\Gamma^{V[H]}$ is forcing equivalent to 
the generic quotient of $\Gamma^V$ by $H$.}

In any case, the first key observation is the following:

\begin{quote} 
In order to achieve task~(\ref{keytask1.2}) for $\Gamma$, a key requirement is for 
$(\Gamma,\leq_\Gamma)$ to have
a dense class of $\Gamma$--rigid elements.
\end{quote}

Now let us address task~(\ref{keytask1.1}).
Assume $\Gamma$ has an iteration theory guaranteeing that all iterations of its posets 
constructed according
to some (reasonable) given rule have a lower bound in $\Gamma$ 
(this is the case, for example, for proper and semiproper forcings). 
Then we can also show that $\Gamma\cap V_\delta$
preserves the regularity of $\delta$. Roughly speaking, the proof goes as follows:
Assume $\dot{f}:\alpha\to\delta$ is a $\Gamma\cap V_\delta$--name for a function.
Given $\bool{B}\in\Gamma$, define $\bp{\bool{B}_{\xi}:\xi\leq\alpha}$, letting $\bool{B}_0$
be some $\Gamma$--rigid element refining $\bool{B}$,
and for every $\xi$ letting $\bool{B}_{\xi+1}\leq_\Gamma\bool{B}_\xi$ be such that:
\begin{itemize} 
\item
there is an \emph{injective} map $i_\xi:\bool{B}_\xi\to\bool{B}_{\xi+1}$ witnessing 
$\bool{B}_{\xi+1}\leq_\Gamma\bool{B}_\xi$,
\item
there is
a maximal antichain $A_\xi$ in $\bool{B}_{\xi+1}$ such that $\bool{B}_{\xi+1}\restriction a$ decides
in $\Gamma\cap V_\delta$ the value of $\dot{f}(\xi)$ for all $a\in A_\xi$.
\end{itemize}
At limit stages $\beta\leq\alpha$, take $\bool{B}'$ as some limit in $\Gamma$ of the iteration
$\bp{\bool{B}_{\xi}:\xi<\beta}$ and let $\bool{B}_\xi\leq_\Gamma \bool{B}'$ be $\Gamma$--rigid.
One can check that 
\begin{itemize}
\item The recursive construction can be successfully carried out: 
\begin{itemize}
\item
At successor stages, we use the fact that the
lottery sums of $\Gamma$--rigid forcings is itself $\Gamma$--rigid. 
We let $\bool{B}_{\xi+1}$ be the lottery
sum of $\Gamma$--rigid conditions 
$\bool{C}_a\leq\bool{B}_\xi$ 
%(with the order relation witnessed by a $\Gamma$--correct 
%$i_a:\bool{B}_\xi\to \bool{C}_a$ with 
%$a=\coker(i_a)=\bigvee_{\bool{B}_\xi}\bp{b:i_a(b)=0_{\bool{C}_a}}$)
deciding the value of $\dot{f}(\xi)$. One needs
at most $|\bool{B}_\xi|$ many such conditions. 
\item For the limit stages one uses 
the iteration theorem for forcings
in $\Gamma$.
\end{itemize}
\item
It can then be shown that $\bool{B}_\alpha\leq_\Gamma\bool{B}$ forces in $\Gamma\cap V_\delta$
that $\dot{f}$ has its range bounded by some $\gamma<\delta$.
\end{itemize}
This takes care also of task~(\ref{keytask1.1}). 

So far, in order to achieve tasks~(\ref{keytask1.1}) and ~(\ref{keytask1.2}) 
we have isolated the following requests on $\Gamma$:
\begin{itemize}
\item $\Gamma$ is closed under preimages by complete injective homomorphisms.
%isomorphisms, under restrictions, under two--step iterations.
\item The set of $\Gamma$--rigid elements is dense in $(\Gamma,\leq_\Gamma)$.
\item There is an iteration theorem granting all iterations of members from $\Gamma$ 
carried out according to some rule have a lower bound in $\Gamma$.
\item $\Gamma$ is closed under lottery sums.
\end{itemize}
We are left with a strategy to fulfill task~(\ref{keytask1.3}) on page \pageref{page:tasks}. 
For this, we define the category forcing axiom $\mathsf{CFA}(\Gamma)$ as, roughly, 
the statement that the class of $\Gamma$--rigid forcings
 which induce generic ultrapowers with strong closure properties is dense in 
 $(\Gamma,\leq_\Gamma)$. The precise formulation, which we will repeat in 
 Subsection \ref{subsec:mainresultsM}, was already given in 
 Definition \ref{definition-of-cfa(gamma)}.

 %We are not giving this formulation, though, as it involves the notion of strongly pre-saturated 
 %tower, it would take too long to introduce it properly, roughly a strongly pre-saturated tower for 
 %$\Gamma$
 %is a $\lambda_\Gamma^+$-presaturated tower which is also $\Gamma$--rigid and
 %satisfies some extra technical property. 

In order to handle this axiom, we resort to (variations of) Foreman's duality theorem~\cite{FORDUATHM}, 
which essentially amounts to the following:

\begin{quote}
Given a suitably large cardinal $\delta$
and a nicely defined
forcing $P_\delta$ which preserves the regularity of $\delta$, assume $G$ is $V$--generic 
for $P_\delta$ and $j:V\to M$ is elementary with critical point $\delta$ and with $M$ sufficiently closed.
Then $P_{j(\delta)}/_G$ is in $V[G]$ a forcing with the following property:
Whenever $K$ is $V[G]$--generic for $P_{j(\delta)}/_G$, the map
$\bar{j}:V[G]\to M[K]$ given by $\bar{j}(\tau_G)=j(\tau)_{G\ast K}$
is a generic ultrapower embedding. Moreover $\bar{j}$ retains in $V[G\ast K]$ most 
of the closure properties which $j$ has in $V$.
\end{quote}
With further work and elaborating on Foreman's duality theorem
one can prove that: 
\begin{quote}
Whenever $\Gamma$ is $\lambda$--suitable, $\kappa$ is a large enough cardinal in $V$, and 
$H$ is $V$--generic for $\Gamma\cap V_\kappa$, $\kappa=(\lambda^+)^{V[H]}$ and
for all sufficiently large cardinals $\delta>\kappa$ of $V[H]$,
$(\Gamma\cap V_\delta)^{V[H]}$ satisfies in $V[H]$ all requirements set forth 
in~(\ref{keytask1.1}),~(\ref{keytask1.2}) and~(\ref{keytask1.3}). 
\end{quote} 
% One rough formulation of the axiom $\mathsf{CFA}(\Gamma)$ states that $V$ is a model of set theory
 %such that, for all large enough cardinals $\delta$, (\ref{keytask1.1}), (\ref{keytask1.2}) and (\ref{keytask1.3}) hold for
% $\Gamma\cap V_\delta$ (s.\ Definition \ref{def-cfagamma}). Another more elegant formulation, which we already referred to in the introduction, states that the class of $\Gamma$--rigid forcings
% which induce generic ultrapowers with strong closure properties is dense in $(\Gamma,\leq_\Gamma)$. We are not giving this formulation, though, as it involves the notion of strongly pre-saturated tower, it would take too long to introduce it properly, roughly a strongly pre-saturated tower for $\Gamma$
% is a $\lambda_\Gamma^+$-presaturated tower which is also $\Gamma$--rigid and
% satisfies some extra technical property. 
 
 The consistency proof of  of $\mathsf{CFA}(\Gamma)$ proceeds by showing that if $\delta$ is a sufficiently large cardinal and $H$ is generic for $\Gamma\cap V_\delta$, then $V[H]$ is a model of 
 $\mathsf{CFA}(\Gamma)$.
 
 Finally, further work (which follows from Woodin's results on stationary tower forcings) shows
 also that:
 \begin{quote}
Whenever $\Gamma$ is $\lambda$--suitable, $\mathsf{CFA}(\Gamma)$  implies that 
$\mathsf{FA}_\lambda(\bool{B})$ holds for all $\bool{B}\in\Gamma$.
  \end{quote}

\subsection{$\lambda$--suitable class forcings}\label{subsec:suitclassforc}

Let us now come to the rigorous definitions.
This section introduces the key properties of a class of forcings 
$\Gamma$ we are  interested in.
All the classes $\Gamma$ we consider from now on are defined as the extension of a formula 
$\phi_\Gamma(x,a_\Gamma)$ in the language of set theory enriched with a constant symbol for a set $a_\Gamma$;
$a_\Gamma$ is a set parameter, all the quantifiers in $\phi_\Gamma$ range over sets, and
the free variable $x$ ranges over sets.

\begin{notation}
Given a class $\Gamma$ defined by a formula with 
quantifiers and parameters ranging over sets, $\phi_\Gamma(x,a_\Gamma)$ and $a_\Gamma$ will always denote 
the formula and the parameter used to define it.
\end{notation}

\begin{remark}
Our official definition of a class forcing $\Gamma$ assumes that 
$\Gamma$ consists of complete Boolean algebras.
This is the case since most of our definitions and calculations on such class forcings $\Gamma$ 
are much easier
 to state and carry out if $\Gamma$ consists solely of complete Boolean algebras.
On the other hand, in many cases there are posets $Q$ of interest (for example the posets
$(\Gamma\cap V_\delta,\leq_\Gamma\cap V_\delta)$ for $\delta$ a large enough cardinal), 
which are not even separative, for which it is important to establish that their Boolean completion 
$\RO(Q)$ is in $\Gamma$. As is often the case in forcing arguments, we have a clear grasp of what 
$Q$ is and how its combinatorial properties work, while this is much less transparent when we pass to 
$\RO(Q)$. It will be convenient in these situations to assume 
$Q\in\Gamma$ even if 
this actually holds just  
for $\RO(Q)$. So we feel free in many cases to assume 
that the extension of a class forcing
$\Gamma$ consists of all the posets $Q$ whose Boolean completions 
satisfy the defining property of $\Gamma$.
If we feel that this can generate misunderstandings, we will be explicitly 
more careful in those situations.
\end{remark}

\begin{definition}
Let $\Gamma$ 
be a definable class of forcing notions.

Let $\bool{B}$, $\bool{C}$ be complete Boolean algebras.
\begin{itemize}
\item
A complete homomorphism $i:\bool{B}\to\bool{C}$ is $\Gamma$--correct 
if\footnote{Notice that a priori we do not require either 
$\bool{B}$ or $\bool{C}$ to be in $\Gamma$, 
even if in what follows we shall mostly be interested in the case
in which this is the case for both of them.}
\[
\Qp{\bool{C}/i[\dot{G}_{\bool{B}}]\in\Gamma}_{\bool{B}}=
\Qp{\phi_\Gamma(\bool{C}/i[\dot{G}_{\bool{B}}],\check{a}_\Gamma)}_{\bool{B}}\geq_{\bool{B}}\coker(i),
\]
where $\coker(i)=\neg\bigvee\bp{a\in\bool{B}: i(a)=0_{\bool{C}}}$ and 
$\dot{G}_{\bool{B}}=\bp{\ap{\check{b},b}:b\in\bool{B}}$ is the canonical $\bool{B}$--name for the $V$--generic filter.
\item
$\bool{C}\leq_\Gamma\bool{B}$ if there is a
$\Gamma$--correct 
\[
i:\bool{B}\to\bool{C}.
\]

\item
$\bool{C}\leq^*_\Gamma\bool{B}$ if there is an \emph{injective}
$\Gamma$--correct complete homomorphism
\[
i:\bool{B}\to\bool{C}.
\]
\item
Assume further that $k:\bool{B}\to\bool{C}$ is $\Gamma$--correct and $\bool{B}$, $\bool{C}\in\Gamma$.

$k$ $\Gamma$--freezes $\bool{B}$ if for all \emph{$\Gamma$--correct} 
$i_0$, $i_1:\bool{C}\to\bool{D}$ we have that $i_0\circ k=i_1\circ k$.

\item
$\bool{B}$ is $\Gamma$--rigid if the identity map $\id:\bool{B}\to\bool{B}$ $\Gamma$--freezes $\bool{B}$.

\item
 Assume $G$ is $V$--generic for $\bool{C}$.
 Let $H\in V[G]$ be $V$--generic for $\bool{B}$ and $\breve{H}$ its dual prime ideal on $\bool{B}$.
 Let $\dot{H}\in V^{\bool{C}}$ be such that $\dot{H}_G=H$ and 
 $q=\Qp{\dot{H}\text{ is $V$--generic for }\bool{B}}\in G$.
  
 For $r\leq_{\bool{C}}q$ define:
 \begin{align*}
 i_{r,\dot{H}}:& \bool{B}\to\bool{C}\restriction r\\
 & b\mapsto\ \Qp{\check{b}\in\dot{H}}_{\bool{C}}\wedge r.
 \end{align*}
 Then $i_{r,\dot{H}}$ is a complete homomorphism such that $H=i_{r,\dot{H}}^{-1}[G]$. 
% by 
% Lemma~\ref{prop:embfromembnames3}. 
 %\end{aMnote}
 $H$ is $\Gamma$--correct for $\bool{B}$ in $V[G]$ if for some $r\in G$, $r\leq_{\bool{C}} q$,
 letting $J$ be the ideal $\downarrow (i_{r,\dot{H}}[\breve{H}])$, we have that
 \[
V[H]\models \bool{C}/_J\in\Gamma^{V[H]}.
\]

\end{itemize}
\end{definition}

\begin{definition}
Let $\Gamma\subseteq V$ be a definable class of posets.
\begin{itemize}
%\item $\Gamma$ is \emph{closed under isomorphisms} if for every $\bool{B}\in\Gamma$, every Boolean algebra isomorphic to $\bool{B}$ is in $\Gamma$.
%\item $\Gamma$ is \emph{closed under restrictions} if for every $\bool{B}\in\Gamma$ and every $b\in \bool{B}$, we have that $\bool{B}\restriction b\in\Gamma$.
\item
$\Gamma$ is \emph{closed under preimages by complete injective homomorphisms} if given any $\bool{C}\in\Gamma$ and any complete Boolean algebra $\bool{B}$,
if $i:\bool{B}\to\bool{C}$ is a complete injective homomorphism, then we also have that $\bool{B}\in\Gamma$.\footnote{Note that if $\Gamma$ is closed under preimages by complete injective homomorphism, then is is closed under complete subalgebras in the obvious sense.}
\item
$\Gamma$ is \emph{closed under two--step iterations} if for all $\bool{B}\in\Gamma$ and 
all $\dot{\bool{C}}\in V^\bool{B}$ such that $\Qp{\dot{\bool{C}}\in\Gamma}_\bool{B}=1_\bool{B}$ we have that
$\bool{B}*\dot{\bool{C}}\in \Gamma$.
\item
$\Gamma$ is \emph{closed under lottery sums} if every set $A\subset\Gamma$ has 
as exact upper bound
$\bigvee_\Gamma A$ in $\leq_\Gamma$ ($\bigvee_{\Gamma}A$ is the lottery sum of the posets in $A$, 
equivalently -- given that $A$ consists of complete Boolean algebras --
the product of the Boolean algebras in $A$).
\item
$\Gamma$ 
has the $\Gamma$--\emph{freezeability property}
if for every $\bool{B}\in \Gamma$ there
is $k:\bool{B}\to\bool{C}$  
$\Gamma$--freezing $\bool{B}$.
\item
$\Gamma$ is \emph{weakly $\lambda$--iterable} if 
for each ordinal 
$\alpha$ Player $II$ has a winning strategy in the game
$\mathcal{G}_\alpha(\Gamma)$ of length $\alpha+1$, between players $I$ and $II$, defined as follows: 
\begin{itemize}
\item
at successor stages $\alpha$, players $I$ and $II$  play $\Gamma$--correct \emph{injective}
homomorphism
$i_{\alpha,\alpha+1}:\bool{B}_\alpha\to\bool{B}_{\alpha+1}$;
\item
Player $I$ plays at odd stages, player $II$ at even stages ($0$ and all limit ordinals are even);
\item
at stage $0$, $II$ plays a $\Gamma$--correct injective embedding 
$i_{0,1}:\2\to\bool{B}_{1}$ (i.e. a $\bool{B}_1\in\Gamma$);
\item
at limit stages $\sigma$, $II$ must play\footnote{$\varinjlim\{\bool{B}_\alpha:\alpha<\sigma\}$ is the direct 
limit
%\footnote{we refer the reader to the appendix~\ref{appendix} for the definition of direct and inverse limit of an 
%iteration system, and for one of the many (possibly inequivalent) definition of RCS-iterations.}
of the iteration system given by the maps $i_{\gamma\beta}:\bool{B}_\gamma\to\bool{B}_\beta$ which are built along the
play of $\mathcal{G}_\alpha(\Gamma)$.} $\varinjlim\{\bool{B}_\alpha:\alpha<\sigma\}$
if $\cof(\sigma)=\lambda$ or if $\sigma$ is regular and all Boolean algebras in $\{\bool{B}_\alpha:\alpha<\sigma\}$ 
have size less than $\sigma$;
\item 
$II$ wins $\mathcal{G}_\alpha(\Gamma)$ if she can play at all stages up and including $\alpha$.
\end{itemize}
\end{itemize}
\end{definition}

There is a tight interaction between the properties of a class of forcings $\Gamma$ and the theory 
$T\supseteq\MK$
in which we analyze this class. 
For example, in our analysis of $\Gamma$ we are naturally led to work with
theories $T$ which extend $\mathsf{MK}$ but which are not preserved by \emph{all} set sized forcings.
For example this occurs for $T=\mathsf{MK}+\text{`}\omega_1\text{ is a regular cardinal'}$, which is not preserved
by $\Coll(\omega, \omega_1)$, but is preserved by all stationary set preserving forcings.

The following definition outlines the key correlations between a theory $T\supseteq\mathsf{MK}$ and 
a class of forcings $\Gamma$ we want to bring forward, and allows us to prove, within $T$, that $\Gamma$ is a well
behaved class forcing. The reader may keep in mind while reading the definition below
that semiproperness and properness will be the
simplest examples of $\omega_1$--suitable classes $\Gamma$, and that for these classes $\Gamma$ a useful 
$\Gamma$--canonical
theory is any enlargement of $\mathsf{MK}+\text{`}\omega_1\emph{ is a regular cardinal'}$ by large cardinal
axioms.

\begin{definition}
We say that $P(x)$ is an \emph{absolutely $\Sigma_2$ property} for $T\supseteq\MK$ 
if there is a $\Sigma_0$--formula $\phi(x,y,z)$ 
with quantifier ranging just over sets such that in any model $V$ of $T$,
%$P(A)$ holds in $V$ if and only if 
%$V\models \exists y\forall z\phi(A,y,z)$ for all $A\in V$. Moreover
for all inaccessible $\delta$ and $A\in V_\delta$ we have that
\[
V_\delta\models \exists y\forall z\phi(A,y,z)\text{ if and only if $P(A)$ holds in $V$.}%V\models \exists y\forall z\phi(A,y,z)
\]
\end{definition}

\begin{fact}
The statements
\emph{`$\bool{B}$ is a complete Boolean algebra'} and
\emph{`$i:\bool{B}\to\bool{C}$ is a complete homomorphism'} are absolutely 
$\Sigma_2$ in the relevant parameters and also $\Pi_2$ for the theory $\MK$.

Assume $T\supseteq \mathsf{MK}$ proves that $\phi_\Gamma(x,a_\Gamma)$ is an absolutely 
$\Sigma_2$ property 
in the parameter $a_\Gamma$.
Then $T$ proves that for all inaccessible cardinals $\delta$, 
%and $i:\bool{B}\to\bool{C}$ in $V_\delta$
\[
V_\delta\models \bool{B}\in \Gamma^{V_\delta}\text{ if and only if } V\models \bool{B}\in\Gamma,
\]
and
\[
V_\delta\models \bool{C}\leq_\Gamma\bool{B}\text{ if and only if } V\models \bool{C}\leq_\Gamma\bool{B}
\]
\end{fact}
This is not hard to verify; for a proof see~\cite{VIAAUDSTEBOOK}. Many classes of forcings are defined by an absolutely 
$\Sigma_2$ property which is also $\Pi_2$ for $\MK$; 
among others, the classes of forcings which are proper, semiproper, stationary set preserving, etc.

\begin{definition}[$\lambda$--suitable class forcings]\label{def:kappasuit}
Consider the language of set theory enriched with two constant symbols $a_\Gamma$ and $\lambda$.
Let $\Gamma$ be a definable class of forcings by means of the formula
$\phi_\Gamma(x,a_\Gamma)$ in parameter $a_\Gamma\in H_{\lambda^+}$.

\begin{itemize}
\item
$T\supseteq\MK$
 is \emph{$\lambda$--canonical} if it extends $\mathsf{MK}$ by a finite list of axioms
expressible without quantifiers ranging over classes (but with no bound on the number of quantifiers
ranging over sets), among which is the axiom \emph{`$\lambda$ is a regular cardinal'}.
\item
$\Gamma$ is \emph{$\lambda$--suitable} for a $\lambda$--canonical theory $T$ if:
\begin{enumerate} 
\item
$T$ proves that $\phi_\Gamma(x,a_\Gamma)$ is equivalent to an absolutely $\Sigma_2$ property in the parameter 
$a_\Gamma\in H_{\lambda^+}$. 
%and to a $\Pi_2$-property in the parameter $a_\Gamma\subseteq\kappa$.
\item\label{def:kappasuit-1}
$T$ proves that all forcing notions in $\Gamma$ preserve all axioms of $T$.
\item\label{def:kappasuit-2}
$T$ proves that $\Gamma$ is closed under preimages by complete injective homomorphisms, two--steps iterations and lottery sums.
\item\label{def:kappasuit-3}
$T$ proves that $\Gamma$ contains all ${<}\lambda$--closed posets.
\item\label{def:kappasuit-5}
$T$ proves that $\Gamma$ 
is weakly $\lambda$--iterable.
\item\label{def:kappasuit-4}
$T$ proves that $\Gamma$ has the $\Gamma$--freezeability property,
\end{enumerate}
\end{itemize}
$T$ is $\Gamma$--canonical if it is $\lambda$--canonical and $\Gamma$ is $\lambda$--suitable for $T$.
\end{definition}

\begin{remark} \label{rmk:canth}
We observe the following:
\begin{itemize}
\item
Theories $T$ of the form $\mathsf{MK}+$ the statement that there a exists a proper class of large cardinals
of a certain kind (supercompact, Woodin, huge, etc) are $\Omega$--canonical, where
$\Omega$ is the class of all set--forcings.
\item
A key feature of a $\lambda$--canonical theory $T$ we need to
exploit is that once it holds in $V_{\delta+1}$ for some inaccessible $\delta$
it holds also in $W_{\delta+1}$ for any $W\supseteq V$ such that:
\begin{itemize}
\item
$\delta$ remains inaccessible in $W$,
\item 
$W_{\delta+1}$ is a model of
$\mathsf{MK}$,
\item
$W_\delta=V_\delta$.
\end{itemize}
This is the case since the extra axioms in $T\setminus\mathsf{MK}$ are defined by properties
which do not take into consideration (in order to evaluate their truth) the new proper classes
appearing in $W_{\delta+1}\setminus V_{\delta+1}$. 
\end{itemize}
\end{remark}

\begin{notation}
Given a category forcing $(\Gamma,\leq_\Gamma)$ with $\Gamma$ a definable class of complete Boolean algebras and 
$\leq_\Gamma$ the order induced on $\Gamma$
by the $\Gamma$--correct homomorphisms between elements of $\Gamma$,
we denote the incompatibility relation with respect to $\leq_\Gamma$ by $\bot_\Gamma$, 
and the subclass of $\Gamma$
given by its
$\Gamma$--rigid elements by $\bool{Rig}^\Gamma$.
\end{notation}

Notice the following.

\begin{fact}\label{fac:keyrmkfreeze}
Assume $\Gamma\,\subseteq\,\Delta$ are definable classes of forcings. 
Then $\leq_\Gamma\subseteq\leq_\Delta$
and $\bot_\Delta\subseteq\bot_\Gamma$.
Hence, if $i:\bool{B}\to\bool{C}$ is $\Gamma$--correct and $\Delta$--freezes $\bool{B}$, we also have that
$i$ is $\Delta$--correct and $\Gamma$--freezes $\bool{B}$. 
\end{fact}
This fact will be repeatedly used in the second part of this paper to show that various classes of forcings $\Delta$ 
have the
$\Delta$--freezeability property by providing for some $\Gamma\subseteq\Delta$ 
an $i:\bool{B}\to\bool{C}$ which is $\Gamma$--correct and $\Delta$--freezes $\bool{B}$. As we will see, all our freezeability results proceed  by proving the existence, given $\bool{B}\in\Gamma$, of a $\bool{B}$--name $\dot{\bool{Q}}$ for a forcing in $\Gamma$ such that $\bool{C}=\bool{B}\ast\dot{\bool{Q}}$ codes the generic filter $\dot{G_{\bool{B}}}$ for $\bool{B}$ as a subset $A_{\dot G_{\bool{B}}}$ of $\omega_1$ in some absolute manner, in the sense that in every outer model $M$ of $V^{\bool{C}}$ preserving stationary subsets of $\omega_1$, $A_{\dot G_{\bool{B}}}$ is the unique object of $\omega_1$ satisfying some given property.  It will thus follow that $\bool{C}$ will $\SSP$--freeze $\bool{B}$, which will be an instance of the above since we will always have $\SSP\supseteq\Gamma$ for the $\Gamma$ of interest to us.

\subsection{Main results}\label{subsec:mainresultsM} Given a class $\Gamma$ of forcing notions and a cardinal $\delta$, we will write $\bool{U}^\Gamma_\delta$ to denote $\Gamma\cap V_\delta$ with the inherited order $\leq_\Gamma\cap V_\delta$.
We can now list the main theorems regarding category forcings.

Our first theorem in this subsection takes care of task~(\ref{keytask1.1}).% and~(\ref{keytask2}).

\begin{theorem}
Assume $\Gamma$ is $\lambda$--suitable for a $\lambda$--canonical $T\supseteq\MK$.
Let $\delta$ be inaccessible and such that $V_{\delta+1}\models T$.

Then:
\begin{itemize}
\item
$\bool{U}^\Gamma_\delta$
is a forcing notion in $\Gamma$,
\item
$\bool{U}^\Gamma_\delta$ preserves the regularity of $\delta$ and makes it
the successor of $\lambda$.
\end{itemize}
\end{theorem}

We deal next with  task~(\ref{keytask1.2}).
%However this is not needed to understand the other results of this paper.

\begin{notation}
Assume $\Gamma$ is $\lambda$--suitable for a $\lambda$--canonical theory $T$.
For each $\bool{R}\in\bool{Rig}^\Gamma$ let
\[
k_{\bool{R}}:\bool{R}\to\Gamma\restriction\bool{R}
\]
be given by $r\mapsto\bool{R}\restriction r$.
Then $k_{\bool{R}}$ is an order and incompatibility preserving embedding of $\bool{R}$ 
in the class forcing
$\Gamma\restriction\bool{R}$ which maps maximal antichains to maximal antichains.
Moreover, for every $\bool{B}\geq_\Gamma\bool{C}$ with $\bool{B}\in \bool{Rig}^\Gamma$, let
\[
i_{\bool{B},\bool{C}}:\bool{B}\to\bool{C} 
\]
denote the unique $\Gamma$--correct homomorphism
from $\bool{B}$ into $\bool{C}$. 
\end{notation}

\begin{definition}\label{def:notcong}
Assume $\Gamma$ is $\lambda$--suitable for a $\lambda$--canonical theory $T$.
Given $\bool{B}_0\in \Gamma$, fix $k_0:\bool{B}_0\to \bool{B}$
$\Gamma$--freezing $\bool{B}_0$ and such that $\bool{B}\in \bool{Rig}^\Gamma$.
Let $i_{\bool{C}}=i_{\bool{B},\bool{C}}\circ k_0$ and
\begin{align*}
k=k_{\bool{B}}\circ k_0:& \bool{B}_0\to \Gamma\restriction\bool{B}\\
& b\mapsto\bool{B}\restriction k_0(b)
\end{align*}
Given $G$, a $V$--generic filter for $\bool{B}_0$, define in $V[G]$ the class quotient forcing
\[
P_{\bool{B}_0}=((\bool{Rig}^\Gamma\restriction\bool{B})^V/_{k[G]},\leq_\Gamma/_{k[G]})
\] 
as follows:
\[
\bool{C}\in P_{\bool{B}_0}
\]
if and only if $\bool{C}\in(\bool{Rig}^\Gamma\restriction\bool{B})^V$ and
letting $J$ be the dual ideal of $G$ we have that
$1_{\bool{C}}\not\in i_{\bool{C}}[J]$ (or equivalently if and only if $\coker(i_{\bool{C}})\in G$).

We let 
\[
\bool{C}\leq_{\Gamma}/_{k[G]}\bool{R}
\] 
if $\bool{C}\leq_\Gamma\bool{R}$ holds in $V$.
\end{definition}

The following theorem takes care of task~(\ref{keytask1.2}).

\begin{theorem}\label{thm:quo-univ}
Assume $\Gamma$ is $\lambda$--suitable for a $\lambda$--canonical theory $T$,
$\bool{B}_0\in \Gamma$, and let $k_0:\bool{B}_0\to\bool{B}$ be a $\Gamma$--freezing
homomorphism for $\bool{B}_0$ with $\bool{B}\in\bool{Rig}^\Gamma$.
Set $k=k_{\bool{B}}\circ k_0$ and $i_{\bool{C}}=i_{\bool{B},\bool{C}}\circ k_0$ for 
all $\bool{C}\leq_\Gamma \bool{B}$ in $\Gamma$.
Let $G$ be $V$--generic for $\bool{B}_0$.

Then:
\begin{enumerate} 
\item \label{thm:equivGamma1}
The class forcing
\[
P_{\bool{B}_0}=((\bool{Rig}^\Gamma\restriction\bool{B})^V/_{k[G]},\leq_\Gamma/_{k[G]})
\] 
is in $V[G]$ forcing equivalent to 
the class forcing
\[
Q_{\bool{B}}=(\Gamma\restriction(\bool{B}/_{k_0[G]}))^{V[G]}
\]
via the map
\begin{align*}
i^*:&P_{\bool{B}} \to Q_{\bool{B}}\\
& \bool{C} \mapsto \bool{C}/_{i_{\bool{C}}[G]}.
\end{align*}
\item \label{thm:equivGamma2}
Moreover, %assume $T$ is strongly $\Gamma$-canonical.
let $\delta>|\bool{B}|$ be inaccessible and such that
$V_{\delta+1}$ models $T$.
Then:
\begin{enumerate} 
\item \label{thm:equivGamma2.1}
$(\bool{U}^\Gamma_\delta\restriction(\bool{B}/_{k_0[G]}))^{V[G]}$ is forcing equivalent in $V[G]$ 
to  $(\bool{U}^\Gamma_\delta\restriction\bool{B})^V/_{k[G]}$ via the same map.
\item \label{thm:equivGamma2.2}
$V$ models that $k_\bool{B}:\bool{B}\to \bool{U}^\Gamma_\delta\restriction\bool{B}$ is
$\Gamma$--correct.
\end{enumerate}
\end{enumerate}
\end{theorem}

We can finally handle task~(\ref{keytask1.3}). Recall our definition of the axiom $\CFA(\Gamma)$ (Definition \ref{definition-of-cfa(gamma)}).
 
 \begin{definition} Given a $\lambda$--suitable class $\Gamma$ of forcing notions, \emph{the category forcing axiom for $\Gamma$}, $\CFA(\Gamma)$, is the following statement: The class of 
 $\Gamma$--super rigid pre-saturated $\lambda^+$--towers $\mtcl T\in\Gamma$ is dense in the category $(\Gamma, \to^\Gamma)$ given by $\Gamma$. \end{definition} 
 
% As mentioned in the introduction, we choose not to define the define the notion of strongly pre-saturated. Instead, 
 
 Let us define for the sake of completeness the notion of 
 $\Gamma$--super rigid tower,  
 %(also called elsewhere in this paper strongly presaturated  $\lambda^+$-towers), 
 even if for the remainder of this paper we will not have any use of this notion since we will resort to a formulation of $\CFA(\Gamma)$
 which does not mention such towers:
 
 \begin{definition}
 Let $T\supseteq\MK$ be  $\lambda$-canonical theory for a 
 $\lambda$--suitable class of forcings $\Gamma$.

$\mtcl T$ is a $\Gamma$--super rigid tower if:
\begin{itemize}
\item $\mtcl T\in \Gamma$ is 
$\lambda^+$--presaturated tower of height $\delta$,
\item there is a dense embedding $i:\mtcl D\to \mtcl T$ with
$\mtcl D$ a suborder of\footnote{The simplest case is when $\mtcl D$ is 
$\bool{Rig}^\Gamma\cap V_\delta$, yielding that $\mtcl T$ is forcing 
equivalent to $\bool{U}^\Gamma_\delta$.
In fact, the simplest consistency proof of $\CFA(\Gamma)$ produces pre-densely many
$\Gamma$--super rigid towers of this form. But it is also consistent 
that there are $\Gamma$--super rigid towers of height $\delta$ which are not forcing equivalent to
$\bool{U}^\Gamma_\delta$.} $\bool{Rig}^\Gamma\cap V_\delta$.
\end{itemize} 
\end{definition}
 
In the presence of a proper class of cspercompact cardinals it is possible to give a different characterization of $\CFA(\Gamma)$, which is the one we will be using.
 % in the presence  of a proper class of supercompact cardinals.
%We are now in a position to give a definition of $\mathsf{CFA}(\Gamma)$ in the presence of a proper class of supercompact cardinals. 

\begin{theorem}\label{cfagamma+supercompact}
Let $\Gamma$ be a $\lambda$--suitable class of forcings for a $\lambda$--canonical theory $T\supseteq\MK^*+$\emph{there are class many supercompact cardinals}.

$\mathsf{CFA}(\Gamma)$ holds if and only if for every $V$ model of $T$, every
supercompact cardinal $\delta$ in $V$
and every $V$--generic filter $H$ for
$\bool{U}^\Gamma_\delta$ there is in $V[H]$ a definable ultrapower embedding
$j:V\to M$ with critical point $\delta$ and such that $M^{{<}\delta}\subseteq M$.
\end{theorem}

%This theorem takes care of task~(\ref{keytask1}).

The consistency proof of $\CFA(\Gamma)$ makes use of the following large cardinal notion. 

%We will next address task~(\ref{keytask1.3}) and the consistency of $\CFA(\Gamma)$.

\begin{definition}
An elementary $j:V\to M$ is $2$--huge if $M\subseteq V$ and $M^{j^2(\delta)}\subseteq M$, where $\delta$ is the critical point of $j$.

$\delta$ is $2$--superhuge if for all $\gamma$ there is a $2$--huge elementary $j:V\to M$ with $j(\delta)>\gamma$ and $\delta$ the critical point of $j$.
\end{definition}

\begin{theorem}\label{consistency-of-CFA}
Let $\Gamma$ be a $\lambda$--suitable class of forcings for a $\lambda$--canonical theory 
$T\supseteq\MK$.
Assume $T+`$\emph{there exists a $2$--superhuge cardinal}' is consistent.
Then so is $\CFA(\Gamma)+T$. In fact, if $\d$ is a $2$--superhuge cardinal, then $\bool{U}^\Gamma_\delta$ forces $\CFA(\Gamma)$. 
\end{theorem}

%This theorem takes care of task~(\ref{keytask3}).

The following can now be proved by an implementation of the analysis in Subsection \ref{subsec:genabsforax} using the other theorems stated so far.

\begin{theorem}\label{finaltheorem-abs}
Let $\Gamma$ be a $\lambda$-suitable class of forcings for a $\lambda$--canonical theory 
$T\supseteq\MK$.
Assume $V$ models $T+$`\emph{there are stationarily many inaccessible cardinals'}+`\emph{There is a proper class of supercompact cardinals}'.

Then:
\begin{itemize}
\item 
$T+\CFA(\Gamma)$ makes the theory of $L(\Ord^\lambda)$ with parameters in $\pow{\lambda}$
invariant with respect to forcings in $\Gamma$ which preserve $\CFA(\Gamma)$.
\item $T+\CFA(\Gamma)$ implies $\FA_\lambda(\bool{B})$ for all $\bool{B}\in\Gamma$.
\end{itemize}
\end{theorem}

It is worth observing that $\CFA(\Gamma)$, for a $\lambda$--suitable class $\Gamma$, is a $\Pi_3$ statement (possibly in some parameter defining $\Gamma$) which decides, modulo forcing in $\Gamma$, a collection of statements of unbounded complexity (namely the entire theory of the Chang model $\mtcl C_\l$).\footnote{Also, note that the existence of a proper class of supercompact cardinals is a $\Pi_4$ sentence.} The same observation of course applies already to Woodin's absoluteness for the $\o$--Chang model with respect to arbitrary set--forcing, where  the relevant theory, viz.\ the existence of unboundedly many Woodin cardinals, is a $\Pi_3$ sentence as already mentioned in the introduction.

\subsection{Freezeability and total rigidity}\label{subsec:freezetotrig}

The hardest technical results in the next section
% which are an original contribution of this paper
 are 
several proofs of the freezeability property for a variety of classes of forcings.
This section aims to make the reader familiar with the key facts regarding rigidity and freezeability,
and details how the combination of the freezeability property and of the weak iterability property 
for $\Gamma$ 
yields the density of the class 
of $\Gamma$--rigid forcings.
All over this section we assume that $\Gamma$ is $\lambda$--suitable for the $\lambda$--canonical theory
 $T$.
 
 %We include these results since they provide a justification for the notion 

\begin{theorem}\label{thm:freztotrigGamma}
Assume $\Gamma$ is $\lambda$--suitable for the $\lambda$--canonical theory
 $T$.
Then $T$ proves that the class of $\Gamma$--rigid
partial orders is dense in $(\Gamma,\leq^*_\Gamma)$.
\end{theorem}

\begin{lemma}\label{lem:eqtrGamma}
The following are equivalent characterizations of 
$\Gamma$--rigidity for an algebra $\bool{B}\in\Gamma$:
\begin{enumerate}
\item\label{lem:eqtrGamma1}
for all $b_0$, $b_1\in \bool{B}$ such that
$b_0\wedge_{\bool{B}}b_1=0_{\bool{B}}$ we have that
$\bool{B}\restriction b_0$ is incompatible with $\bool{B}\restriction b_1$ in $(\Gamma,\leq_\Gamma)$. 
\item\label{lem:eqtrGamma2}
For all $\bool{C}\leq_\Gamma\bool{B}$ and all $H$, $V$--generic filter for 
$\bool{C}$, there is just one $\Gamma$--correct $V$--generic filter $G\in V[H]$ 
 for $\bool{B}$.
\item\label{lem:eqtrGamma3}
For all $\bool{C}\leq_\Gamma\bool{B}$ in $\Gamma$ there is only one $\Gamma$--correct
homomorphism
$i:\bool{B}\to\bool{C}$.
\end{enumerate}
\end{lemma}

\begin{remark}
%$\Gamma$--rigidity entails rigidity by its very definition. 
%Nonetheless, 
It is conceivable that even if $\bool{B}$ is $\Gamma$--rigid, there could be a complete (and non-surjective) homomorphism 
$k:\bool{B}\to \bool{B}\restriction b$  
which is not $\Gamma$--correct. If $H$ is $V$--generic for $\bool{B}$ with $b\in H$,
$k^{-1}[H]=G\in V[H]$ is also $V$--generic for $\bool{B}$.
Hence in $V[H]$ there could be distinct $V$--generic filters for $\bool{B}$ even if $\bool{B}$ is 
$\Gamma$--rigid. This is not in conflict with~\ref{lem:eqtrGamma}(\ref{lem:eqtrGamma2}),
since $G\in V[H]$ would not be $\Gamma$--correct for $\bool{B}$ in $V[H]$. 
\end{remark}  

\begin{proof} 
We prove these equivalences by contraposition as follows:
\begin{description}
\item[~\ref{lem:eqtrGamma2} implies~\ref{lem:eqtrGamma1}] 
Assume~\ref{lem:eqtrGamma1} 
fails as witnessed by $i_j:\bool{B}\restriction b_j\to\bool{Q}$ for $j=0,1$ with
$b_0$ incompatible with $b_1$ in $\bool{B}$.
Pick $H$, a $V$--generic filter for $\bool{Q}$. Then $G_j=i_j^{-1}[H]\in V[H]$ (for $j=0$, $1$) are distinct  and $\Gamma$--correct
$V$--generic filters for $\bool{B}$ in $V[H]$,
since $b_j\in G_j\setminus G_{1-j}$.  

\item[~\ref{lem:eqtrGamma1} implies~\ref{lem:eqtrGamma3}]
Assume~\ref{lem:eqtrGamma3} fails for $\bool{B}$ as witnessed by $i_0\neq i_1:\bool{B}\to\bool{C}$.
Let $b$ be such that $i_0(b)\neq i_1(b)$. 
W.l.o.g.\ we can suppose that $r=i_0(b)\wedge i_1(\neg b)>0_{\bool{C}}$.
Then $j_0:\bool{B}\restriction b\to \bool{C}\restriction r$ and
$j_1:\bool{B}\restriction \neg b\to \bool{C}\restriction r$ given by $j_k(a)=i_k(a)\wedge r$ for $k=0,1$ and
$a$ in the appropriate domain witness that $\bool{B}\restriction \neg b$ and $\bool{B}\restriction b$
are compatible in $(\Gamma,\leq_\Gamma)$, i.e.\ that~\ref{lem:eqtrGamma1} fails.

\item[~\ref{lem:eqtrGamma3} implies~\ref{lem:eqtrGamma2}]
Assume~\ref{lem:eqtrGamma2} fails for $\bool{B}$ as witnessed by some 
$\bool{C}\leq_{\Gamma}\bool{B}$, a
$V$--generic filter $H$ for $\bool{C}$, and  
$\Gamma$--correct $V$-generic filters $G_1\neq G_2\in V[H]$ for $\bool{B}$.
Let $\dot{G}_1,\dot{G}_2\in V^{\bool{C}}$ 
be such that $(\dot{G}_1)_H=G_1\neq (\dot{G}_2)_H=G_2$
are $\Gamma$--correct $V$--generic filters for $\bool{B}$ in
$V[H]$ for both $j=1,2$.
%\begin{equation}\label{eqn:Gamma}
%\[
%\text{\emph{$(\dot{G}_1)_H=G_1\neq (\dot{G}_2)_H=G_2$
%are $\Gamma$-correct $V$-generic filters for $\bool{B}$ in
%$V[H]$ for both $j=1,2$.}}
%\]
%\end{equation}
Find $q\in G$ forcing that $b\in G_1\setminus G_2$ for some fixed $b\in\bool{B}$.
Then %, %by Lemma~\ref{lem:twostepchar}($2 \implies 3$),
for some $r\in H$ refining $q$, we have that both homomorphisms
$i_j=i_{\dot{G}_j,r}:\bool{B}\to\bool{C}$ defined by 
$a\mapsto\Qp{\check{a}\in\dot{G}_j}_{\bool{C}}\wedge r$
are $\Gamma$--correct. However $i_1(b)=r=i_2(\neg b)$, and hence $i_1\neq i_2$
witness that~\ref{lem:eqtrGamma3} fails for $\bool{B}$.
\end{description}
\end{proof}

We can also give the following characterizations of $\Gamma$--freezeability, the proof of which is along the same lines of the proof of the previous lemma.

\begin{lemma}\label{lem:eqfrGamma}
Let $k:\bool{B}\to\bool{Q}$ be a $\Gamma$--correct homomorphism.
The following are equivalent:
\begin{enumerate}
\item\label{lem:eqfrGamma1}
For all $b_0,b_1\in \bool{B}$ such that
$b_0\wedge_{\bool{B}}b_1=0_{\bool{B}}$ we have that
$\bool{Q}\restriction k(b_0)$ is incompatible with $\bool{Q}\restriction k(b_1)$ in 
$(\Gamma,\leq_\Gamma)$. 
\item\label{lem:eqfrGamma2}
For every $\bool{R}\leq_\Gamma\bool{Q}$ and every
$V$--generic filter $H$ for $\bool{R}$, there is just one  
$\Gamma$--correct $V$--generic filter $G\in V[H]$ for $\bool{B}$ such that $G=k^{-1}[K]$ for all $\Gamma$--correct $V$--generic 
filters $K\in V[H]$ for $\bool{Q}$.
\item\label{lem:eqfrGamma3}
For all $\bool{R}\leq_\Gamma\bool{Q}$ in $\Gamma$ 
and $i_0,i_1:\bool{Q}\to\bool{R}$ witnessing that $\bool{R}\leq_\Gamma\bool{Q}$ we have that
$i_0\circ k=i_1\circ k$.
\end{enumerate}
\end{lemma}
%\begin{proof}
%Left to the reader, along the same lines of the proof of the previous lemma.
%\end{proof}

A $\Gamma$--freezeable $\bool{B}\in\Gamma$ can be embedded in $\Gamma\restriction\bool{C}$
for some $k:\bool{B}\to\bool{C}$ $\Gamma$--freezing $\bool{B}$
as follows:
\begin{lemma}\label{lem:freezemb}
Assume 
$\Gamma$ is a class of posets having the $\Gamma$--freezebility property.
Let $k:\bool{B}\to\bool{C}$ be a $\Gamma$--correct freezing homomorphism of
$\bool{B}$ into $\bool{C}$.
Then the map
$k_{\bool{B}}:\bool{B}\to\Gamma\restriction\bool{C}$ 
which maps $b\mapsto \bool{C}\restriction k(b)$
defines a complete embedding\footnote{The lemma does not (as yet) assert that 
$k_{\bool{B}}:\bool{B}\to\bool{U}^\Gamma_\delta$ is $\Gamma$--correct, whenever $\delta$ is large enough.
This is true but the proof will not be given here.} of 
the partial order $(\bool{B}^+,\leq_{\bool{B}})$ into
$(\Gamma\restriction\bool{C},\leq_\Gamma)$.
\end{lemma}
\begin{proof}
Left to the reader. It is immediate to check that $k_{\bool{B}}$ preserve predense sets and the 
$\leq_\bool{B}$--order relation.
The $\Gamma$--freezeability property of $k$ is designed exactly in order to get that $k_{\bool{B}}$ preserves 
also the incompatibility relation on $\bool{B}$.
\end{proof}

\begin{lemma}\label{lem:freezelemGamma}
 Assume 
 \[
 \{i_{\alpha\beta}:\bool{B}_\alpha\to\bool{B}_\beta:\alpha<\beta\leq\delta\}
 \]
 is a complete iteration system such that
 for each $\alpha$ there is $\beta>\alpha$ such that
 \begin{itemize}
 \item
 $i_{\alpha,\beta}$ $\Gamma$--freezes 
 $\bool{B}_\alpha$. 
 \item
 $\bool{B}_\delta$ is the direct limit  of the iteration system and is in $\Gamma$.
\end{itemize}
 Then
 $\bool{B}_\delta$ is $\Gamma$--rigid. 
 \end{lemma}
 \begin{proof}
 Assume the lemma fails.
 Then there are  incompatible threads $f_0$, $f_1$ in $\bool{B}_\delta$ such that
 $\bool{B}_\delta\restriction f_0$ is compatible with  
 $\bool{B}_\delta\restriction f_1$ in $(\Gamma,\leq_\Gamma)$.
 Now $\bool{B}_\delta\in \Gamma$ is a direct limit, 
hence
 $f_0$, $f_1$ have support in some  
 $\alpha<\delta$. Thus 
 $f_0(\beta)$, $f_1(\beta)$ are incompatible in 
$\bool{B}_\beta$ for all $\alpha<\beta<\delta$.
Now, for eventually all $\beta>\alpha$, $\bool{B}_\beta$ $\Gamma$--freezes 
$\bool{B}_\alpha$
as witnessed by $i_{\alpha,\beta}$.
In particular, since $f_i=(i_{\alpha,\delta}\circ f_i)(\alpha)$
for $i=0,1$, we get that $\bool{B}_\delta\restriction f_0$ cannot be compatible with  
 $\bool{B}_\delta\restriction f_1$ in $(\Gamma,\leq_\Gamma)$, 
contradicting our assumption.
\end{proof}

We can now prove Theorem~\ref{thm:freztotrigGamma}:
\begin{proof}
Given $\bool{B}\in\Gamma$ let $A\subseteq\bool{B}$ be a maximal antichain such that for all
$b\in A$ there is $k_b:\bool{B}\to\bool{C}_b$ $\Gamma$--freezing 
$\bool{B}$ 
with
$\coker(k_b)=b$.
Let 
\begin{align*}
k=\bigvee_A k_b:& \bool{B}\to \bool{C}=\bigvee_\Gamma\{\bool{C}_b:b\in A\}\\
& a\mapsto (k_b(a):b\in A)
\end{align*}
Then $k:\bool{B}\to\bool{C}$ $\Gamma$--freezes $\bool{B}$
 and is injective.

Now, given $\bool{B}_0$ let 
\[
\mathcal{F}=\{k_{ij}:\bool{B}_i\to\bool{B}_j:i\leq j<\lambda\}
\]
witness that $\bp{\bool{B}_i:i<\lambda}$ is a decreasing sequence in $\leq^*_\Gamma$ such that
$k_{ii+1}$ $\Gamma$--freezes 
$\bool{B}_i$.
Then $\varinjlim \mathcal{F}\in\Gamma$ is $\Gamma$--rigid 
and refines $\bool{B}_0$ in
$\leq^*_\Gamma$.
\end{proof}

%\subsection{Background material on iterations}

	% !TEX root = david-matteo.tex

\section{$\omega_1$--suitable classes}\label{suitable classes}

We organize this part of the paper as follows:
\begin{itemize}
\item We start giving the necessary definitions in \ref{subsec:defdavid}.
\item We state our main results in \ref{david:mainresults}. Specifically we assert that there are uncountably many $\omega_1$--suitable classes for the theory $\MK$ +`\emph{$\omega_1$ is the first uncountable cardinal}',  whose category forcing axioms yield pairwise incompatible theories for $H_{\omega_2}$ (this is incompatibility in first order logic).
\item In~\ref{david:proofs} we give the proofs, specifically:
\begin{itemize} 
\item
In \ref{4-freezing-posets}
we isolate four types of freezing posets which will be used to establish the freezeability property.
%for all these classes. 
\item
In \ref{iterability} we present the iteration lemmas 
that will be used to establish the weak iterability property (all of which were already known). 
\item 
In \ref{omega1suitability} we give  the proof that there are $\aleph_1$--many classes of forcing notions which are
$\omega_1$--suitable for the appropriate theories 
$T\supseteq \MK$  + `\emph{$\omega_1$ is the first uncountable cardinal}'.
\item In \ref{incompatible-category-forcing-axioms} we prove that category forcing axioms for the uncountably many $\omega_1$--suitable classes we produced in \ref{omega1suitability} yield pairwise 
incompatible theories for $H_{\omega_2}$.
\end{itemize}
\end{itemize}

\subsection{Forcing classes.} \label{subsec:defdavid}

In this section our background theory $T$ is $\MK$ + `\emph{$\omega_1$ is the first uncountable cardinal}'.

We will now define the main classes of forcing notions considered in this paper. Most of these classes are 
completely standard, but we nevertheless include their definition here for the benefit of some readers. 
%From now on when we state that a certain property holds for a poset $P$ 
%we automatically infer that the property holds for its boolean completion $\RO(P)$.
%Conversely when we assert that a complete boolean algebra $\bool{B}$ has a property defined for posets.
%it means the the poset $\bool{B}^+=\bool{B}\setminus\bp{0_{\bool{B}}}$ with the order inherited from 
%$\bool{B}$ has this property.

\begin{definition} 
A poset has the \emph{countable chain condition} (is \emph{c.c.c.}, for short) if and only it has no uncountable antichains. 
\end{definition}

Given an ordinal $\rho$, we will call a sequence $(X_i)_{i\leq\rho}$ a \emph{continuous chain} (or a \emph{continuous $\rho$--chain}, if we want to bring in the length) if 
\begin{itemize}
\item $(X_k)_{k\leq i}\in X_{i+1}$ whenever $i+1\leq \rho$ and 
\item $X_i=\bigcup_{k<i}X_k$ for every nonzero limit ordinal $i \leq\rho$.
\end{itemize}

An ordinal $\rho$ is said to be \emph{indecomposable} if $\rho=\o^\tau$ for some ordinal $\tau$.\footnote{Here, and elsewhere in the remainder of the paper, $\o^\tau$ denotes ordinal exponentiation.} 
Equivalently, $\rho$ is indecomposable if $\ot(\rho\setminus\eta)=\rho$ for every $\eta<\rho$.  $1$ is of course the 
first indecomposable ordinal. 

\begin{definition}
Given a countable indecomposable ordinal $\rho$, a poset $\mtcl P$ is \emph{$\rho$--proper} if and only if there is a cardinal 
$\t$ such that $\mtcl P\in H_\t$ and there is a club 
$D\sub [H_\t]^{\al_0}$ with the property that for every continuous chain $(N_i)_{i\leq \rho}$ of countable 
elementary submodels of $H_\t$ containing $\mtcl P$ and every $p\in N_0\cap \mtcl P$, there is an extension 
$q$ of $p$ such that $q$ is \emph{$(N_i, \mtcl P)$--generic} for all $i\leq\rho$, i.e., for every $i\leq\rho$ and every 
dense subset $D$ of $\mtcl P$, $D\in N_i$, $q\Vdash_{\mtcl P}D\cap \dot G\cap N_i\neq\emptyset$. 
\end{definition}

\begin{remark} $\mtcl P$ is $\rho$--proper if and only if for every cardinal $\t$ such that $\mtcl P\in H_\t$ there is such a club $D\sub [H_\t]^{\al_0}$ as in the above definition.\end{remark}

 $\rho$--$\PR$ denotes the class of $\rho$--proper posets. 
We write ${<}\o_1$--$\PR$ to denote the class of those posets that are in $\rho$--$\PR$ for every indecomposable 
$\rho<\o_1$. We say that $\mtcl P$ is \emph{proper} if it is $1$--proper, and denote $1\mbox{--}\PR$ also by $\PR$. 

The following is a simple but crucial observation:
\begin{fact}\label{fac:sigma2proper}
For any countable indecomposable ordinal $\rho$, the theory 
$\MK$+ `\emph{$\omega_1$ is the first uncountable cardinal}' + `\emph{$\rho$ is a countable indecomposable ordinal'}
%\begin{quote}
%$\MK+$\emph{($\omega_1$ is the first uncountable cardinal)}$+$\emph{($\rho$ is a countable indecomposable ordinal)}
%\end{quote}
 proves that
`\emph{$\RO(P)$ is $\rho$--proper}' is an absolutely $\Sigma_2$ property in parameters $\rho$ and $\omega_1$ and is also a 
$\Pi_2$ property
in the same parameters. The same can be proved for `\emph{$\RO(P)$ is ${<}\omega_1$--proper}' with respect to
$\MK$+`\emph{$\omega_1$ is the first uncountable cardinal}'.
\end{fact}
\begin{proof}
Let $\t$ be large enough such that  $\RO(P)\in H_\lambda$ for some $\lambda<\t$.
Then `\emph{$\RO(P)$ is $\rho$--proper}' holds in $V$ if and only if it holds in any (some) transitive set $X\supseteq H_\t$.
Hence,  $\RO(P)$ is $\rho$--proper if and only if there is some regular cardinal $\theta$ and some transitive $X\supseteq H_\t$ such that 
$(X,\in)\models\mbox{`$\RO(P)$ is $\rho$--proper'}$.

%\[
%\exists\t [\t\text{\emph{ is a regular cardinal}}\wedge \exists X (X\supseteq H_\t\wedge X\text{\emph{ is transitive }}
%\wedge(X,\in)\models\text{\emph{$\RO(P)$ is $\rho$--proper}})].
%\]
Now:
\begin{itemize}
\item
The formulae $(X,\in)\models\text{\emph{$\RO(P)$ is $\rho$--proper}}$ and \emph{$X$ is transitive} are 
$\Delta_1$ in the parameters $X$, $P$, $\rho$, $\omega_1$.
\item 
The formula $\t\text{\emph{ is a regular cardinal}}$ is $\Pi_1$ in parameter $\t$.
\item
The formula $X\supseteq H_\t$ is $\Pi_1$ in parameters $X,\t$ since it can be stated as 
\[
\forall w(|\text{trcl}(w)|<\t\rightarrow w\in X), 
\]
where $\text{trcl}(w)$ is the $\Delta_1$--definable operation assigining to the set $w$ its transitive closure.
\end{itemize}
It is now easy to check that \emph{$\RO(P)$ is $\rho$--proper} is an absolutely $\Sigma_2$ property in parameters 
$\rho$ and $\omega_1$. We leave it to the reader to check that it is also $\Pi_2$ in the same parameters.
\end{proof}
Being $\rho$--proper, for a forcing $\mtcl P$, is equivalent to $\mtcl P$ preserving a certain combinatorial property: 
Given a set $X$, we say that $S\sub\, ^\rho([X]^{\al_0})$ is \emph{$\rho$--stationary} if for every club $D\sub[X]^{\al_0}$ 
there is a continuous $\rho$--chain $\s$ of members of $D$ such that $\s\in S$. 
 
Recalling the standard characterization of properness, the 
following is not difficult to see over the theory 
$\MK$ + `\emph{$\omega_1$ is the first uncountable cardinal}'+`\emph{$\rho$ is a countable indecomposable ordinal}'.
 
 \begin{fact}\label{char-proper} Given an indecomposable ordinal $\rho<\o_1$,  the following are equivalent for every 
 poset $\mtcl P$.
 
 \begin{enumerate}
 
 \item $\mtcl P$ is $\rho$--proper.
 \item For every set $X$,  $\mtcl P$ preserves $\rho$--stationary subsets of $^\rho([X]^{\al_0})$; i.e., if 
 $S\sub\,^\rho([X]^{\al_0})$ is $\rho$--stationary, then $\Vdash_{\mtcl P}S\mbox{ is $\rho$--stationary}$. 
\end{enumerate}
\end{fact}

Using the above fact we can prove:
\begin{fact}\label{fac:Birkhoffproper}
For every countable indecomposable ordinals $\rho$, the theory $\MK$+`\emph{$\omega_1$ is the first uncountable cardinal}'+`\emph{$\rho$ is a countable indecomposable ordinal}'
%\begin{quote}
%$\MK+$\emph{($\omega_1$ is the first uncountable cardinal)}$+$\emph{($\rho$ is a countable indecomposable ordinal)}
%\end{quote}
proves that $\rho$--$\PR$ and ${<}\o_1$--$\PR$ are 
closed under preimages by complete injective homomorphisms, two--step iterations and products,
and contain
all countably closed forcings.
\end{fact}

\begin{definition}
A forcing notion $\mtcl P$ is \emph{$\rho$--semiproper} iff there is a cardinal $\t$ such that 
$\mtcl P\in H_\t$ for which there is a club $D\sub [H_\t]^{\al_0}$ 
with the property that for every continuous chain $(N_i)_{i\leq \rho}$ of countable elementary submodels of $H_\t$ 
containing $\mtcl P$ and every $p\in N_0\cap \mtcl P$, there is an extension $q$ of $p$ such that 
$q$ is \emph{$(N_i, \mtcl P)$--semi-generic} for all $i\leq\rho$. This means now that for every $i\leq\rho$ and every 
$\mtcl P$--name $\dot \a\in N_i$ for an ordinal in $\o_1^V$, $q\Vdash_{\mtcl P}\dot \a\in N_i$.
\end{definition}

\begin{remark} $\mtcl P$ is $\rho$--semiproper if and only if for every cardinal $\t$ such that $\mtcl P\in H_\t$ there is a club $D\sub [H_\t]^{\al_0}$ as in the above definition. \end{remark}

$\rho$--$\SP$ denotes the class of $\rho$--semiproper posets. Also, we write ${<}\o_1$--$\SP$ to denote 
 the class of those posets that are in $\rho$--$\SP$ for every indecomposable ordinal $\rho<\o_1$. We say that 
 $\mtcl P$ is \emph{semiproper} if it is $1$--semiproper, and denote $1\mbox{--}\SP$ also by $\SP$. 
 
 As before we have:
 \begin{fact}\label{fac:sigma2semiproper}
$\MK$ + `\emph{$\omega_1$ is the first uncountable cardinal}'+`\emph{$\rho$ is a countable indecomposable ordinal}'
proves that 
`\emph{$\RO(P)$ is $\rho$--semiproper}' is an absolutely $\Sigma_2$ property in parameters $\rho$ and $\omega_1$ and also a 
$\Pi_2$ property
in the same parameters. The same holds for `\emph{$\RO(P)$ is ${<}\omega_1$--semiproper}'.
\end{fact}
 
 Let $X$ be a set such that $\o_1\sub X$. We say that $S\sub\,^\rho([X]^{\al_0})$ is \emph{$\rho$--semi-stationary} if for every club $D\sub [X]^{\al_0}$ there  are continuous $\rho$--chains $\s=(x_i\,:\,i\leq\rho)$ and $\s'=(x_i'\,:\,i\leq\rho)$ such that 
 
 \begin{itemize}
 
 \item $\s\in S$,  
 \item $\range(\s')\sub D$, and
 \item for each $i\leq\rho$, $x_i\sub x_i'$ and $x_i\cap\o_1=x_i'\cap\o_1$.
 \end{itemize}
 
 We have the following characterization of $\rho$--semiproperness (for any given indecomposable $\rho<\o_1$).
 
 \begin{fact}\label{char-semiproper} Given an indecomposable ordinal $\rho<\o_1$,  the following are equivalent for every poset $\mtcl P$.
 
 \begin{enumerate}
 
 \item $\mtcl P$ is $\rho$--semiproper.
 \item  $\mtcl P$ preserves $\rho$--semi-stationary subsets of $^\rho([X]^{\al_0})$ for every set $X$; i.e., if $S\sub\,^\rho([X]^{\al_0})$ is $\rho$--semi-stationary, then $S$ remains $\rho$--semi-stationary after forcing with $\mtcl P$. 
\end{enumerate}
\end{fact}
Again we get that:
\begin{fact}\label{fac:Birkhoffsemiproper}
For all countable indecomposable ordinals $\rho$
$\MK$+`\emph{$\omega_1$ is the first uncountable cardinal}'+ `\emph{$\rho$ is a countable indecomposable ordinal}'
proves that 
$\rho$--$\SP$ and ${<}\o_1$--$\SP$ are
closed under preimages by complete injective homomorphisms, two--step iterations and products, and contain
all countably closed forcings.
\end{fact}

\begin{definition}
Given a regular cardinal $\k\geq\o_1$, a poset $\mtcl P$ \emph{preserves stationary subsets of $\k$} if 
every stationary subset of $\k$ remains stationary after forcing with $\mtcl P$.  
\end{definition}
$\SSP$ denotes the class of 
partial orders preserving stationary subsets of $\o_1$. More generally, given a  cardinal $\l$, $\SSP(\l)$ denotes
the class of partial orders preserving stationary subsets of $\k$ for every uncountable regular cardinal $\k\leq\l$.
%\begin{fact}\label{fac:sigma2SSP}
%\emph{$\RO(P)$ is $\SSP(\l)$--proper} is an absolutely $\Sigma_2$-property in parameter $\l$ and is also a 
%$\Pi_2$-property
%in the same parameter. %The same can be proved for \emph{$\RO(P)$ is $<\omega_1$--proper}
%\end{fact} 

Recall that a Suslin tree  is an $\o_1$--tree $T$ (i.e., $T$ is a tree of height $\o_1$ all of whose levels are countable) without uncountable chains or antichains (a subset of $T$ is called an \emph{antichain} iff it consists of pairwise incomparable nodes). We will consider the above properties in conjunction with the preservation of some combination of the following two properties.

 \begin{definition}
 A poset $\mtcl P$ \emph{preserves Suslin trees} if $\Vdash_{\mtcl P}T\mbox{ is Suslin}$ for every Suslin tree $T$ in the ground model. %We use $\STP$ to denote the class of all posets preserving Suslin trees.
 \end{definition}
  
 \begin{definition}
 A poset $\mtcl P$ is \emph{$^\o\o$--bounding} iff every function $f:\o\into\o$ added by $\mtcl P$ is bounded by a function 
 $g:\o\into\o$ in the ground model; i.e., iff  for every $\mtcl P$--generic filter $G$ and every $f:\o\into \o$, $f\in V[G]$, there is some $g:\o\into g$, $g\in V$, such that $f(n)<g(n)$ for all $n$. 
 \end{definition}
$\STP$ denotes the class of all posets preserving Suslin trees and
$^\o\o\mbox{--bounding}$ the class of $^\o\o$--bounding posets. 

By the same arguments we gave for $\rho$--properness one gets:
\begin{fact}\label{fac:sigma2SPTomombound}
$\MK$+ `\emph{$\omega_1$ is the first uncountable cardinal}'
proves that 
`\emph{$\RO(P)$ preserves Suslin trees}' and `\emph{$\mtcl P$ is $^\o\o$--bounding}'
are absolutely $\Sigma_2$ properties in parameters $\omega_1$ and $\omega^\omega$, and 
$\Pi_2$ properties
in the same parameters. Moreover, this theory  
%$\MK$+ `\emph{$\omega_1$ is the first uncountable cardinal}'
proves that 
$\STP$  and
$^\o\o\mbox{--bounding}$ are both closed under preimages by complete injective homomorphisms, two--step iterations and products,
and contain all countably closed forcings.
%The same can be proved for \emph{$\RO(P)$ is $<\omega_1$--proper}
\end{fact}

In \cite[XI]{SHEPRO} Shelah isolates a property he calls $S$--condition for 
which the following can be proved.
\begin{lemma}\label{S-cond-shelah}
Assume $\mtcl P$ is a forcing notion satisfying the $S$--condition. Then:
\begin{enumerate}
\it $\mtcl P$ preserves  stationary subsets of $\o_1$;\footnote{See \cite[XI--Thm. 3.6]{SHEPRO}. This theorem says that forcing notions with the $S$--condition do not collapse $\o_1$.  However, its proof actually establishes that such forcings in fact preserve stationary subsets of $\o_1$.}
\it if $\CH$ holds, then $\mtcl P$ adds no new reals.
\end{enumerate}
\end{lemma}

As shown in  \cite[XI--4]{SHEPRO}, among the forcing notions satisfying the $S$--condition are Namba forcing (and natural variations thereof), all  countably closed forcing notions, and the natural poset which, for a fixed stationary $S\sub \{\a<\o_2\,:\,\cf(\a)=\o\}$, adds an $\o_1$--club through $S$ with countable conditions. 

Given a tree $T$ and a node $\eta$ of $T$, let $\textsf{succ}_T(\eta)$ denote the set of immediate successors of $\eta$ in $T$. It will be convenient to define the following game $\mtcl P^{\mtcl P}_p$ (for a partial order $\mtcl P$ and a $\mtcl P$--condition $p$).

\begin{definition}
Given a partial order $\mtcl P$ such that $\av\mtcl P\av\geq\al_2$, $\mtcl G^{\mtcl P}$ is 
the following game of length $\o$ between players I and II, with player I playing at even stages and player II playing at 
odd stages.

\begin{enumerate}

\it At any given stage $n$ of the game, the corresponding player picks a pair $T^n$, $(p^n_\eta)_{\eta\in T^n}$, 
where $T^n$ is a tree consisting of finite sequences of ordinals in $\av\mtcl P\av$ without infinite branches and 
where $(p^n_\eta)_{\eta\in T^n}$ is a sequence of conditions in $\mtcl P$ extending $p$ such that $p^n_\n$ 
extends $p^n_\eta$ in $\mtcl P$ whenever $\n$ extends $\eta$ in $T^n$.  

\it If $n>0$, then 

\begin{enumerate}
 \it $T^n$ and $(p^n_\eta)_{\eta\in T^n}$ end--extend $T^{n-1}$ and $(p^{n-1}_\eta)_{\eta\in T^{n-1}}$, respectively, 
\it every terminal node in $T^{n-1}$ has a proper extension in $T^n$, and
\it every node in $T^n\setminus T^{n-1}$ extends a unique terminal node in $T^{n-1}$.
\end{enumerate} 

\it Player I starts by playing $T_0=\{\emptyset\}$ and $p^0_\emptyset \in \mtcl P$.

\it At any given even stage $n>0$ of the game, player I picks, for every terminal node $\eta$ of $T^{n-1}$, a finite sequence $\n_\eta$ of ordinals in $\av\mtcl P\av$ such that $\n_\eta$ extends $\eta$ properly.  He then builds $T^n$ as $$T^{n-1}\cup\{\n_\eta\restr k\,:\,k\leq \av\n_\eta\av,\,\eta\mbox{ a terminal node of }T^{n-1}\}.$$ Player I also has to choose of course  $(p^n_\eta)_{\eta\in T^n}$ in such a way that (1) and (2) are satisfied. 

\it At any given odd stage $n$ of the game, player II chooses, for every terminal node $\eta$ of $T^{n-1}$, a regular cardinal $\k^n_\eta\in [\al_2,\,\av\mtcl P\av]$, and builds $T^n$ from $T^{n-1}$ by adding to $T^{n-1}$ a next level where, for each terminal node $\eta$ of $T^{n-1}$, the set of immediate successors of $\eta$ in $T^n$ is $\{\eta^\smallfrown \la\a\ra\,:\,\a<\k^n_\eta\}$. Player II also has to choose of course  $(p^n_\eta)_{\eta\in T^n}$ in such a way that (1) and (2) are satisfied.

\end{enumerate}

After $\o$ moves, the players have naturally built a tree $T=\bigcup_n T^n$ of height $\o$ whose nodes are 
finite sequences of ordinals in $\av\mtcl P\av$, together with a sequence 
$(p_\eta)_{\eta\in T}=\bigcup (p^n_\eta)_{\eta\in T^n}$ of $\mtcl P$--conditions such that for all nodes $\eta$, 
$\nu$ in $T$, if $\nu$ extends $\eta$ in $T$, then $p_{\nu}$ extends $p_{\eta}$ in $\mtcl P$. Finally, player II wins the 
game iff for every subtree $T'$ of $T$ in $V$, 
if $\av\textsf{succ}_{T'}(\eta)\av=\av\textsf{succ}_T(\eta)\av$ for every $\eta\in T'$, 
then there is a condition in $\mtcl P$ forcing that there is an $\o$--branch $b$ through $T'$ such that 
$p_{b\restr n}\in\dot G$ for all $n<\o$. 
 
\end{definition}

The definition of the $S$--condition is the following.\footnote{Shelah's definition is more general, but 
the present form suffices for our purposes.}

\begin{definition}\label{s-cond-def}
A partial order $\mtcl P$ satisfies the $S$--condition if and only if $\av\mtcl P\av\geq\al_2$ and %for every $p\in\mtcl P$ 
player II has a winning strategy $\s$ in the game $\mtcl G^{\mtcl P}$%_p$ 
such that for every partial run of the game, 
the output of $\s$ at any given sequence $\eta\in\,^{{<}\o}\av\mtcl P\av$ depends only on $\eta$, 
$(p_{\eta\restr k})_{k\leq \av\eta\av}$ and $\{k<\av\eta\av\,:\,\av \textsf{succ}_T(\eta\restr k)\av>1\}$, 
where $T$ denotes the tree built by the players up to that point.
\end{definition}

$S\textsf{--cond}$ is the class of complete Boolean algebras $\bool{B}$ satisfying the $S$--condition.

One has:
\begin{fact}\label{fac:sigma2Scond}
$\MK$+`\emph{$\omega_1$ is the first uncountable cardinal}'
proves that 
`\emph{$\RO(P)$ satisfies the $S$--condition}' is an absolutely $\Sigma_2$ property in parameter $\omega_2$ 
and also a 
$\Pi_2$ property
in the same parameter. 

$\MK$+ `\emph{$\omega_1$ is the first uncountable cardinal}'
proves that 
$S\textsf{--cond}$ is closed under preimages by complete injective homomorphisms, two--step iterations and products, and contains
all countably closed forcings.
%The same can be proved for \emph{$\ROP}$ is $<\omega_1$--proper}
\end{fact}
\begin{proof}
As in the case of all other classes dealt with in this section, if $\RO(P)\in H_\theta$, then 
$H_\theta\models\mbox{`$\RO(P)$ satisfies the $S$--condition}$' if and only if 
$\RO(P)$ satisfies the $S$--condition.

%\[
%\text{$H_\theta\models$\emph{ $\RO(P)$ satisfies the $S$--condition} if and only if 
%$V\models$\emph{ $\RO(P)$ satisfies the $S$--condition}.}
%\]
%\begin{aMnote}{Problem!}
%Assume $P$ is a complete suborder of a poset with the
%$S$-condition, and $P$ forces that $\dot{Q}$ is a complete suborder of a poset with the
%$S$-condition. How do we prove that $P\ast\dot{Q}$ is a complete suborder of a poset with the
%$S$-condition?
The remaining properties of $S\textsf{--cond}$ other than the closure under complete subalgebras
are left to the reader.

We prove now that $S\textsf{--cond}$ is a class closed under complete subalgebras;
it will be convenient for this to resort to the algebraic properties of complete injective homomorphisms
with adjoints outlined in Theorem~\ref{eRetrProp} of the appendix.

Assume $\bool{B}$ is a complete subalgebra of some $\bool{C}$ satisfying the $S$--condition.
Let $\pi:\bool{C}\to\bool{B}$ be the adjoint map of the inclusion map of $\bool{B}$ into $\bool{C}$.
Let $\sigma$ be the winning strategy for player $II$ in $\mtcl G^{\bool{C}^+}$. % for any $c\in\bool{C}$.
Define %for any $b\in\bool{B}$ 
$\sigma'$ to be the strategy for player $II$ in 
 $\mtcl G^{\bool{B}^+}$ obtained by the following procedure:
 \begin{itemize}
% \item At stage $1$ player $II$ takes 
% \[
% \sigma(\emptyset)=\ap{T_0,\bp{c_\eta:\eta\in T_0}}
% \] 
% and 
% let 
% \[
% \sigma'(\emptyset)=\ap{T_0,\bp{\pi(c_\eta):\eta\in T_0}}.
% \]
 \item
 Player $I$ and $II$ build 
 partial plays $\ap{T_{n},\bp{b_\eta:\eta\in T_n}}$  in 
 $\mtcl G^{\bool{B}^+}$ and partial plays $\ap{T_n,\bp{c_\eta:\eta\in T_n}}$ in 
 $\mtcl G^{\bool{C}^+}$ according to these prescriptions:
 \begin{itemize}
 \item
 for all $\eta\in T_n$ and all $n$
 we have that $\pi(c_\eta)=b_\eta$;
 \item
 for all terminal nodes $\eta\in T_{2n}$ 
 with $\eta=\ap{\gamma_0,\dots,\gamma_m}$, we have that 
 \[
 c_\eta=b_\eta\wedge c_{\ap{\gamma_0,\dots,\gamma_{m-1}}};
 \]
  \item 
  $\ap{T_{2n+1},\bp{c_\eta:\eta\in T_{2n+1}}}=\sigma(\ap{T_{2n},\bp{c_\eta:\eta\in T_{2n}}})$;
 \end{itemize}
 \item Player $II$ defines $\sigma'(\ap{T_{2n},\bp{b_\eta:\eta\in T_{2n}}})=\ap{T_{2n+1},\bp{b_\eta:\eta\in T_{2n+1}}}$.
 \end{itemize}
 
 Now assume $\ap{T,\bp{b_\eta:\eta\in T}}$ is built according to a play of $\mtcl G^{\bool{B}^+}$ in which II follows 
 $\sigma'$.
 Fix a subtree $T'\subseteq T$ as given by the winning condition for $II$ in
 $\mtcl G^{\bool{B}^+}$.
 Given  a $V$--generic filter $H$ for $\bool{B}$, we must find some infinite branch $\eta$ of $T'$ in $V[H]$ such that
 $b_{\eta\restriction n}\in H$ for all $n$.
 To find this branch let 
 $\ap{T,\bp{c_\eta:\eta\in T}}$ be the tree built in tandem with $\ap{T,\bp{b_\eta:\eta\in T}}$
 according to the rules we used to define $\sigma'$. 
 Fix $G\supseteq H$ $V$--generic for $\bool{C}$.
 Since $\bool{C}$ satisifes the $S$--condition, we can find some infinite 
 branch $\eta$ of $T'$ in $V[G]$ such that $c_{\eta\restriction n}\in G$ for all $n$.
 This gives that $b_{\eta\restriction n}=\pi(c_{\eta\restriction n})\in H$ for all $n$.
 Hence, in $V[G]$ there is an infinite branch $\eta$ of $T'$ such that
 $b_{\eta\restriction n}\in H$ for all $n$. Therefore the tree $T^*=\bp{\eta\in T': b_\eta\in H}\in V[H]$ 
 is ill--founded in $V[G]$.
But then the same is true in $V[H]$ by absoluteness of ill--foundedness. Finally, any infinite branch of $T^*$ witnesses the winning condition for II 
 using
 $\sigma'$ relative to $T,T',H$.

 Since this can be done for all possible choices of $T\supseteq T'$ in $V$ with $T$ constructed using $\sigma'$, and 
 for all  $V$--generic filters $H$ for $\bool{B}$, we have that 
 $\bool{B}$ satisfies the $S$--condition.
\end{proof}

%The rest of this part is structured as follows. 
%In the next section, we define all relevant notions leading to the definition of $\l$--$\CFA(\Gamma)$, for a 
%class $\Gamma$ of complete Boolean algebras, and present, without proofs, the general theory of category 
%forcing axioms. We refer the reader to the forthcoming~\cite{VIAAUDSTEBOOK} for the proofs.  
%In Section \ref{4-freezing-posets} we give four instances of $\SSP$--freezing posets. In 
%Section \ref{incompatible-category-forcing-axioms} we state our main results,  viz.\ Theorems \ref{main-thm-suitable} 
%and  \ref{main-thm-incompatible}. Theorem \ref{main-thm-suitable} is proved in Subsection \ref{Proving Theorem 
%main-thm-suitable}, using the results from Section \ref{4-freezing-posets}, and 
%Theorem \ref{main-thm-incompatible} is proved in Subsection \ref{incompatibility}. 

\subsection{Incompatible category forcing axioms}\label{david:mainresults}

In this section we  isolate $\al_1$--many classes of forcing notions, all of which are $\o_1$--suitable (in some cases modulo the existence of unboundedly many measurable cardinals), and which are pairwise incompatible, provably in $\ZFC$ + $\LC$, where $\LC$ denotes `\emph{There is a proper class of supercompact cardinals}' + `\emph{There is a $2$--superhuge cardinal}'.\footnote{Sometimes, a weaker theory than this suffices for our incompatibility result.} Our main results are the following.

\begin{theorem}\label{main-thm-suitable} 
\begin{enumerate}
\item Each of the following classes is $\o_1$--suitable with respect
to $\MK$+`\emph{$\omega_1$ is the first uncountable cardinal}'.

\begin{enumerate}
\item $\PR$,
\item $\PR\cap \STP$,
\item $\PR\cap\,^\o\o\textsf{--bounding}$,

\item  $\PR\cap\STP\cap\,^\o\o\textsf{--bounding}$. 
\end{enumerate}

\item $\rho\mbox{--}\PR$ is $\o_1$--suitable for the theory
$\MK$+`\emph{$\omega_1$ is the first uncountable cardinal}' +
`\emph{$\rho$ is a countable indecomposable ordinal}'
for every countable indecomposable ordinal $\rho$ such that $1<\rho<\o_1$.

\item ${<}\o_1\mbox{--}\PR$ is $\o_1$--suitable for 
$\MK$+`\emph{$\omega_1$ is the first uncountable cardinal}'.

\item  $S\textsf{--cond}$ is $\o_1$--suitable for 
$\MK$+`\emph{$\omega_1$ is the first uncountable cardinal}'.

\item %Suppose there  is a proper class of measurable cardinals. If  $\rho<\omega_1$ is an indecomposable ordinal, then 
Each of the following classes is $\o_1$--suitable for 
$\MK$+`\emph{$\omega_1$ is the first uncountable cardinal}'+`\emph{$\rho$ is a countable indecomposable ordinal}'
+ `\emph{there is a proper class of measurable cardinals}' for any countable indecomposable ordinal $\rho<\omega_1$:

\begin{enumerate}
\item $\rho\mbox{--}\SP$,
\item $(\rho\mbox{--}\SP)\cap \STP$,
\item $(\rho\mbox{--}\SP)\cap\,^\o\o\textsf{--bounding}$,
\item $(\rho\mbox{--}\SP)\cap\STP\cap\,^\o\o\textsf{--bounding}$.
\end{enumerate}

\item 
Each of the following classes is $\o_1$--suitable for 
$\MK$+ `\emph{$\omega_1$ is the first uncountable cardinal}'+ `\emph{there is a proper class of measurable cardinals}':
%Suppose there  is a proper class of measurable cardinals. Then each of the following classes is $\o_1$--suitable.

\begin{enumerate}
\item ${<}\o_1\mbox{--}\SP$ 
\item $({<}\o_1\mbox{--}\SP)\cap \STP$
\item $({<}\o_1\mbox{--}\SP)\cap\,^\o\o\textsf{--bounding}$
\item $({<}\o_1\mbox{--}\SP)\cap\STP\cap\,^\o\o\textsf{--bounding}$ 
\end{enumerate}

\end{enumerate}
\end{theorem}

\begin{theorem}\label{main-thm-incompatible} ($\ZFC$+$\LC$)
Suppose $\Gamma$ and $\Gamma'$ are any two different classes of forcing notions mentioned in Theorem 
\ref{main-thm-suitable}. Then $\CFA(\Gamma)$ implies $\lnot\CFA(\Gamma')$.
\end{theorem}

As will be clear from the proofs, Theorems \ref{main-thm-suitable} and \ref{main-thm-incompatible} are just 
selected samples of a zoo of possibly incompatible instances of $\CFA(\Gamma)$.
 In particular, it should be possible to combine (some of) the classes mentioned in Theorem  \ref{main-thm-suitable} with other classes of forcing notions, besides $\STP$ and $^\o\o\mbox{--bounding}$, so long as these classes have a suitable iteration theory and reasonable closure properties, are absolutely $\Sigma_2$ definable, and the resulting classes contain $\SSP$--freezing posets.

%The next section contains the lemmas which, together, prove Theorem \ref{main-thm-suitable}.
%Section \ref{incompatibility} contains the lemmas that yield the proof of Theorem \ref{main-thm-incompatible}.

We should point out that the following natural question -- in the present context -- remains open.

\begin{question} Is there, under any reasonable large cardinal, any indecomposable $\rho<\o_1$, $\rho>1$, for which any of the following classes is $\o_1$--suitable?

\begin{enumerate}
\item $(\rho\mbox{--}\PR)\cap \STP$
\item $(\rho\mbox{--}\PR)\cap\,^\o\o\textsf{--bounding}$
\item $(\rho\mbox{--}\PR)\cap\STP\cap\,^\o\o\textsf{--bounding}$ 
\item $({<}\o_1\mbox{--}\PR)\cap \STP$
\item $({<}\o_1\mbox{--}\PR)\cap\,^\o\o\textsf{--bounding}$
\item $({<}\o_1\mbox{--}\PR)\cap\STP\cap\,^\o\o\textsf{--bounding}$ 
\end{enumerate}
\end{question}

The following question, of a more foundational import, addresses the possibility of there being $\l$--suitable classes for 
$\l>\o_1$.

\begin{question} Are there, under some reasonable large cardinal assumption, any cardinal $\l\geq\o_2$ and any class 
$\Gamma$ of forcing notions such that $\Gamma$ is $\l$--suitable? Are there, again  under some reasonable large cardinal 
assumption, any cardinal $\l\geq\o_2$ and any class $\Gamma$ of forcing notions such that $\Gamma$ is $\l$--suitable and 
such that $\CFA_{\l}(\Gamma)$ is compatible with -- or, even, implies -- $\CFA_{\o_1}(\Gamma')$ for any 
$\o_1$--suitable class $\Gamma'$?
\end{question}

\subsection{Proof of theorems ~\ref{main-thm-suitable} and ~\ref{main-thm-incompatible}.}\label{david:proofs}

\subsubsection{Four freezing posets}\label{4-freezing-posets}

In this section we introduce four instances of $\SSP$--freezing posets. We feel free to confuse posets with complete
Boolean algebras, as the context will dictate which is the correct intended meaning of the concept.
%These posets will be crucially used in 
%Section \ref{incompatible-category-forcing-axioms} when proving that the classes of forcing notions considered there are 
%$\omega_1$--suitable.
When proving $\SSP$--freezability, we will actually be showing the following sufficient condition (for $\lambda=\omega_1$).

\begin{lemma}\label{freezability-sufficient} Let $\lambda\geq\omega_1$ be a cardinal, $\bool{B}$ a forcing notion, and 
$\dot{\bool{C}}$ a $\bool{B}$--name for a forcing notion. 
Suppose that $p$ is a set, and that if $G$ is a $\bool{B}$--generic filter, then $\bool{C}=\dot{\bool{C}}_G$ forces that there is some $A_G\subseteq\lambda$ coding $G$ in an absolute way mod.\ $\SSP(\lambda)$, in the sense that there is some $\Sigma_1$--formula $\varphi(x, y, z)$ such that, if $H$ is a $\bool{B}\ast\dot{\bool{C}}$--generic filter over $V$ such that $H\cap \bool{B}=G$, then

\begin{enumerate}

\item $(H_{\lambda^+}; \in, \NS_\lambda)^{V[H]}\models\varphi(G, A_G, p)$, and

\item  in every outer model $M$ of $V[H]$ such that $\mtcl P(\lambda)^{V[H]}\cap (\NS_\lambda)^M=(\NS_\lambda)^{V[H]}$, if $$(H_{\lambda^+}; \in, \NS_\lambda)^{M}\models\varphi(G_0, A_{G_0}, p)$$ and $$(H_{\lambda^+}; \in, \NS_\lambda)^{M}\models\varphi(G_1, A_{G_1}, p),$$ then $G_0=G_1$.\end{enumerate} 

Then the natural inclusion $$i:\bool{B}\into \bool{B}\ast\dot{\bool{C}}$$ $\SSP(\lambda)$--freezes $\bool{B}$.

\end{lemma}

\begin{proof}
Suppose, towards a contradiction, that $b_0$, $b_1\in \bool{B}$ are incompatible, $\bool{D}$ is a complete Boolean algebra, $k_0:(\bool{B}\ast\dot{\bool{C}})\restr b_0\into\bool{D}$, $k_1:(\bool{B}\ast\dot{\bool{C}})\restr b_1\into \bool{D}$ are complete homomorphisms, $K$ is $\bool{D}$--generic and, for each $\epsilon\in \{0, 1\}$, $H_\epsilon = k_\epsilon^{-1}(K)$ and every stationary subset of $\lambda$ in $V[H_\epsilon]$ remains stationary in $V[K]$. For each $\epsilon\in\{0, 1\}$, let $G_\epsilon$ be the filter on $\bool{B}$ generated by $H_\epsilon\cap (\bool{B}\restr b_\epsilon)$, and let $A_\epsilon\sub \lambda$ be such that $$(H_{\lambda^+}; \in, \NS_\lambda)^{V[H_\epsilon]}\models\varphi(G_\epsilon, A_\epsilon, p)$$ Since $$(H_{\lambda^+}; \in, \NS_\k)^{V[K]}\models\varphi(G_0, A_0, p)\wedge\varphi(G_1, A_1, p),$$ we have that $G_0=G_1$ by (2). But this is impossible since $b_0\in G_0$, $b_1\in G_1$, and since $b_0$ and $b_1$ are incompatible conditions in $\bool{B}$.  
\end{proof}

Our first freezing poset comes essentially from \cite{MOO06}.

Given a set $X$, the Ellentuck topology on $[X]^{\aleph_0}$ is the topology on $[X]^{\aleph_0}$ generated by the sets of the form $[s, Y]$, for $Y\in [X]^{\al_0}$ and $s\in [Y]^{{<}\omega}$, where $[s, Y]=\{Z\in [Y]^{\aleph_0}\,:\, s \subseteq Z\}$.

The following lemma, except for the conclusion that $\mtcl P$ preserves Suslin trees, is due to Moore \cite{MOO06}. The conclusion that $\mtcl P$ preserves Suslin trees is due to Miyamoto  \cite{Miyamoto-Yorioka}. 

\begin{lemma}\label{MRP}
Let  $X$ be a set, $\t$ a cardinal such that $X\in H_{\t}$,  and $\S$ a function with domain $[H_\t]^{\al_0}$ such that for every countable $M\elsub  [H_\t]^{\al_0}$, 

\begin{itemize} 
 
 \item $\S(M)\sub [X]^{\al_0}$ is open in the Ellentuck topology, and 

\item $\S(M)$ is $M$--stationary (meaning that for every function $F:[X]^{{<}\o}\into X$, if $F\in M$, then there is some $Z\in \S(M)\cap M$ such that $F``[Z]^{{<}\o}\sub Z$).\end{itemize}

Let $\mtcl P=\mtcl P_{X, \t, \S}$ be the set, ordered by reverse inclusion, of all countable  $\sub$--continuous $\in$--chains  $p=(M^p_i)_{i\leq\n}$ of countable elementary substructures of $H_\t$ such that for every limit ordinal $i\leq\n$ there is some $i_0<i$ with the property that $M^p_{k}\cap X\in \S(M^p_i)$ for all $k$ such that $i_0<k<i$. 

Then \begin{enumerate}

\item $\mtcl P$ is proper, preserves Suslin trees, and does not add new reals.

\item Whenever $G$ is $\mtcl P$--generic over $V$ and $M^G_i= M^p_i$ for $p\in G$ and $i\in \dom(p)$, 
$(M^G_i)_{i<\o_1}$ is in $V[G]$ the $\sub$--increasing enumeration of a club of $[H_\t^{V}]^{\al_0}$ 
and is such that:
\begin{quote} 
For every limit ordinal 
$i<\o_1$ there is some $i_0<i$ with the property that $M^G_{k}\cap X\in \S(M^G_i)$ for all $k$ such that $i_0<k<i$.   
\end{quote}
\end{enumerate}
\end{lemma}

\begin{remark} In most interesting cases, the forcing $\mtcl P_{X, \t, \S}$ in the above lemma is 
not $\o$--proper. \end{remark}

In \cite{MOO06}, Moore defines the Mapping Reflection Principle ($\MRP$) as the following statement: Given $X$, 
$\t$, and $\S$ as in the hypothesis of Lemma \ref{MRP}, there is a $\sub$--continuous $\in$--chain $(M_i)_{i<\o_1}$ 
of countable elementary substructures of $H_\t$ such that for every limit ordinal $i<\o_1$ there is some $i_0<i$ with the property 
that $M_{k}\cap X\in \S(M_i)$ for every $k$ such that $i_0<k<i$. 

It follows from Lemma \ref{MRP} that $\MRP$ is a consequence of $\textsc{PFA}$, and of the 
forcing axiom for the class of  forcing notions in $\PR\cap \STP$ not adding new reals. 

We will call a partial order $\mtcl R$ an \emph{$\MRP$--poset} if there are $X$, $\t$ and $\S$ as in the hypothesis of 
Lemma \ref{MRP} such that $\mtcl R=\mtcl P_{X, \t, \S}$.

\begin{proposition}\label{freeze1}
Given a forcing notion $\mtcl P$,  there is $\mtcl P$--name $\dot{\mtcl Q}$ for  a forcing notion such that 
\begin{enumerate} 
\it $\dot{\mtcl Q}$ is forced to be of the form $\Coll(\o_1, \mtcl P)\ast\dot{\mtcl R}$, where $\dot{\mtcl R}$ is a $\Coll(\o_1, \mtcl P)$--name for an $\MRP$--poset, and 
\it $\mtcl P\ast\dot{\mtcl Q}$ $\SSP$--freezes $\mtcl P$, as witnessed by the inclusion map.
  \end{enumerate}
\end{proposition}

\begin{proof}
By Lemma \ref{freezability-sufficient}, it suffices to prove that $\mtcl P$ forces that in $V^{\Coll(\o_1, \mtcl P)}$ there is an  $\MRP$--poset $\dot{\mtcl R}$ such that $\Coll(\o_1, \mtcl P)\ast\dot{\mtcl R}$ codes the generic filter for $\mtcl P$ in an absolute way mod.\ $\SSP$ in the sense of that lemma. For this, let us work in $V^{\mtcl P\ast \Coll(\o_1,\,\mtcl P)}$. Let $\dot B_G$ be a subset of $\o_1$ coding the generic filter for $\mtcl P$ in some canonical way, let $\vec C=(C_\delta \,:\, \delta \in \Lim(\omega_1))\in V$ be a ladder system on $\omega_1$ (i.e., every $C_\d$ is a cofinal subset of $\d$ of order type $\o$), and let $(S_\alpha)_{\alpha<\omega_1}\in V$ be a partition of $\omega_1$ into stationary sets. Given $X\sub Y$, countable sets of ordinals, such that  $Y\cap \omega_1$ and $\ot(Y)$ are both limit ordinals and such that $X$ is bounded in $\mbox{sup}(Y)$, let $c(X, Y)$ mean  $$\av C_{\ot(Y)} \cap \mbox{sup}(\pi_Y``X)\av < \av C_{Y\cap \omega_1}\cap X\cap \omega_1\av,$$ where $\pi_Y$ is the collapsing function of $Y$.

  Let  $\t$ be a large enough cardinal and let $\S$ be the function sending a countable $N\elsub H_\t$ to the set of  $Z\in [\o_2\cap N]^{\al_0}$ such that $c(X, \o_2\cap N)$ iff the unique $\a<\o_1$ such that $N\cap \o_1\in S_\a$ is in $B_{\dot G}$. Now,  $X=\o_2$, $\t$ and $\S$ satisfy the hypothesis of Lemma \ref{MRP} (\cite{MOO06}).\footnote{$\S$ is in essence the mapping used by Moore in \cite{MOO06} to prove that $\BPFA$ implies $2^{\al_1}=\al_2$.} Let $\mtcl R=\mtcl P_{\o_2, \t, \S}$. By Lemma \ref{MRP}, $\mtcl R$ adds  $(Z^{\dot G}_i)_{i<\o_1}$, a strictly $\sub$--increasing enumeration of a club of $[\o_2^V]^{\al_0}$, such that for every limit ordinal $i<\o_1$, if $Z^{\dot G}_i\cap\o_1\in S_\a$, then  there is a tail of $k< i$ such that  $c(Z^{\dot G}_{k}, Z^{\dot G}_i)$ if and only if $\a\in B_{\dot G}$.   
  
 Let $\k=\o_2$, let $H$ be $(\mtcl P\ast \Coll(\o_1, \mtcl P))\ast \dot{\mtcl R}$--generic, let $G=H\cap\mtcl P$, and let $A_G$ be a subset of $\omega_1$ which canonically codes $B_G$ and $(Z^G_i)_{i<\omega_1}$. If $M$ is any outer model such that every stationary subset of $\o_1$ in $V[H]$ remains stationary in $M$, then $B_G$ is the unique subset $B$ of $\o_1$ for which there is,  in $M$, a set $A\sub\o_1$ coding $B$ together with an $\sub$--increasing enumeration $(Z_i)_{i<\o_1}$ of  a club of $[\k]^{\al_0}$ with the property that  for every limit ordinal $i<\o_1$, if $Z_i\cap\o_1\in S_\a$, then  there is a tail of  $k< i$ such that  $c(Z_{k}, Z_i)$ if and only if $\a\in B$. Indeed, If $B'\in M$ were another such set, as witnessed by $A'\sub M$, $\a\in B\Delta B'$, and $(Z'_i)_{i<\o_1}\in M$ were an $\sub$--increasing enumeration of a club of $[\k]^{\al_0}$ with the property that for every limit ordinal $i<\o_1$, if $Z'_i\cap\o_1\in S_\a$, then there is a tail of  $k< i$ such that  $c(Z_{k}, Z_i)$ if and only if $\a\in B'$, then we would be able to find some $i$ such that $Z_i=Z'_i$, $Z_i\cap \o_1\in S_\a$, and such that $Z_{k}=Z'_k$ for  all $k$ in some cofinal subset $J$ of $i$. But then we would have that  $c(Z_{k}, Z_i)$, for all $k$ in some final segment of $J$, both holds and fails. 
 
 Finally, it is immediate to see that there is a $\Sigma_1$--formula $\varphi(x, y, z)$ such that $\varphi(G, A_G, p)$ expresses the above property of $G$ and $A_G$ over $(H_{\omega_2}; \in)^{M}$ for any $M$ as above, for $p=(\kappa, \vec C, (S_\a)_{\a<\o_1})$. 
\end{proof}

Using coding techniques from \cite{CAIVEL06}, one can prove the following stronger version of Lemma \ref{freeze1}. However, we do not have any use for this stronger form, so we will not give the proof here. 

\begin{lemma} 
Given a  partial order $\mtcl P$  there is $\mtcl P$--name $\dot{\mtcl Q}$ for  a  partial order with the following properties.
\begin{enumerate} 
\it $\dot{\mtcl Q}$ is forced to be of the form $\Coll(\o_1, \mtcl P)\ast\dot{\mtcl R}$ where $\dot{\mtcl R}$ is a 
$\Coll(\o_1, \mtcl P)$--name for a forcing of the form $\dot{\mtcl R}_0\ast\dot{\mtcl R}_1$, 
where $\dot{\mtcl R}_0$ has the countable chain condition and $\dot{\mtcl R_1}$ is forced to be an $\MRP$--poset. 
\it Suppose $b_0$, $b_1\in \RO(\mtcl P)$ are incompatible, $\bool{B}$ is a complete Boolean algebra, and 
$k_\epsilon:\RO(\mtcl P\ast\dot{\mtcl Q})\restr b_\epsilon\into\bool{B}$ are complete homomorphisms for 
$\epsilon\in\{0, 1\}$. Then $\bool{B}$ collapses $\o_1$.
  \end{enumerate}

\end{lemma}

Our second freezing poset comes from \cite[Section 1]{Velickovic}, where the following is proved, 
using a result of Todor\v{c}evi\'{c} from \cite{Todor}.

\begin{lemma}\label{lemma-velickovic}
There is a sequence $((K^\x_0, K^\x_1)\,:\,\x<\o_1)$ of colourings of $[\k]^2$, for $\kappa=\cf(2^{\al_0})$, 
with the property that in any $\o_1$--preserving outer model in which $\av\k\av=\al_1$, if $B\sub\o_1$, then 
there is a c.c.c.\ partial order $\mtcl R$ forcing the existence of $\al_1$--many decompositions 
$\k=\bigcup_{n<\o}X^\x_n$, for $\x<\o_1$, such that
for all $\x <\o_1$:
\begin{itemize} 
\item
for some fixed $i_\x=0,1$, $X^\x_n$ is $K^\x_{i_\x}$--homogeneous for all $n<\o$;
\item
$\x\in B$ if and only if $i_\x=0$.
\end{itemize}
\end{lemma}

\begin{proposition}\label{freeze1.5}
Given a forcing notion $\mtcl P$,  there is $\mtcl P$--name $\dot{\mtcl Q}$ for  a forcing notion with the following properties.
\begin{enumerate} 
\it Letting $\mu = \av\mtcl P\av+\cf(2^{\al_0})$, $\dot{\mtcl Q}$ is forced to be a forcing of the form 
$\Coll(\o_1,\, \mu)\ast\dot{\mtcl R}$, where $\dot{\mtcl R}$ is a $\Coll(\o_1,\, \mu)$--name for a c.c.c.\  forcing.
\it $\mtcl P\ast\dot{\mtcl Q}$ $\SSP$--freezes $\mtcl P$, as witnessed by the inclusion map. In fact, if 
$b_0$, $b_1\in \RO(\mtcl P)$ are incompatible, $\bool{B}$ is a complete Boolean algebra, and 
$k_\epsilon:\RO(\mtcl P\ast\dot{\mtcl Q})\restr b_\epsilon\into\bool{B}$ is a complete homomorphism for 
$\epsilon\in\{0, 1\}$, then $\bool{B}$ collapses $\o_1$.
  \end{enumerate}
\end{proposition}

\begin{proof}
 Let us work in $V^{\mtcl P\ast \Coll(\o_1,\,\mu)}$. Let $\dot B_G$ be a subset of $\o_1$ coding the generic filter for 
 $\mtcl P$ in some canonical way, let $\vec K=((K^\x_0, K^\x_1)\,:\,\x<\o_1)$ be a sequence of colourings of $[\k]^2$, for 
 $\kappa=\cf^V(2^{\al_0})$, as given by Lemma \ref{lemma-velickovic}, and let $\mtcl R$  be a c.c.c.\ partial order forcing 
 the existence of $\al_1$--many decompositions $\k=\bigcup_{n<\o}X^\x_n$ such that for all 
 $\x <\o_1$,
 \begin{itemize}
 \item there is $i_\x$ such that
 $[X^\x_n]^2\sub K^\x_{i_\x}$ for all $n<\o$
and 
\item $\x\in B$ if and only if for all $n<\o$, 
 $[X^\x_n]^2\sub K^\x_0$.
 \end{itemize}
 
Let $A_G$ be a subset of $\omega_1$ which canonically codes $B_G$ and $\bp{(X^\x_n)_{n<\o}:\x<\o_1}$. 
If $M$ is any outer model in which $\o_1^V$ has not been collapsed, then $B_G$ is the unique $B\sub\o_1$ 
for which there is,  in $M$, a set $A\sub\o_1$ coding $B$ together with $\al_1$--many decompositions  
$\bp{(X^\x_n)_{n<\o}:\x<\o_1}$ of $\k$ such that for all $n<\o$ and $\x <\o_1$, $\x\in B$ if and only if 
$[X^\x_n]^2\sub K^\x_0$. Indeed, if $B'\in M$ were another such set, as witnessed by $A'\sub M$, 
$\x\in B\Delta B'$, and $\al_1$--many decompositions  
$\bp{(Y^\x_n)_{n<\o}:\x<\o_1}\in M$ 
of $\k$,
%such that for all $\x <\o_1$, $\x\in B'$ if and only if 
%for all $n<\o$, $[Y^\x_n]^2\sub K^\x_0$, 
then there would be some $n$ and $m$ such that $X^\x_n\cap Y^\x_m$ 
has more than one element, and is in fact uncountable. But then, for every 
$s\in[X^\x_n\cap Y^\x_m]^2$, we would have that $s$ is both in $K^\x_0$ and $K^\x_1$, which is impossible.

 Finally, it is immediate to see that there is a $\Sigma_1$--formula $\varphi(x, y, z)$ such that $\varphi(G, A_G, p)$ expresses the above property of $G$ and $A_G$ over $(H_{\omega_2}; \in)^{M}$ for any $M$ as above, for $p=\vec K$. 
\end{proof}

The following principle, as well as Lemma \ref{psiAC-measurable}, are due to Woodin (\cite{WoodinBOOK}). 

\begin{definition} $\psi_{\AC}$ is the following statement: Suppose $S$ and $T$ are stationary and co-stationary subsets of $\o_1$. Then there are $\a<\o_2$ and a club $C$ of $[\a]^{\al_0}$ such that for every $X\in C$, $X\cap\o_1\in S$ if and only if $\ot(X)\in T$.
\end{definition}  

The $\AC$--subscript in the above definition hints at the fact that $\psi_{\AC}$ implies $L(\mtcl P(\o_1))\models\AC$ 
(which comes from an argument similar to the one in the proof of Lemma \ref{freeze2}).

Our third freezing poset is essentially the following forcing for adding a suitable instance of $\psi_{AC}$ by initial segments, using a measurable cardinal $\kappa$ (i.e., turning $\kappa$ into an ordinal $\alpha$ as required by the conclusion of $\psi_{AC}$).  

\begin{lemma}\label{psiAC-measurable}
Let $\kappa$ be a measurable cardinal and let $S$ and $T$ be stationary and co-stationary subsets of $\o_1$. 
Let $\mtcl Q=\mtcl Q_{\k, S, T}$ be the set, ordered by reverse inclusion, of all countable  $\sub$--continuous 
$\in$--chains  $p=(M^p_i)_{i\leq\n}$ of countable elementary substructures of $H_\k$ such that for 
every $i\leq\n$, $M^p_i\cap\o_1\in S$ if and only if $\ot(M^p_i\cap\k)\in T$. 
 
 \begin{enumerate}

\item $\mtcl Q$ is ${<}\omega_1$--semiproper, preserves Suslin trees, and does not add new reals.

\item if $G$ is $\mtcl Q$--generic over $V$ and $M^G_i= M^p_i$ whenever $p\in G$ and $i\in \dom(p)$, then 
$(M^G_i)_{i<\o_1}$ is the $\sub$--increasing enumeration of a club of $[H_\k^V]^{\al_0}$ such that for every limit ordinal $i<\o_1$, $M^G_i\cap\o_1\in S$ if and only if $\ot(M^G_i\cap\k)\in T$. 
\end{enumerate}
\end{lemma}

\begin{proof} The proofs of all assertions, except the fact that $\mtcl Q$ preserves Suslin trees, are standard. 
For the reader's convenience,  we sketch the proof that $\mtcl Q$ is ${<}\o_1$--semiproper,  though. We 
also prove that $\mtcl Q$ preserves Suslin trees. 

We get the ${<}\o_1$--semiproperness of $\mtcl Q$ as follows: the main point is that if $\mtcl U$ is a normal 
measure on $\kappa$, $N$ is an elementary submodel of some $H_\theta$ such that $\mtcl U\in N$ and 
$\av N\av<\k$, and $\eta\in \bigcap(\mtcl U\cap N)$, then $$N[\eta]:=\{f(\eta)\,:\, f\in N,\,f\mbox{ a function with 
domain }\kappa\}$$ is an elementary submodel of $H_\theta$ such that $N\cap\kappa$ is a proper initial segment 
of $N[\eta]\cap\kappa$ ($\eta\in N[\eta]$ is above every ordinal in $N\cap\kappa$, and any $\gamma\in N[\eta]\cap\eta$ 
is of the form $f(\eta)$, for some regressive function $f:\kappa\into\kappa$ in $N$ which, by normality of $\mtcl U$, 
is constant on some set in $\mtcl U\cap N$). If $N$ is countable, then by iterated applications of this construction, 
taking unions at nonzero limit stages, one  obtains a $\sub$--continuous and $\sub$--increasing sequence 
$(N_\nu)_{\nu<\omega_1}$ of elementary submodels of $H_\theta$ such that $N_0=N$ and such that 
$N_{\n'}\cap\kappa$ is a proper end--extension of $N_\n\cap\kappa$ for all $\n<\n'<\o_1$. 
Since $(\ot(N_\n\cap\kappa)\,:\,\n<\omega_1)$ is then a strictly increasing and continuous sequence 
of countable ordinals, we may find, by stationarity of $S$ and $\o_1\setminus S$, some $\n<\o_1$ 
such that $N\cap\omega_1\in S$ if and only if $\ot(N_\nu)\in T$. 

This observation yields the ${<}\omega_1$--semiproperness of $\mtcl Q$ since, given $\a<\o_1$ and a $\in$--chain $(N^\xi)_{\xi<\alpha}$ of countable elementary submodels of some $H_\theta$ such that $\mtcl U$, $S$, $T\in N^0$, one can run the above  construction for each $N^\xi$ by working inside $N^{\xi+1}$. 

The preservation of Suslin trees can be proved by the following version of the argument in \cite{Miyamoto-Yorioka} for showing that $\MRP$--forcings preserve Suslin trees. Suppose $U$ is a Suslin tree, $\dot A$ is a $\mtcl Q$--name for a maximal antichain of $U$, and $N$ is a countable elementary submodel of some large enough $H_\t$ containing $U$, $\dot A$, and all other relevant objects. By moving to an $\o_1$--end-extension of $N$ if necessary as in the proof of ${<}\o_1$--semiproperness, we may assume that $N\cap\o_1\in S$ if and only if $\ot(N\cap\k)\in T$.  Let $(u_n)_{n<\o}$ enumerate all nodes in  $U$ of height $N\cap\o_1$. Given a condition $p\in\mtcl Q$ in $N$, we may build an $(N, \mtcl Q)$--generic sequence $(p_n)_{n<\o}$ of conditions in $N$ extending $p$ and such that for every $n$ there is some $v\in U$ below $u_n$ such that $p_{n+1}$ forces $v\in \dot A$. By the choice of $N$, we have in the end that $p^\ast=\bigcup_n p_n\cup\{(N\cap\o_1, N\cap \k)\}$ is a condition in $\mtcl Q$ extending $p$. But, by construction of $(p_n)_{n<\o}$, $p^\ast$ forces $\dot A$ to be contained in the countable set $U\cap N$: If $u\in \dot A\setminus N$, and $u_n$ is the unique node of height $N\cap\o_1$ such that $u_n\leq_U u$, then $u_n$ is forced to extend some node in $\dot A$ of height less than $N\cap\o_1$, which is a contradiction since $\dot A$ was supposed to be a name for an antichain. 

It remains to show how to find $p_{n+1}$ given $p_n$. Working in $N$, we first extend $p_n$ to some $p_n'$ in some suitable dense subset $D\in N$ of $\mtcl Q$. 
Since $U$ is a Suslin tree, we have that $u_n$ is totally $(U, N)$--generic, in the sense that for every antichain $B$ of $U$ in $N$, $u_n$ extends a unique node in $B$. Also, the set $E\in N$ of $u\in U$ for which there is some $v\in U$ below $u$ and some $q\in\mtcl Q$ extending $p_n'$ and forcing that $v\in\dot A$ is dense in $U$. It follows that we may find  some $u\in E\cap N$ below $u_n$, as witnessed by some $q\in \mtcl Q\cap N$ and some $v\in U\cap N$. But then we may let $p_{n+1}=q$. 
\end{proof}

Given a measurable cardinal $\kappa$, we will call a partial order $\mtcl R$ a \emph{$\psi_{\AC}^\kappa$--poset} if $\mtcl R$ is of the form $\mtcl Q_{\kappa, S, T}$ for stationary and co-stationary subsets $S$, $T$ of $\omega_1$.

\begin{proposition}\label{freeze2}

Given a forcing notion $\mtcl P$  and a measurable cardinal $\kappa$, there is $\mtcl P$--name $\dot{\mtcl Q}$ for  a forcing notion such that 
\begin{enumerate} 
\it $\dot{\mtcl Q}$ is forced to be of the form $\Coll(\o_1, \mtcl P)\ast\dot{\mtcl R}$, where $\dot{\mtcl R}$ is a $\Coll(\o_1, \mtcl P)$--name for a $\psi^\kappa_{AC}$--poset, and 
\it $\mtcl P\ast\dot{\mtcl Q}$ $\SSP$--freezes $\mtcl P$, as witnessed by the inclusion map.
  \end{enumerate}
\end{proposition}

\begin{proof} Let $(S_\a\,:\,\a<\o_1)$ be a partition of $\omega_1$ into stationary sets and let $T$ be a stationary 
and co-stationary subset of $\o_1$. Working in $V^{\mtcl P\ast\Coll(\o_1,\,\mtcl P)}$, let $B_{\dot G}\neq\emptyset$ 
be a subset of $\omega_1$ coding $\dot G$ in  a canonical way. We may assume that $B_{\dot G}\neq\omega_1$. 
Let $\dot{\mtcl R}$ be $\mtcl Q_{\kappa, S, T}$ for $S= \bigcup_{\alpha\in B_{\dot G}}S_\alpha$. 
By Lemma \ref{psiAC-measurable}, $\dot{\mtcl R}$ adds a club $C_{\dot G}$ of $[\kappa]^{\al_0}$ 
with the property that for each $X\in C_{\dot G}$, $X\cap \o_1\in \bigcup_{\alpha\in B_{\dot G}}S_\alpha$ if and only 
if $\ot(X)\in T$.  Now it is easy to see that $\mtcl P\ast\dot{\mtcl R}$ codes $\dot G$ in an absolute way in the sense 
of Lemma \ref{freezability-sufficient}. The main point is that if $H$ is a 
$(\mtcl P\ast\Coll(\o_1,\,\mtcl P))\ast\dot{\mtcl R}$--generic filter, $G=H\cap\mtcl P$, and $M$ is an outer model 
such that every stationary subset of $\o_1$ in $V[H]$ remains stationary in $M$, then in $M$ there is no $B'\sub\o_1$ 
such that $B'\neq B_G$ and such that there is a club  $C$ of $[\kappa]^{\al_0}$ with the property that for all 
$X\in C$, $X\cap\o_1\in\bigcup_{\alpha\in B'}S_\alpha$ if and only if $\ot(X)\in T$. Otherwise, if $\a\in B'\Delta B_G$, 
then there would be some $X\in C\cap C_G$ such that $X\cap \omega_1\in S_\alpha$. 
But then we would have that $\ot(X)$ is both in $T$ and in $\omega_1\setminus T$. 
\end{proof}

Let us move on now to our fourth freezing poset.   

Given cardinals $\m<\l$ with $\m$ regular, let $$S^\l_\m=\{\x<\l\,:\,\cf(\x)=\m\}$$ Let $\vec S = (S_\a)_{\a<\o_1}$ 
be a sequence of pairwise disjoint stationary subsets of $S^{\o_2}_\o$,  let $U\sub S^{\o_3}_\o$ be such that both 
$U$ and  $S^{\o_3}_\o\setminus U$ are stationary, and let $B\sub\o_1$. Then $\mtcl S_{\vec S, U, B}$ is 
the partial order, ordered by end--extension, consisting of all strictly $\sub$--increasing and $\sub$--continuous 
sequences $(Z_\n)_{\n\leq\n_0}$, for some $\n_0<\o_1$, such that for all $\n\leq\n_0$ and all $\a<\o_1$,  

\begin{itemize}

\it $Z_\n\in[\o_3]^{\al_0}$, and \it if $\mbox{sup}(Z_\n\cap\o_2)\in S_\a$, then $\mbox{sup}(Z_\n)\in U$ if and only if $\a\in B$. \end{itemize}

The proof of the following lemma appears in \cite{FOREMAGISHELAH} essentially.\footnote{See also the argument in the proof of Proposition \ref{freeze3} that $\dot{\mtcl Q}$ satisfies the $S$--condition.} 

\begin{lemma}\label{freezing-poset-3}
Let $\vec S= (S_\a)_{\a<\o_1}$ be a sequence of pairwise disjoint stationary subsets of $S^{\o_2}_\o$, let $U\sub S^{\o_3}_\o$ be such that both $U$ and  $S^{\o_3}_\o\setminus U$ are stationary, and let $B\sub\o_1$. Then $\mtcl S_{\vec S, U, B}$ preserves stationary subsets of $\o_1$, as well as the stationarity of all $S_\a$, and forces the existence of strictly $\sub$--increasing and $\sub$--continuous enumeration $(Z_\n\,:\,\n<\o_1)$ of a club of $[\o_3^V]^{\al_0}$ such that for all $\n$, $\a<\o_1$,  if $\mbox{sup}(Z_\n\cap\o_2^V)\in S_\a$, then $\mbox{sup}(Z_\n)\in U$ if and only if $\a\in B$.
\end{lemma}

\begin{proposition}\label{freeze3} For every partial order $\mtcl P$ and for all cardinals $\k^1>\k^0\geq\d\geq \av\mtcl P\av$, if $\Vdash_{\mtcl P\ast\Coll(\o_1,\,\d)}\k^0=\o_2$, $\Vdash_{\mtcl P\ast\Coll(\o_1,\,\d)}\k^1=\o_3$, $(S_\a)_{\a<\o_1}\in V$  is a  sequence of pairwise disjoint stationary subsets of $S^{\k^0}_\o$, and $U\sub S^{\k^1}_\o$ is a a stationary set in $V$ such that $S^{\k^1}_\o\setminus U$ is also stationary in $V$, then there is a $\mtcl P\ast \Coll(\o_1,\,\d)$--name $\dot B$  for a subset of $\o_1$ such that 

\begin{enumerate}
\item $\mtcl P$ forces $\Coll(\o_1,\, \d)\ast\dot{\mtcl S}_{\vec S, U, \dot B}$ to have the $S$--condition, and such that

\item $\mtcl P\ast (\Coll(\o_1,\, \d)\ast\dot{\mtcl S}_{\vec S, U, \dot B})$ $\SSP$--freezes $\mtcl P$, as witnessed by the inclusion map $$i:\mtcl P\into \mtcl P\ast (\Coll(\o_1,\,\d)\ast\dot{\mtcl S}_{\vec S, U, \dot B})$$
\end{enumerate}
\end{proposition}

\begin{proof}
Working in $V^{\mtcl P\ast \Coll(\o_1,\,\d)}$,  let $(S_\a)_{\a<\o_1}\in V$ and $U\in V$ be as stated, and let $B_{\dot G}$ be a subset of $\o_1$ coding the generic filter $G$ for $\mtcl P$ in a canonical way. Let $\dot{\mtcl R}$ be a $\mtcl P$--name for $\Coll(\o_1,\,\d)\ast \dot{\mtcl S}_{\vec S, U, B_{\dot G}}$.

\begin{claim}
$\mtcl P$ forces that $\dot{\mtcl R}$ has the $S$--condition. 
\end{claim}

\begin{proof} Since $\Coll(\o_1,\,\d)$ has the $S$--condition, it suffices to prove that $\mtcl P\ast\Coll(\o_1,\,\d)$ 
forces $\dot{\mtcl S}_{\vec S, U, \dot B}$ to have the $S$--condition. Let us work in $V^{\mtcl P\ast \Coll(\o_1,\,\d)}$. 
%Let $p\in\dot{\mtcl R}$. 
Let $\s$ be the following strategy for player II in $\mtcl G^{\dot{\mtcl R}}$: 
Whenever it is her turn to play, player II will alternate between the following courses of action (a), (b) 
(i.e., she will opt for (a) or (b) depending on the parity of the finite set 
$\{k<\av\eta\av\,:\,\av \textsf{succ}_T(\eta\restr k)\av>1\}$, with the notation used in Definition \ref{s-cond-def}).

\begin{itemize}

\item[(a)] Player II chooses $\k_\eta=\k^0$, $\textsf{succ}_T(\eta)=\{\eta^\smallfrown\la\a\ra\,:\,\a<\k^0\}$, and 
$(p_{\eta^\smallfrown\la\a\ra})_{\a<\k^0}$ where, for each $\a<\k^0$, $p_{\eta^\smallfrown\la\a\ra}$ is a condition extending 
$p_\eta$ and such that $\a\in \bigcup\range(p_{\eta^\smallfrown\la\a\ra})$.

\item[(b)]  Player II chooses $\k_\eta=\k^1$, $\textsf{succ}_T(\eta)=\{\eta^\smallfrown\la\a\ra\,:\,\a<\k^1\}$, and 
$(p_{\eta^\smallfrown\la\a\ra})_{\a<\k^1}$ where, for each $\a<\k^1$, $p_{\eta^\smallfrown\la\a\ra}$ is a condition 
extending $p_\eta$ and such that $\a\in \bigcup\range(p_{\eta^\smallfrown\la\a\ra})$.

\end{itemize}

Let now $T$ be the tree built along a run of $\mtcl G^{\dot{\mtcl R}}$ in which player II has played according to $\s$, 
let $T'$ be a subtree of $T$ such that $\av\textsf{succ}_{T'}(\eta)\av=\av\textsf{succ}_T(\eta)\av$ for every $\eta\in T'$, 
and let $N$ be a countable elementary substructure of some large enough $H_\t$ containing all relevant objects 
(which includes our run of $\mtcl G^{\dot{\mtcl R}}$ and $T'$), such that $\mbox{sup}(N\cap\k^0)\in S_0$, and such that 
$\mbox{sup}(N\cap \k^1)\in U$ if $0\in B_{\dot G}$ and $\mbox{sup}(N\cap \k^1)\notin U$ if $0\notin B_{\dot G}$. 

Such an $N$ can be easily found (s.\ \cite{FOREMAGISHELAH}):
Indeed, suppose, for concreteness, that $0\in B_{\dot G}$. Then,  letting $F:[H_\t]^{{<}\omega}\into H_\t$ 
be a function generating the club of countable elementary submodels of $H_\t$ containing all relevant objects, 
we  may find, using the stationarity of $U$, an ordinal $\a\in U$ such that the closure $X_0$ of $[\a]^{{<}\omega}$ 
under $F$ is such that $X_0\cap \k^1 = \alpha$. We may of course assume that $\a>\k^0$. Since $\cf(\a)=\omega$, 
we may pick a countable cofinal subset $Y$ of $\a$.  Using now the stationarity of $S_0$, we may find $\b\in S_0$ 
with the property that the closure $X_1$ of $[\b\cup Y]^{{<}\omega}$ under $F$ is such that $X_1\cap\kappa^0=\beta$. 
Since $\cf(\b)=\o$, we may now pick a countable subset $Z$ of $\b$. But then, letting $N$ be the closure of
 $Y\cup Z$ under $F$, we have that $\mbox{sup}(N\cap\k^0)=\b$ and $\mbox{sup}(N\cap\k^1)=\alpha$, and so 
 $N$ is as desired.

Letting $(p_\eta)_{\eta\in T'}$ be the tree of $\dot{\mtcl R}$--conditions corresponding to $T'$, it is now easy to find a 
cofinal branch $b$ through $T'$ such that for all $n<\o$, $b\restr n\in N$, and such that 
\[
\mbox{sup}(\bigcup_{n<\o}(\cup\range(p_{b\restr n})\cap\k^0))=  \mbox{sup}(N\cap\k^0)
\]
and 
\[
\mbox{sup}(\bigcup_{n<\o}(\cup\range(p_{b\restr n})))=  \mbox{sup}(N\cap\k^1).
\] 
Let 
\[
\n=\mbox{sup}\{\dom(p_{b\restr n})\,:\,n<\o\}\footnote{Incidentally, note that we cannot guarantee that $\n= N\cap\o_1$.}
\]
and 
\[
X=\bigcup_{n<\o}(\cup\range(p_{b\restr n})\cap\k^1).
\] 
It follows now that, letting $p_b=\bigcup_n p_{b\restr n}\cup\{\la\n, X\ra\}$, $p_b$ is a condition in 
$\dot{\mtcl R}$ forcing that $p_{b\restr n}$ is in the generic filter for all $n$.

\end{proof}

Going back to $V$, the proof that $\mtcl P\ast (\Coll(\o_1,\, \d)\ast\dot{\mtcl S}_{\vec S, U, \dot B})$ $
\SSP$--freezes $\mtcl P$ (as witnessed by the inclusion map) is very much like the proofs of Propositions \ref{freeze1} 
and \ref{freeze2}. Suppose $H$ is a generic filter for  $\mtcl P\ast (\Coll(\o_1,\, \d)\ast\dot{\mtcl S}_{\vec S, U, \dot B})$, 
$G=H\cap\mtcl P$, and $M$ is any outer model such that every stationary subset of $\o_1$ in $V[H]$ remains stationary 
in $M$. Suppose, towards a contradiction, that in $M$ there is some subset $B'\neq B_G$ of $\omega_1$ 
for which there is an $\sub$--increasing and $\sub$-continuous enumeration $(Z'_\n\,:\,\n<\o_1)$ of a club of 
$[\k^1]^{\al_0}$ such that for all $\n$, $\a<\o_1$,  if $\mbox{sup}(Z'_\n\cap\k^0)\in S_\a$, then $\mbox{sup}(Z'_\n)\in U$ 
if and only if $\a\in B'$. If $\a\in B'\Delta B_G$, there is some $\n$ such that $Z_\n=Z'_\n$ and 
$\mbox{sup}(Z_\n\cap \k^0)\in S_\a$. But then we have both $\mbox{sup}(Z_\n)\in U$ and $\mbox{sup}(Z_\n)\notin U$.

Finally, the existence of a $\S_1$ definition -- with $p=(\l, \vec S, U)$ as parameter  -- as required by 
Lemma \ref{freezability-sufficient} is easy. 
 \end{proof}

\subsubsection{Iteration Lemmas}\label{iterability}

%We start by proving that $\PR$, and each of its restrictions mentioned in Theorem \ref{main-thm-suitable}, is $\o_1$--suitable.

We need the following preservation lemmas, due to Shelah (\cite[III, resp. VI]{SHEPRO}, 
see also \cite{Goldstern}).

\begin{lemma}\label{presproper} Suppose $\rho<\o_1$ is an indecomposable ordinal and $$(\mtcl P_\a, \dot{\mtcl Q}_\b\,:\,\a\leq\g,\,\b<\g)$$ is a countable support iteration such that for all $\b<\g$, $$\Vdash_{\mtcl P_\b}\dot{\mtcl Q}_\b\in \rho\mbox{--}\PR$$ Then $\mtcl P_\g\in \rho\mbox{--}\PR$.

\end{lemma}

\begin{lemma}\label{presomegaomegabounding} Suppose $\rho<\o_1$ is an indecomposable ordinal and $$(\mtcl P_\a, \dot{\mtcl Q}_\b\,:\,\a\leq\g,\,\b<\g)$$ is a countable support iteration such that for all $\b<\g$, $$\Vdash_{\mtcl P_\b}\dot{\mtcl Q}_\b\in (\rho\mbox{--}\PR)\cap\,^\o\o\textsf{--bounding}$$ Then $\mtcl P_\g\in (\rho\mbox{--}\PR)\cap\,^\o\o\textsf{--bounding}$.
\end{lemma}

We will also use the following preservation result due to Miyamoto.

\begin{lemma}\label{pressuslintreepres} Suppose $\rho<\o_1$ is an indecomposable ordinal and $$(\mtcl P_\a, \dot{\mtcl Q}_\b\,:\,\a\leq\g,\,\b<\g)$$ is a countable support iteration such that for all $\b<\g$, $$\Vdash_{\mtcl P_\b}\dot{\mtcl Q}_\b\in (\rho\mbox{--}\PR)\cap\STP$$ Then $\mtcl P_\l\in (\rho\mbox{--}\PR)\cap\STP$.

\end{lemma}

Shelah defines a certain variant of the notion of countable support iteration, which he calls \emph{revised countable 
support (RCS) iteration}. Variants of the notion of RCS iteration have been proposed by Miyamoto and others 
(for example a detailed account of RCS-iterations in line with Donder and Fuchs' approach is given 
in~\cite{VIAAUDSTEBOOK}).\footnote{It is not always clear whether these notions are equivalent 
in any reasonable sense.}
In the following, any mention of revised countable support iteration will refer to either Shelah's or Miyamoto's version. 

The first preservation result involving RCS iterations we will need is the following lemma, proved in \cite[XI]{SHEPRO}. 

\begin{lemma}\label{preservation-S-cond}
Suppose $\la\mtcl P_\a, \dot{\mtcl Q}_\b\,:\,\a\leq\g,\,\b<\g\ra$ is an RCS iteration such that the following holds for all $\b<\g$.

\begin{enumerate}

\item If $\b$ is even, $\Vdash_{\mtcl P_\b}\dot{\mtcl Q}_\b=\textsf{Coll}(2^{\av\mtcl P_\b\av},\,\o_1)$.
\item If $\b$ is odd, $\Vdash_{\mtcl P_\b}\dot{\mtcl Q}_\b\mbox{ has the $S$--condition}$.
\end{enumerate}

Then $\mtcl P_\g$ has the $S$--condition. \end{lemma}

The following is a well--known result of Shelah.

\begin{lemma}\label{pres-semiproper} Suppose  $\rho<\o_1$ is an indecomposable ordinal and $$(\mtcl P_\a, \dot{\mtcl Q}_\b\,:\,\a\leq\g,\,\b<\g)$$ is a revised countable support iteration such that for all $\b<\g$, $$\Vdash_{\mtcl P_\b}\dot{\mtcl Q}_\b\in \rho\mbox{--}\SP$$ Then $\mtcl P_\g\in \rho\mbox{--}\SP$.
\end{lemma}

We will also need the following lemmas due to Miyamoto \cite{Miyamoto,miyamoto-suslin-sp}.

\begin{lemma}\label{pres-semiproper-bounding}  Suppose $\CH$ holds, $\rho<\o_1$ is an indecomposable ordinal, and $$(\mtcl P_\a, \dot{\mtcl Q}_\b\,:\,\a\leq\g,\,\b<\g)$$ is a revised countable support iteration such that for all $\b<\g$, $$\Vdash_{\mtcl P_\b}\dot{\mtcl Q}_\b\in (\rho\mbox{--}\SP)\cap\,^\o\o\textsf{--bounding}$$ Then $\mtcl P_\g\in (\rho\mbox{--}\SP)\cap\,^\o\o\textsf{--bounding}$.
\end{lemma}

\begin{lemma}\label{miyamotosuslinsp}  Suppose $\rho<\o_1$ is an indecomposable ordinal and $$(\mtcl P_\a, \dot{\mtcl Q}_\b\,:\,\a\leq\g,\,\b<\g)$$ is a revised countable support  iteration such that for all $\b<\g$, $$\Vdash_{\mtcl P_\b}\dot{\mtcl Q}_\b\in(\rho\mbox{--}\SP)\cap\STP$$ Then $\mtcl P_\g\in (\rho\mbox{--}\SP)\cap\STP$.
\end{lemma}

\subsubsection{$\omega_1$--suitable classes}\label{omega1suitability}

\begin{lemma}\label{pr-suitable} The following classes are $\o_1$--suitable with respect to the theory
$\MK$+`\emph{$\omega_1$ is the least uncountable cardinal}' + `\emph{$\rho$ is a countable indecomposable ordinal}' for every indecomposable ordinal $\rho<\o_1$.

\begin{enumerate}

\item $\rho\mbox{--}\PR$

\item  ${<}\o_1\mbox{--}\PR$
\item $\PR\cap \STP$ 
\item $\PR\cap\,^\o\o\textsf{--bounding}$ 
\item $\PR\cap\STP\cap\,^\o\o\textsf{--bounding}$ 
\end{enumerate}
\end{lemma}

\begin{proof}
Given any of these  classes $\Gamma$, all conditions in the definition of $\o_1$--suitable class -- 
except for the fact that $\Gamma$ has the $\Gamma$--freezability property -- are clearly satisfied for 
$\Gamma$. In particular, $\Gamma$ is defined by an absolutely $\Sigma_2$ property by Fact~\ref{fac:sigma2proper},
and 
is closed under preimages by complete injective homomorphisms, two--step iterations and products, and contains all countably closed forcings 
by Fact~\ref{fac:Birkhoffproper}.
The weak iterability property follows from Lemmas \ref{presproper}, \ref{presomegaomegabounding} 
and \ref{pressuslintreepres}: the winning strategey for player II is to play the countable support limit at all limit stages
(notice that at stages of cofinality $\omega_1$ this limit is the direct limit).  
 As to the freezability property, it turns out that $\Gamma$ has in fact the $\SSP$--freezability property. This 
 follows immediately from Proposition \ref{freeze1} together with Lemma \ref{MRP} for $\PR$,  as well as for the classes 
 in (3), (4) and (5), and from Proposition \ref{freeze1.5} for $\rho\mbox{--}\PR$, for any given indecomposable 
 ordinal $\rho<\o_1$ such that $\rho>1$, and for ${<}\o_1\mbox{--}\PR$.
\end{proof}

We move on now to our first class not contained in $\PR$.

\begin{lemma}  $S\textsf{--cond}$ is $\o_1$--suitable with respect to 
$\MK$+ `\emph{$\omega_1$ is the least uncountable cardinal}'.
\end{lemma}

\begin{proof}
Except for the freezability condition, all conditions in the definition of $\o_1$--suitable class are clearly satisfied by 
$S\textsf{--cond}$:
%\footnote{The closure of  $S\textsf{--cond}$ under complete subalgebras follows of course directly 
%from the definition of $S\textsf{--cond}$  as the class of complete Boolean algebras $\bool{B}$ such that $\bool{B}$ is 
%a complete suborder of some complete Boolean algebra satisfying the $S$--condition. On the other hand, it is not clear to 
%us whether a complete Boolean subalgebra of a complete Boolean algebra with the $S$--condition necessarily has the 
%$S$--condition.}. 
$S\textsf{--cond}$ is defined by an absolutely $\Sigma_2$ property
and 
is closed under preimages by complete injective homomorphisms, two--step iterations and products, and contains all countably closed forcings 
by Fact~\ref{fac:sigma2Scond}.
The iterability condition follows immediately from Lemma \ref{preservation-S-cond}: 
the winning strategy for player II is to play the revised countable support limit at all limit stages
(notice that at stages of cofinality $\omega_1$ this limit is the direct limit), and to play at all non-limit stages
$\alpha+2n$ the algebra $\bool{B}_{\alpha+2n}=\bool{B}_{\alpha+2n-1}\ast\dot{\bool{C}}$, where 
$\dot{\bool{C}}$ is a $\bool{B}_{\alpha+2n-1}$--name for the Boolean completion of 
$\Coll(\omega_1,2^{|\bool{B}_{\alpha+2n-1}|})$. 
As to the freezability condition, 
we have that $S\textsf{--cond}$ has in fact,  by Proposition \ref{freeze3}, the $\SSP$--freezability condition  -- 
which implies the $S\textsf{--cond}$--freezability condition by Lemma \ref{S-cond-shelah} (1).
\end{proof}

\begin{lemma}\label{ssp-and-restrictions}
 Given any indecomposable ordinal $\rho<\omega_1$, 
 each of the following classes is $\o_1$--suitable with respect to
 $\MK$+ `\emph{$\omega_1$ is the least uncountable cardinal}'+`\emph{$\rho$ is a countable indecomposable ordinal}'+ `\emph{there are class many measurable cardinals}'.

\begin{enumerate}
\item $\rho\mbox{--}\SP$
\item $(\rho\mbox{--}\SP)\cap \STP$
\item $(\rho\mbox{--}\SP)\cap\,^\o\o\textsf{--bounding}$
\item $(\rho\mbox{--}\SP)\cap\STP\cap\,^\o\o\textsf{--bounding}$ 
\end{enumerate}

Also, each of the following classes is $\o_1$--suitable with respect to the same theory.

\begin{enumerate}
\item ${<}\o_1\mbox{--}\SP$
\item $({<}\o_1\mbox{--}\SP)\cap \STP$
\item $({<}\o_1\mbox{--}\SP)\cap\,^\o\o\textsf{--bounding}$
\item $({<}\o_1\mbox{--}\SP)\cap\STP\cap\,^\o\o\textsf{--bounding}$ 
\end{enumerate}
\end{lemma}

\begin{proof}
Each of these $\Gamma$ is defined by an absolutely $\Sigma_2$ property by Fact~\ref{fac:sigma2semiproper},
and 
is closed under preimages by complete injective homomorphisms, two--step iterations and products, and contains all countably closed forcings 
by Fact~\ref{fac:Birkhoffsemiproper}.
The iterability condition for each of these classes follows from (some combination of) 
Lemmas \ref{pres-semiproper}, \ref{pres-semiproper-bounding}, and \ref{miyamotosuslinsp}: 
the winning strategey for player II is to play the revised countable support limit at all limit stages
(notice that at stages of cofinality $\omega_1$ this limit is the direct limit).
%, and to play at all non-limit stages
%$\alpha+2n$ $\bool{B}_{\alpha+2n}=\bool{B}_{\alpha+2n-1}\ast\dot{\bool{C}}$ where 
%$\dot{\bool{C}}$ is a $\bool{B}_{\alpha+2n-1}$-name for the boolean completion of 
%$\Coll(\omega_1,|\bool{B}_{\alpha+2n-1}|)$. 
The freezability condition follows  from Lemma \ref{psiAC-measurable}, together with Proposition \ref{freeze2} 
(for the case $\rho=1$, one could as well invoke Lemma \ref{MRP} together with Proposition \ref{freeze1} instead).  
\end{proof}

The standard proof, due to Shelah  (\cite{SPFA=MM}), that $\SPFA$ implies $\SSP=\SP$ actually shows the following.

\begin{proposition} $\textsf{FA}(({<}\o_1\mbox{--}\SP)\cap\STP\cap\,^\o\o\mbox{--bounding})$ implies 
$\SSP=\SP$. \end{proposition}

\begin{proof} This follows from the fact that  the natural semiproper forcing $\mtcl Q_{\mtcl P}$  
(s.\ \cite{SPFA=MM}) such that an application of $\FA_{\al_1}(\{\mtcl Q_{\mtcl P}\})$ 
yields the semiproperness of a given $\SSP$ $\mtcl P$ is in fact ${<}\o_1$--semiproper, does not add new reals, 
and preserves Suslin trees. The proof of the first two assertions is straightforward, and the preservation of 
Suslin trees follows by an argument as in the final part of the proof of Lemma \ref{psiAC-measurable}.
\end{proof} 
 
\begin{corollary} 
The following holds:
\begin{enumerate}
\item $\CFA(\SSP)$ ($=\mathsf{MM}^{+++}$) and $\CFA(\SP)$ are equivalent statements.
\item $\CFA(\SSP\cap \STP)$ and $\CFA(\SP\cap \STP)$ are equivalent statements.
\item $\CFA(\SSP\cap\,^\o\o\textsf{--bounding})$ and $\CFA(\SP\cap\,^\o\o\textsf{--bounding})$ are equivalent statements.
\item $\CFA(\SSP\cap\STP\cap\,^\o\o\textsf{--bounding})$  and $\CFA(\SP\cap\STP\cap\,^\o\o\textsf{--bounding})$ are equivalent statements.
\end{enumerate}
\end{corollary}

\subsubsection{Pairwise incompatibility of $\CFA(\Gamma)$ for $\omega_1$--suitable $\Gamma$}
\label{incompatible-category-forcing-axioms}

Each one of the incompatibilities contained in Theorem \ref{main-thm-incompatible} follows from two or more of 
the lemmas in this subsection put together.

Recall that $\bm{\delta}^1_2$ is the supremum of the the set of lengths of $\bm{\Delta}^1_2$--definable  
pre--well-orderings on $\mtbb R$. 

Also, given an ordinal $\a<\o_2$,  a function $g:\o_1\into \o_1$ is a \emph{canonical function for $\a$} if 
there is a surjection $\p:\o_1\into\a$ and a club $C\sub\o_1$ such that for all $\nu\in C$, $g(\nu)=\ot(\pi``\n)$. 
Let \emph{Club Bounding} denote the following statement: For every function $f:\o_1\into \o_1$ there is 
some $\a<\o_2$ such that  $\{\n<\o_1\,:\, f(\n)<g(\n)\}$ contains a club  whenever $g$ is a canonical function for $\a$.

\begin{lemma}\label{consespfa+++} ($\ZFC+\LC$) Let $\Gamma$ be any $\o_1$--suitable class such that 
$\Gamma\sub \PR$. If $\CFA(\Gamma)$ holds, then 

\begin{enumerate}
\item $\bm{\delta}^1_2<\o_2$ and 
\item Club Bounding fails.
\end{enumerate}
\end{lemma}

\begin{proof}  We know that $\bool{U}^{\Gamma}_\delta$, for any $2$--superhuge cardinal $\d$, is in $\Gamma$, collapses $\o_2^{V}$ to $\al_1$ and, by Theorem \ref{consistency-of-CFA}, forces $\CFA(\Gamma)$ . Also, using our background large cardinal assumption (in fact a  proper class of 
Woodin cardinals suffices), by a result of Neeman and Zapletal \cite{Neeman-Zapletal} we have that if $\mtcl P$ 
is a proper poset and $G$ is $\mtcl P$--generic over $V$, then the identity on $L(\mtbb R)^{V}$ is an elementary 
embedding between $L(\mtbb R)^{V}$ and  $L(\mtbb R)^{V[G]}$. It follows from these two facts together that  
$V^{\bool{U}^{\Gamma}_\delta}\models\bm{\delta}^1_2<\o_2$. Since `$\bm{\delta}^1_2<\o_2$'  is expressible 
over $H_{\o_2}$, it follows now from the absoluteness theorem that $V\models\bm{\delta}^1_2<\o_2$. 

To see that Club Bounding fails in $V$, we first add  generically a function $f:\o_1\into \o_1$ by initial segments and 
then force with $\bool{U}^{\Gamma}_\delta$. It is immediate to check that, after  adding $f$, 
$$\{X\in [\a]^{\al_0}\,:\, X\cap\o_1\in\o_1,\,\ot(X)<f(X\cap\o_1)\}$$ is a stationary subset of $[\a]^{\al_0}$ for every ordinal 
$\a$. Since every proper forcing will preserve the stationarity of these sets, it follows that Club Bounding fails in 
$V^{\bool{U}^{\Gamma}_\delta}$. But then it also has to fail in $V$ by the absoluteness theorem.
\end{proof}

\begin{lemma} $\FA_{\al_1}(({<}\o_1\mbox{--}\PR)\cap\,^\o\o\textsf{--bounding})$ implies that there are no Suslin trees. \end{lemma} 

\begin{proof}
Suppose, towards a contradiction, that $T$ is a Suslin tree and the forcing axiom $\FA_{\al_1}(({<}\o_1\mbox{--}\PR)\cap\,^\o\o\textsf{--bounding})$ 
holds. Without loss of generality we may assume that $T$ is a normal Suslin tree. We have that $T$ is a c.c.c.\ forcing 
which is $^\o\o$--bounding as in fact it does not add new reals. But forcing with $T$ adds an $\o_1$--branch through 
$T$. Hence, by $\FA_{\al_1}(T)$, $T$ has an $\o_1$--branch and so it is not Suslin, which is a contradiction. 
\end{proof}

Recall that $\mathfrak d$ is the minimal cardinality of  a family $\mtcl F\sub\,^\o\o$ with the property that for every 
$f:\o\into \o$ there is some $f\in\mtcl F$ such that $g(n)<f(n)$ for a tail of $n<\o$.

\begin{lemma}\label{consespfa+++omegaomega-bounding} ($\ZFC$+$\LC$) Let $\Gamma$ be any 
$\o_1$--suitable class such that $\Gamma\sub\,^\o\o\mbox{--bounding}$.  If $\CFA(\Gamma)$ holds, 
then $\mathfrak d = \omega_1$.
\end{lemma}

\begin{proof}
We have that $\bool{U}^{\Gamma}_\delta$, for any $2$--superhuge cardinal $\d$, is in $\Gamma$, 
forces $\CFA(\Gamma)$, and collapses $(2^{\al_0})^{V}$ to $\al_1$. Hence we have that 
$V^{\bool{U}^{\Gamma}_\delta}\models \mathfrak d=\o_1$ since $\Gamma\sub\,^\o\o\mbox{--bounding}$. 
But `$\mathfrak d=\o_1$' is expressible over $H_{\o_2}$, and therefore $V\models\mathfrak d =\o_1$ by 
the absoluteness theorem.
\end{proof}

\begin{lemma}\label{FAPRcapSTPvsd}
$\FA_{\al_1}(({<}\o_1\mbox{--}\PR)\cap\STP)$ implies $\mathfrak d>\omega_1$.
\end{lemma}

\begin{proof} This is immediate since Cohen forcing, being countable, preserves Suslin trees.\end{proof}

Recall that the Strong Reflection Principle ($\SRP$) is the following assertion: For every set $X$ such that 
$\o_1\sub X$ and every $S\sub [X]^{\al_0}$ there is a strong reflecting sequence $(x_i)_{i<\o_1}$ for $S$, i.e., 
$x_i\in [X]^{\al_0}$, $(x_i)_{i<\o_1}$ is strictly $\sub$--increasing and $\sub$--continuous, and for all $i$, $x_i\notin S$ if 
and only if there is no $y\in S$ such that $x_i\sub y$ and $y\cap\o_1 =x_i\cap\o_1$. 

\begin{lemma}\label{sspomegaomegaboundingvsdelta12} 
 $\FA_{\al_1}(\SP\cap \STP\cap\,^\o\o\mbox{--bounding})$ implies $\bm{\delta}^1_2=\o_2$. 
\end{lemma} 

\begin{proof} 
We have that  $\FA_{\al_1}((\SP\cap \STP\cap\,^\o\o\mbox{--bounding})$ implies $\SRP$ since, given 
$S\sub [X]^{\al_0}$, the standard forcing for adding a strong reflecting sequence for $S$ is semiproper, does not add reals, 
and preserves Suslin trees, where the last fact follows from an argument as in the final part of the proof of 
Lemma \ref{psiAC-measurable}. Also, $\SRP$ implies $\lnot\Box_\kappa$, for every cardinal $\k\geq\o_1$, 
and hence implies that the universe is closed under sharps. Since it also implies the saturation of $\NS_{\o_1}$, 
by a classical result of Woodin (\cite{WoodinBOOK}) it implies $\bm{\delta}^1_2=\o_2$. 
\end{proof}

\begin{question} Does  $\FA_{\al_1}(\o\mbox{--}\SP)$ imply $\bm{\delta}^1_2=\o_2$?
\end{question}

\begin{lemma} Suppose  $\FA_{\al_1}(({<}\o_1\mbox{--}\SP)\cap \STP\cap\,^\o\o\mbox{--bounding})$ holds and 
there is a measurable cardinal. Then Club Bounding holds.
\end{lemma}

\begin{proof} Given a function $f:\o_1\into\o_1$ and a measurable cardinal $\k$, let $\mtcl Q^\k_f$ be the set, 
ordered by reverse inclusion, of all strictly $\sub$--increasing and $\sub$--continuous sequences $(x_i)_{i\leq\a}$, 
for $\a<\o_1$, of countable subsets of $\k$ such that for all $i$, $x_i\cap\o_1\in \o_1$ and $\ot(x_i)>f(x_i\cap\o_1)$. 
Then $\mtcl Q^\k_f$ is a ${<}\o_1$--semiproper forcing not adding reals, preserving Suslin trees, and adding a 
canonical function for $\k$ dominating $f$ on a club. All this can be proved by a straightforward variation of 
the proof of Lemma \ref{psiAC-measurable}.
\end{proof}

The following lemma is proved in \cite{Velickovic} (s.\ Proposition \ref{freeze1.5}).

\begin{lemma} $\FA_{\al_1}({<}\o_1\mbox{--}\PR)$ implies $2^{\al_0}=2^{\al_1}=\al_2$ 
\end{lemma}

$2^{\al_0}=2^{\al_1}$ follows of course already from $\MA_{\o_1}$.

A partial order $\mtbb P$ is said to have the \emph{$\sigma$--bounded chain condition} if 
$\mtbb P=\bigcup_{n<\o}\mtbb P_n$ and for each $n$ there is some $k_n<\o$ such that for every 
$X\in [\mtbb P_n]^{k_n}$ there are distinct $p$, $p'\in\mtbb P_n$ which are compatible in $\mtbb P$. Also, 
a partial order $\mtbb P$ is \emph{Knaster} if every uncountable subset of $\mtbb P$ contains an uncountable 
subset consisting of pairwise compatible conditions in $\mtbb P$.  

It is easy to see, and a well--known fact, that random forcing preserves Suslin trees. This follows from the fact that 
random forcing has the $\s$--bounded chain condition, that every forcing with the $\s$--bounded chain condition is 
Knaster, and that every Knaster forcing preserves Suslin trees.

Lemma \ref{random} follows from the above, together with the fact that random forcing is $^\o\o$--bounding and 
adds  a new real. 

\begin{lemma}\label{random} $\FA_{\al_1}(({<}\o_1\mbox{--}\PR)\cap \STP\cap\,^\o\o\textsf{--bounding})$ implies $\lnot\CH$.
\end{lemma}

\begin{lemma}\label{scnd-conses} ($\ZFC$+$\LC$)  Let $\Gamma$ be any $\o_1$--suitable class such that  
$\Gamma \sub S\textsf{--cond}$. If $\CFA(\Gamma)$ holds, then so does $\CH$.
\end{lemma}

\begin{proof} We force $\CH$ with $\s$--closed forcing, and then force $\CFA(S\textsf{--cond})$ via 
$\bool{U}^{\Gamma}_\delta$, for some $2$--superhuge cardinal $\d$. Let $V_1$ be the resulting model. 
Since $\bool{B}$ is a complete subalgebra of a poset with the $S$--condition, forcing with $\bool{B}$ over 
the $\CH$--model did not add new reals thanks to Lemma \ref{S-cond-shelah} (2). In particular, 
$V_1\models\CH$. But then $\CH$ holds in $V$ by the absoluteness theorem.
\end{proof}

It tuns out that $\CFA(\Gamma)$, where $\Gamma$ is any $\o_1$--suitable class contained in 
$\CFA(S\textsf{--cond})$, actually implies $\diamondsuit$. The proof is essentially the same as above, 
using the following recent result due to Magidor, together with the fact that if $V\sub V_1\sub W$ are models 
with the same $\o_1$ and $\vec X\in V$ is a $\diamondsuit$--sequence in $W$, then $\vec X$ is also a 
$\diamondsuit$--sequence in $V_1$. 

\begin{theorem} (Magidor) Suppose $\diamondsuit$ holds. Then there is a $\diamondsuit$--sequence 
that remains a $\diamondsuit$--sequence after any forcing with the $S$--condition.
\end{theorem}

The following well--known fact can be proved by an argument as in the final part of the proof of 
Lemma \ref{psiAC-measurable}.

\begin{fact}\label{suslin-tree-pres-sigma-closed} If $\mtcl P$ is $\s$--closed, then $\mtcl P$ preserves Suslin trees.
\end{fact}

The proof of the following lemma is like the proofs of Lemmas 
\ref{consespfa+++}, \ref{consespfa+++omegaomega-bounding}, and \ref{scnd-conses}, 
using the well--known fact that $\Add(\o_1, 1)$ adds a Suslin tree $T$. 

\begin{lemma}\label{consesmm+++pressuslintrees} ($\ZFC$ + $\LC$) Let $\Gamma$ be an $\o_1$--suitable 
class such that $\Add(\o_1, 1)\in\Gamma$ and $\Gamma\sub\STP$. If $\CFA(\Gamma)$ holds, then 
there is a Suslin tree. 
\end{lemma}

It will be convenient to consider the following families of Club--Guessing principles on $\o_1$ (s.\ \cite{AFMS}).

\begin{definition}
Let $\tau<\o_1$ be a nonzero ordinal. 

\begin{enumerate}

\item $\tau$--$\TWCG$ denotes the following statement: There is a a sequence 
$$\vec C=(C_\d\,:\,\d=\o^\tau\cdot\eta\mbox{ for some nonzero }\eta<\o_1)$$ such that 
$\av\{C_\d\cap\g\,:\,\g<\o_1\}\av\leq\al_0$ for every $\d\in\dom(\vec C)$, and such that for every club $C\sub\o_1$ 
there is some $\d\in\dom(\vec C)$ with $\ot(C_\d\cap C)=\o^\tau$. 

\item $\tau$--$\TCG$ denotes the following statement: There is a a sequence 
$$\vec C=(C_\d\,:\,\d=\o^\tau\cdot\eta\mbox{ for some nonzero }\eta<\o_1)$$ 
such that $\av\{C_\d\cap\g\,:\,\g<\o_1\}\av\leq\al_0$ for every $\d\in\dom(\vec C)$, and such that for every club 
$C\sub\o_1$ there is some $\d\in\dom(\vec C)$ with $C_\d\sub C$. 
\end{enumerate}
\end{definition}

In the above definition, $\TWCG$ and $\TCG$ stand for \emph{thin weak club--guessing} 
and \emph{thin club--guessing}, respectively.

\begin{lemma} 
Let $\tau<\o_1$ be a nonzero ordinal. Then $$\FA_{\al_1}((\o^\tau\mbox{--}\PR)\cap\STP\cap\,^\o\o\textsf{--bounding})$$ 
implies the failure of $\tau'$--$\TWCG$ for every $\tau'$ such that $\tau<\tau'<\o_1$. 
\end{lemma}

\begin{proof}
Let us consider the following natural forcing $\mtcl P_{\vec C}$ for killing an instance 
$$\vec C=(C_\d\,:\,\d=\o^\tau\cdot\eta\mbox{ for some nonzero }\eta<\o_1)$$ of $\tau'$--$\TWCG$: 
$\mtcl P_{\vec C}$ is the set, ordered by reverse end--extension, of countable closed subset $c$ of $\o_1$ 
such that $\ot(C_\d\cap c)<\o^{\tau'}$ for every $\d\in \dom(\vec C)$. It is proved in \cite{AFMS} that 
$\mtcl P_{\vec C}$ is $\o^\tau$--proper, does not add new reals, and adds a club $C\sub\o_1$ such that 
$\ot(C_\d\cap C)<\o^{\tau'}$ for every $\d\in\dom(\vec C)$. Hence, it only remains to prove that $\mtcl P_{\vec C}$ 
preserves Suslin trees. This can be shown by an argument similar to the main argument 
in the proof in \cite{Miyamoto-Yorioka} that $\MRP$--posets preserve Suslin tree. 
We present the argument here for the reader's convenience.

Suppose $U$ is a Suslin tree, $\dot A$ is a $\mtcl Q$--name for a maximal antichain of $U$, and $N$ is a 
countable elementary submodel of some large enough $H_\t$ containing $U$, $\dot A$, and all other relevant objects. 
Let $\d=N\cap\o_1$. As in  the last part of the proof of Lemma \ref{psiAC-measurable}, let $(u_n)_{n<\o}$ enumerate 
all nodes in  $U$ of height $\d$. Given a condition $c\in\mtcl P_{\vec C}\cap N$, we aim to build an 
$(N, \mtcl P_{\vec C})$--generic sequence $(c_n)_{n<\o}$ of conditions in $N$ extending $c$  such that for every $n$ 
there is some $v\in U$ below $u_n$ such that $c_{n+1}$ forces $v\in \dot A$. We will make sure that 
$C_\d\cap\bigcup_{n<\o}c_b\sub c$, which will guarantee that $c^\ast = \bigcup_{N<\o}c_n\cup\{\d\}\in\mtcl P_{\vec C}$. 
But this will be enough, as then $c^\ast$ will be an extension of $c$ in $\mtcl P_{\vec C}$ forcing $\dot A\sub U\cap N$. 

It thus remains to show how to find $c_{n+1}$ given $c_n$. Working in $N$, we may first fix some countable 
$M\preccurlyeq H_{\chi}$ (for some large enough $\chi$) containing $c_n$ and all other relevant objects 
(including some relevant dense set $D\sub\mtcl P_{\vec C}$ that we need to meet), and such that 
$[\eta,\,\d_M]\cap C_\d=\emptyset$, where $\d_M=M\cap\o_1$. In order to find $M$, we first consider a strictly 
$\sub$--increasing and continuous sequence $(M_\n)_{\n<\o_1}\in N$  of elementary submodels containing 
all relevant objects. Since $(M_\nu\cap\o_1)_{\n<\d}$ is a club of $\d$ of order type $\d$ and 
$\ot(C_\d)=\o^{\tau'}<\d$, we can then find some $\n<\d$ such that $M=M_\n$ is as desired. 
Now, working in $M$, we may, first,  extend $c_n$  to a condition $c_n'$ such that $\max(c_n')>\eta$ and 
$[\max(c_n),\,\eta]\cap c_n'=\emptyset$, and then extend $c_n'$ to a condition $c_n''$ in $D$.  
Let now $\bar u$ be the unique node in $U$  below $u_n$ of height $\d_M$. 
Since $U$ is a Suslin tree, we have that $u_n$ is totally $(U, M)$--generic. Also, the set $E\in N$ of $u\in U$ for which there 
is some $v\in U$ below $u$ and some $\bar c\in\mtcl P_{\vec C}$ extending $c_n''$ and forcing that $v\in\dot A$ is 
dense in $U$. It follows that we may find  some $u\in E\cap M$ below $u_n$, as witnessed by some 
$\bar c\in \mtcl P_{\vec C}\cap N$ and some $v\in U\cap M$. But then we may let $c_{n+1}=\bar c$. 
\end{proof}

\begin{lemma} ($\ZFC$+$\LC$) Let $\tau<\o_1$ be a nonzero ordinal. Suppose $\Gamma$ is an 
$\o_1$--suitable class containing all $\s$--closed forcing notions and such that $\Gamma\sub \o^\tau\mbox{--}\SP$. 
If $\CFA(\Gamma)$ holds, then so does $\tau$--$\TCG$.
\end{lemma}

\begin{proof}
By a result in \cite{Zapletal}, there is a $\s$--closed forcing notion adding a $\tau$--$\TCG$--sequence. 
The rest of the argument is as in the proof of Lemma \ref{consespfa+++} (and subsequent lemmas), using 
the preservation of $\tau$--$\TCG$--sequences by any $\o^\tau$--semiproper forcing, which is a completely  standard fact.
\end{proof}

\section{Appendix}\label{appendix}

We collect here a few results translating the approach to forcing and iterations via posets to that
done via complete Boolean algebras; for details see the forthcoming~\cite{VIAAUDSTEBOOK}
or \cite{VIAAUDSTE13} (the latter is available on ArXiv).

Given a Boolean algebra $\bool{B}$ and a prefilter $G$  on $\bool{B}$
(i.e., a family such that $\bigwedge F>0_\bool{B}$ for all finite $F\subseteq G$), we 
denote by $\bool{B}/_G$ the quotient algebra obtained using the ideal 
$J=\bp{b:b\wedge c=0_{\bool{B}}\text{ for all $c\in G$}}$. $\bool{B}^+$ denotes the positive elements
of $\bool{B}$ and $\dot{G}_{\bool{B}}=\bp{\ap{\check{b},b}:b\in\bool{B}}$ is the canonical $\bool{B}$--name for the $V$--generic filter.

\begin{theorem}
Assume $(P,\leq)$ is a partial order. Let $\RO(P)$ denote the family of regular open sets
for the topology $\tau_P$ whose open sets are the downward closed subsets of $P$. The function $i_p:P\to \RO(P)$ given by
$i_P:p\mapsto \Reg{\downarrow p}$ (where $\downarrow p=\bp{q\in P:q\leq p}$ and $\Reg{A}$ is the interior of the closure of $A$ for the topology generated by the sets $\downarrow p$) is
an order and incompatibility preserving map of $(P,\leq)$ into $\RO(P)^+$ with image dense in $\RO(P)^+$,
and is such that
$p\Vdash_P\phi(\tau_1,\ldots, \tau_n)$ if and only if $i_P(p)\leq\Qp{\phi(\tau_1, \ldots, \tau_n)}_{\RO(P)}$.
\end{theorem}

\begin{theorem}\label{qIso}
	If $i: \bool{B} \rightarrow \bool{C}$ is an injective complete homomorphism of complete Boolean algebras, then 
	$$\bool{B} * (\bool{C}/_{i[\dot{G}_{\bool{B}}]}) \cong \bool{C}.$$
	
	Conversely, if $\dot{\bool{Q}}\in V^{\bool{B}}$ is a $\bool{B}$-name for a 
	complete Boolean algebra and $G$ is $V$-generic for $\bool{B}$, then
	\[
	(\bool{B}*\dot{\bool{Q}})/_{i_{\bool{B}\ast\dot{\bool{C}}}[G]}\cong\dot{\bool{Q}}_G.
	\]
\end{theorem}

\begin{theorem}  \label{eRetrProp}
	Assume $i: \bool{B} \to \bool{C}$ is a complete injective homomorphism of complete Boolean algebra.
	Then the following holds:
	\begin{enumerate}
	\item
	$i$ has an adjoint $\pi_i:\bool{C} \to \bool{B}$ defined by $\pi_i(c)=\bigwedge\bp{b\in\bool{B}:i(b)\geq c}$
	\item
	For any $b \in \bool{B}$ and $c,d \in \bool{C}$, we have that:
	\begin{enumerate}[(a)]
		\item $\pi_i$ is order preserving; %and is defined by $c\mapsto\bigwedge \bp{b \in \bool{B}: ~ i(b) \geq c}$;
		\item $(\pi_i\circ i)(b)= b$, hence $\pi_i$ is surjective;
		\item \label{eRPComp} 
		$(i\circ\pi_i)(c)\geq c$; in particular, 
		$\pi_i$ maps $\bool{C} ^+$ to $\bool{B}^+$;
		\item \label{eRPJoins} $\pi_i$ preserves joins, i.e., $\pi_i (\bigvee X)= \bigvee \pi_i[X]$ for all 
		$X \subseteq \bool{C}$ for which the supremum $\bigvee X$ exists in $\bool{C}$;
		\item $i(b)= \bigvee \{e : \pi_i (e) \leq b \}$;
		\item \label{eRPHomo} $\pi_i (c \wedge i(b)) = \pi_i(c)\wedge b = \bigvee \{ \pi_i(e): e \leq c, \pi_i(e) \leq b\}$;
		%hence $(i,\pi_i)$ forms an adjoint pair;
		\item \label{eRPMeets} $\pi_i$ preserves neither meets nor complements whenever 
		$i$ is not surjective, but $\pi_i(d \wedge c) \leq \pi_i(d) \wedge \pi_i(c)$ and 
		$\pi_i(\neg c) \geq \neg \pi_i(c)$.\qedhere
	\end{enumerate}
	\end{enumerate}
	%Hence $(i,\pi_i)$ is a residuated pair.
\end{theorem}

\begin{definition} 
 $\FFF=\{i_{\alpha \beta}:\bool{B}_\alpha\to \bool{B}_\beta\,\mid\,\alpha \leq \beta < \lambda\}$ is a \emph{complete iteration system} of complete Boolean algebras iff for all $\alpha \leq \beta \leq \gamma < \lambda$:
	\begin{enumerate}
		\item $\bool{B}_\alpha$ is a complete Boolean algebra and $i_{\alpha \alpha}$ is the identity on it;
		\item $i_{\alpha \beta}$ is a complete injective homomorphism and 
		$\pi_{\alpha\beta}:\bool{B}_\beta\to\bool{B}_\alpha$, given by 
		$c\mapsto\bigwedge\bp{b\in\bool{B}:i_{\alpha\beta}(b)\geq c}$, is its associated 
		adjoint\footnote{See Theorem~\ref{eRetrProp} for the relevant properties of the adjoint map.};
		\item $i_{\beta \gamma}\circ i_{\alpha \beta}=i_{\alpha \gamma}$.
 \end{enumerate}

Let $\FFF$ be a complete iteration system of length $\lambda$. Then:

\begin{itemize}
\item
 The \emph{inverse limit} of the iteration is
	\[
	\varprojlim(\FFF) = \bp{ f \in \prod_{\alpha < \lambda} \bool{B}_\alpha : ~ \forall \alpha \forall \beta > \alpha\, ~ (\pi_{\alpha \beta}\circ f)(\beta) = f(\alpha) }
	\]
	and its elements are called \emph{threads}. 
\item
	The \emph{direct limit} is
	\[
	\varinjlim(\FFF) = \bp{ f \in \varprojlim(\FFF) : ~ \exists \alpha \forall \beta > \alpha ~ f(\beta) = i_{\alpha\beta}(f(\alpha)) }
	\]
	and its elements are called \emph{constant threads}. 
	The support of a constant thread, $\supp(f)$, is 
	the least $\alpha$ such that $(i_{\alpha\beta}\circ f)(\alpha)=f(\beta)$ for all $\beta\geq\alpha$. 
\item
	The \emph{revised countable support limit} 
	is\footnote{This definition of revised contable support limit is due to Donder and Fuchs; it is not clear
	whether it is in any sense equivalent to Miyamoto or Shelah's definition of rcs-limit. It is nonetheless effective
	to prove the main results about semiproper iterations: specifically, the 
	forthcoming~\cite{VIAAUDSTEBOOK} or \cite{VIAAUDSTE13} (which is available on ArXiv) give complete
	proofs of the preservation of semiproperness through RCS--iterations. However, in this paper
	we also want to preserve properties stronger than semiproperness. The proofs of 
	such preservation results are not given in
	in~\cite{VIAAUDSTEBOOK} or in~\cite{VIAAUDSTE13}; hence we refer the reader to Miyamoto's or Shelah's
	iteration theorems for the preservation through RCS--limits (according to their definitions) of these stronger properties.}
	\[
	\rcslim(\FFF) = \bp{ f \in \varprojlim(\FFF) : ~ f \in \varinjlim(\FFF) \vee \exists \alpha ~ f(\alpha) \Vdash_{\bool{B}_\alpha} \cof(\check{\lambda}) = \check{\omega} }.
	\]
\end{itemize}	
\end{definition}

\begin{theorem}
Assume $\bp{P_\alpha,\dot{Q}_\beta:\alpha\leq\lambda,\beta<\lambda}$ is an iteration of posets.
Let $i_{\alpha\beta}:\RO(P_\alpha)\to\RO(P_\beta)$ be the complete homomorphism induced
by the natural inclusion of $P_\alpha$ into $P_\beta$. Then:
\begin{itemize}
\item 
$\FFF=\bp{i_{\alpha,\beta}:\alpha\leq\beta<\lambda}$ is an iteration system of  complete injective homomorphisms
of complete Boolean algebras.
\item
If $\lambda=\omega$,  $\invlim(\FFF)$ is isomorphic to the Boolean completion of the full limit of
$\bp{P_\alpha,\dot{Q}_\beta:\alpha\leq\lambda,\beta<\lambda}$.
\item For any regular $\lambda$,
$\dirlim(\FFF)$ is isomorphic to the Boolean completion of the direct limit of
$\bp{P_\alpha,\dot{Q}_\beta:\alpha\leq\lambda,\beta<\lambda}$.
\end{itemize}
\end{theorem}

These results suffice to prove all the needed equivalences between results on forcing and iterations proved in the language of 
partial orders and their corresponding formulation in terms of complete Boolean algebras.

	\bibliographystyle{amsplain}
	\bibliography{Biblio}
\end{document}